\documentclass[english]{article}
\usepackage{graphicx}
\usepackage[T1]{fontenc}
\usepackage{geometry}
\usepackage{xcolor}
\usepackage{verbatim}
\usepackage{float}
\usepackage{booktabs}
\usepackage[algo2e,algoruled]{algorithm2e}
\usepackage{amsmath}
\usepackage{amsthm}
\usepackage{amssymb}
\usepackage{amsfonts}
 \usepackage{jabbrv}
\PassOptionsToPackage{normalem}{ulem}
\usepackage{ulem}

\makeatletter

\providecommand{\tabularnewline}{\\}
\providecolor{lyxadded}{rgb}{0,0,1}
\providecolor{lyxdeleted}{rgb}{1,0,0}
\DeclareRobustCommand{\lyxsout}[1]{\ifx\\#1\else\sout{#1}\fi}

\theoremstyle{plain}
\newtheorem{theorem}{\protect\theoremname}
\theoremstyle{plain}
\newtheorem{proposition}[theorem]{\protect\propositionname}
\ifx\proof\undefined
\newenvironment{proof}[1][\protect\proofname]{\par
	\normalfont\topsep6\p@\@plus6\p@\relax
	\trivlist
	\itemindent\parindent
	\item[\hskip\labelsep\scshape #1]\ignorespaces
}{%
	\endtrivlist\@endpefalse
}
\providecommand{\proofname}{Proof}
\fi
\theoremstyle{plain}
\newtheorem{lemma}[theorem]{\protect\lemmaname}
\theoremstyle{remark}
\newtheorem{remark}[theorem]{\protect\remarkname}
\theoremstyle{plain}
\newtheorem{corollary}[theorem]{\protect\corollaryname}

\usepackage{babel}
\providecommand{\corollaryname}{Corollary}
\providecommand{\lemmaname}{Lemma}
\providecommand{\propositionname}{Proposition}
\providecommand{\remarkname}{Remark}
\providecommand{\theoremname}{Theorem}
\usepackage[short]{optidef}
\usepackage{hyperref}
\usepackage{appendix}

\author{ 
	Qi Deng \thanks{qideng@mail.shufe.edu.cn, Shanghai University of Finance and Economics} \and 
    Chenghao Lan \thanks{daniellan@163.sufe.edu.cn, Shanghai University of Finance and Economics}}
\date{\vspace{-5ex}}
\makeatother

\global\long\def\wtil#1{\widetilde{#1}}%

\global\long\def\mbb#1{\mathbb{#1}}%

\global\long\def\br#1{\left(#1\right)}%

\global\long\def\norm#1{\left\Vert #1\right\Vert }%

\global\long\def\argmin{\operatornamewithlimits{argmin}}%

\global\long\def\dis{\text{dist}\,}%

\global\long\def\dom{\text{dom}}%

\global\long\def\raw{\rightarrow}%

\global\long\def\vep{\varepsilon}%

\global\long\def\dom{\operatornamewithlimits{dom}}%

\global\long\def\tsum{{\textstyle {\sum}}}%

\global\long\def\Ncal{\mathcal{N}}%

\global\long\def\Rbb{\mathbb{R}}%

\global\long\def\Ncal{\mathcal{N}}%

\global\long\def\Ocal{\mathcal{O}}%

\global\long\def\Pcal{\mathcal{P}}%

\global\long\def\Ebb{\mathbb{E}}%

\global\long\def\Ubf{\mathbf{U}}%

\global\long\def\Ibf{\mathbf{I}}%

\global\long\def\ik{i_{k}}%

\global\long\def\i{i}%

\begin{document}

\global\long\def\pt{\pi_{t}}%

\global\long\def\tone{\wtil{[1]}}%

\global\long\def\ts{\wtil{[s]}}%

\title{Efficiency of Coordinate Descent Methods For Structured Nonconvex
Optimization}
\maketitle
\begin{abstract}
Novel coordinate descent (CD) methods are proposed for minimizing nonconvex functions consisting of three terms: (i) a continuously differentiable term, (ii) a simple convex term, and (iii)  a concave and continuous term. First, by extending randomized CD to nonsmooth nonconvex settings, we develop a coordinate subgradient method that randomly updates block-coordinate variables by using block composite subgradient mapping. This method converges asymptotically to critical points with proven sublinear convergence rate for certain optimality measures. Second, we develop a randomly permuted CD method with two alternating steps: linearizing the concave part and cycling through variables. We prove asymptotic convergence to critical points and sublinear complexity rate for objectives with both smooth and concave parts. Third, we extend accelerated coordinate descent (ACD) to nonsmooth and nonconvex optimization to develop a novel randomized proximal DC algorithm whereby we solve the subproblem inexactly by ACD. Convergence is guaranteed with at most a few number of ACD iterations for each DC subproblem, and  convergence complexity is established for identification of some approximate critical points. Fourth, we further develop the third method to minimize certain ill-conditioned nonconvex functions: weakly convex functions with high Lipschitz constant to negative curvature ratios. We show that, under specific criteria, the ACD-based randomized method has superior complexity compared to conventional gradient methods. Finally, an empirical study on sparsity-inducing  learning models demonstrates that CD methods are superior to gradient-based methods for certain large-scale problems.
\end{abstract}

\section{Introduction}

%Nonconvex optimization is central to modern statistics, machine learning
%as well as operations research. %
%\begin{comment}
%While nonconvex optimization is generally classed as being NP-hard,
%it is now commonly agreed upon that many important nonconvex models
%can be optimized using simple optimization methods, such as first
%order methods by carefully exploiting data structures; examples include
%sparse regression, low rank optimization, and deep learning \cite{poczos2018gradient,RN322,recht2010guaranteed,chi2018nonconvex,RN323,ge2016matrix}.
%\end{comment}
%{} As data acquisition becomes increasingly ubiquitous, it becomes prohibitive
%to apply gradient algorithms to nonconvex models due to the fact that
%the models can easily have millions or zillions variables and samples.
%It is more advantageous for contemporary machine learning models to
%use randomized algorithms, such as stochastic gradient descent (\cite{nemirovski2009})
%and coordinate descent (\cite{RN160,RN180}), when analyzing very
%large datasets due to the fact that only a small portion of the data
%are interrogated in each iteration. %
%\begin{comment}
%Indeed, the use of randomized algorithms in machine learning models
%better exploits data structure and is more adaptable when datasets
%are distributed or incomplete. Thus, it is important to develop efficient
%randomized algorithms for use in large-scale nonconvex optimization.
%\end{comment}

Coordinate descent (CD) methods update only a subset of coordinate
variables in each iteration, keeping other variables fixed. 
Due to their scalability to the so-called ``big data'' problems (see
\cite{RN216,RN180,hong2016a,richtarik2016parallel,RN18}), CD methods
have attracted significant attention from machine learning and data science. 
This paper will develop  efficient CD methods for large-scale structured
nonconvex problems in the following form:

\begin{mini}
{x\in\Rbb^d}{F(x)=f(x)+\phi(x) - h(x),}{}{}
{\label{main-problem}}
\end{mini} where $f(x)$ is continuously differentiable, $\phi(x)$ is convex
lower-semicontinuous with a simple structure, and $h(x)$ is convex
continuous. The nonconvex problem formed in \eqref{main-problem}
is sufficiently powerful to express a variety of machine learning
applications, including sparse regression, low rank optimization,
and clustering (see \cite{gotoh2018dc,RN343,thi2015dc}).

Our main contribution to the field is that we propose a number of
novel CD methods with guaranteed convergence for a broad class of
nonconvex problems described by \eqref{main-problem}. Our methods
include extending RCD and cyclic CD to nonsmooth and nonconvex settings,
and new randomized proximal DC and proximal point methods by using
ACD to solve the subproblems. For all the proposed algorithms, we
not only provide guarantees to asymptotic convergence, but also prove
rate of convergence for properly defined optimality measures. To the
best of our knowledge, this is the first study of coordinate descent
methods for such nonsmooth and nonconvex optimization with complexity
efficiency guarantee. Our results are summarized as follows.

Our first result is a new randomized coordinate subgradient descent
(RCSD) method for nonsmooth nonconvex and composite problems. While
our algorithm recovers existing nonconvex CD methods \cite{patrascu2015efficient}
as a special case, it allows the nonsmooth part to be inseparable
and concave, and coordinates to be sampled either uniformly or non-uniformly
at random. We show the asymptotic convergence to critical points,
and we establish the sublinear rate of convergence for a proposed
optimality measure which naturally extends the proximal gradient mapping
to nonsmooth and nonconvex settings.

Motivated by block coordinate gradient descent (BCGD) \cite{RN18},
we propose a new randomly permuted coordinate descent (RPCD) for nonsmooth
and nonconvex optimization. Our primary innovation is to alternate
RPCD between linearizing the concave part $h(x)$ and successively
updating all the block-coordinate variables based on some cycling
order. The cycling order can be either deterministic or randomly shuffled,
provided that each block of variables is updated once in each loop.
We provide asymptotic convergence of BCGD method to critical points.
For a certain case ($\phi(x)=0$), we also establish a sublinear rate
of convergence of the subgradient norm.

We next extend accelerated coordinate descent methods to the nonsmooth and nonconvex setting by
considering the difference-of-convex representation of Problem \eqref{main-problem}.
We propose an ACD-based proximal DC (ACPDC) algorithm by transforming
$F(x)$ into a sequence of strongly convex functions that are approximately
minimized by the ACD method. We show that ACPDC is sufficiently fast
that only a few rounds of ACD are needed in each iteration. Hence,
ACPDC offers significant improvements compared to the classic DC
algorithm, which requires exact optimal solutions to the subproblems.
ACPDC also offers advantages to the proximal DC algorithm, an extension
of the DC algorithm that performs one proximal gradient descent step
to minimize the majorized function. Taking advantages of the fast
convergence of ACD, ACPDC is more efficient than gradient-based DC
algorithm in exploiting the problem structure, offering a much better trade-off
between iteration complexity and running time.

Finally, we draw attention to minimization of weakly convex functions,
namely, the nonconvex functions with bounded negative curvature, and
propose a new ACD-based proximal point method (ACPP) for solving such
problems. By assuming that the objective function is weakly convex,
faster rates of convergence can be attained. Specifically, we show
that the complexity rate of ACPP can be significantly better than
that of classic CD approaches for ill-conditioned problems; ill-conditioned
problems are those that have a relatively high ratio between Lipschitz
constant to negative curvature.

\paragraph{Related work}
There is a significant body of work on CD methods for nonconvex optimization.
We refer to \cite{hong2016a} for some general strategies
to develop block update algorithms for nonconvex optimization. \cite{xu2017a}
proposed proximal cyclic CD methods for minimizing composite functions
with nonconvex  separable regularizers. However, their
work didn't include the nonconvex and nonsmooth function described in
\eqref{main-problem}. See such an example in Section \ref{sec:Applications}.
A recent work \cite{hien2019inertial}
extended the cyclic block mirror descent to the nonsmooth
setting, and guaranteed asymptotic convergence to block-wise optimality for minimizing Problem \eqref{main-problem}. In contrast, our work directly guarantees convergence
to the critical points by employing a different linearization technique to handle the nonsmooth part $h(x)$. 
The work \cite{beck2018optimization} proposed efficient CD-type algorithms for achieving  specific coordinate-wise optimality for  problems with group sparsity terms;
they proved that the proposed coordinate-wise optimality is more restrictive than stationarity for such nonconvex problems. 
Apart from these works, a number of  CD methods with complexity guarantees
have also been developed. For example, \cite{patrascu2015efficient}
proposed randomized coordinate descent for nonconvex composite problems,
possibly with a linear constraint. Their algorithm 1-RCD can be viewed
as a special case of our first algorithm RCSD when $h(x)$ is void
and uniform sampling is adopted. Additionally, \cite{RN66} proposed
randomized block mirror descent for stochastic optimization of convex
and nonconvex composite problems. However, none of these proposals
consider a concave nonsmooth part in the objective, nor do they improve
on algorithm efficiency for the ill-conditioned nonconvex problem;
we fully address both of these challenges for coordinate descent methods
in this paper.

Another research line relevant to our study is the so-called DC optimization
(see \cite{thi2018dc,yuille2002concave,gotoh2018dc,thi2015dc}). Specified
for minimizing the difference-of-convex (DC) functions, a DC algorithm
alternates between linearizing the concave part and optimizing some
convex surrogate of DC function by applying convex algorithms. Lately,
much progress has been made in developing more efficient DC algorithms
and in applying DC algorithms to machine learning and statistics (see
\cite{gotoh2018dc,RN343,wen2018proximal,xu2018stochastic,nouiehed2018on}).
We refer to the recent work \cite{thi2018dc} for a general review
of DC algorithms and the applications. Notably, as a special case
of DC functions, the weakly convex
function has been increasingly popular due to its importance in machine
learning and statistics (\cite{fan2001variable,zhang2012a,drusvyatskiy2018efficiency}).
While it is possible to optimize weakly convex functions by directly
generalizing convex methods to the nonconvex settings (see \cite{saeed-lan-nonconvex-2013,RN97,patrascu2015efficient,RN66}),
much stronger efficiency guarantees can be obtained by indirect approaches
such as proximal point methods and prox-linear methods (see \cite{lan2018accelerated,kong2018complexity,RN323,davis2019proximally,duchi2018stochastic,drusvyatskiy2018efficiency}).
However, for the weakly-convex and the more general DC problems, it
remains to develop efficient CD-type methods that are scalable to high-dimensional
and large-scale data.

\begin{comment}
A more general approach for nonconvex optimization is by implementing
the majorization-minimization framework that alternates between majorizing
the objective by using a surrogate function and solving the relatively
easier surrogate problem with some external solvers. See \cite{sun2017majorization}
for a general overview. The so-called ``difference-of-convex'' (DC)
functions are an important and broad class of functions that are tractable
by this framework. Specified for minimizing such functions, a DC algorithm
alternates between linearizing the concave part and optimizing the
convex surrogate of DC function. DC functions cover a variety of applications
in machine learning (see \cite{thi2018dc,yuille2002concave,gotoh2018dc,thi2015dc}).
Lately, much progress has been made in more efficient DC algorithms
and in extension of DC problems to cover to stochastic and composite
problems (see \cite{gotoh2018dc,RN343,wen2018proximal,xu2018stochastic}).
We refer to \cite{thi2018dc} for a general review of this topic.
Weakly convex functions are a special case of DC functions that have
bounded negative curvature. While it is possible to optimize weakly
convex functions by directly extending convex methods (see \cite{saeed-lan-nonconvex-2013,RN97,patrascu2015efficient,RN66,allen-zhu2016variance}),
much more promising efficiency guarantee can be obtained by indirect
approaches (see \cite{lan2018accelerated,kong2018complexity,RN323,davis2019proximally,duchi2018stochastic}).
\end{comment}

\paragraph{Outline of the paper}

Section \ref{sec:prelim} introduces notations and preliminaries.
Section \ref{sec:rcd} and Section \ref{sec:rand-perm-cd} present
RCSD and RPCD for nonsmooth nonconvex optimization, and then establish
their convergence results. Section \ref{sec:acd-dc} presents a new
DC algorithm based on ACD and demonstrates its asymptotic convergence
to critical points and its convergence complexity of a proposed optimality
measure. Section \ref{sec:acd-pp} considers the nonconvex problem
with bounded negative curvature and presents an even faster proximal
point algorithm based on using ACD. Section \ref{sec:Applications}
discusses the applications of proposed methods on sparsity-inducing
machine learning models, and present preliminary experiments to demonstrate
the advantages of our proposed CD methods. Finally, Section \ref{sec:Conclusion}
draws conclusions.

\section{Notations and Preliminaries\label{sec:prelim}}

We denote $[m]=\left\{ 1,2,...,m\right\} $. Let $d$ be a positive integer and $\Ibf_{d}$
be the $d\times d$ identity matrix. Assume that matrix $\Ubf_{i}\in\Rbb^{d\times d_{i}}$
($i\in[m]$) satisfies: $\left[\Ubf_{1},\Ubf_{2},...,\Ubf_{m}\right]=\Ibf_{d}$
where $\sum_{i=1}^{m}d_{i}=d$. Let $x_{i}=\Ubf_{i}^{T}x$ be the
restriction of $x$ to the $i$-th block, and we hereby express $x=\sum_{i=1}^{m}\Ubf_{i}x_{\i}$.
Let $\|\cdot\|_{\i}$ be the standard Euclidean norm on $\mbb R^{d_{i}}$,
and the norm $\norm{\cdot}$ on $\Rbb^{d}$ is denoted by $\|x\|=\sqrt{\sum_{i=1}^{m}\|x_{\i}\|_{\i}^{2}}$
. We say that $f$ is block-coordinate-wise (or block-wise) Lipschitz
smooth, if there exist constants $L_{1},L_{2},...,L_{m}>0$ such that
for any $x,x+\Ubf_{i}t\in\mbb R^{d_{i}}$ and $t\in\mbb R^{d_{i}}$,
$i\in[m]$, we have

\[
\|\nabla_{i}f(x)-\nabla_{i}f(x+\Ubf_{i}t)\|_{\i}\le L_{i}\|t\|_{\i}.
\]
For any $s\in[0,1]$, we define $T_{s}=\sum_{i=1}^{m}L_{i}^{s}$ and
$\norm x_{[s]}=\sqrt{\sum_{i=1}^{m}L_{i}^{s}\|x_{\i}\|_{\i}^{2}}.$
The dual norm is defined by $\|y\|_{[s],*}=\sqrt{\sum_{i=1}^{m}L_{i}^{-s}\|y_{\i}\|_{\i}^{2}}$.
Let $L_{\max}=\max_{1\le i\le m}L_{i}$ and $L_{\min}=\min_{1\le i\le m}L_{i}$.

We say that a function $f(x)$ is $\mu_{s}$-strongly convex with
respect to $\norm{\cdot}_{[s]}$ if 
\begin{equation}
f(y)\ge f(x)+\left\langle \nabla f(x),y-x\right\rangle +\frac{\mu_{s}}{2}\norm{x-y}_{[s]}^{2}.\label{eq:strong-cvx}
\end{equation}
It immediately follows that $\mu_{s}\in[0,1]$.

Given a proper lower semi-continuous (lsc) function $f:\Rbb^{d}\raw\Rbb$,
for any $x\in\dom(f)$, the limiting subdifferential of $f$ at $x$
is defined as
\begin{align*}
\partial f(x) & =\bigl\{ u:\exists x^{k}\raw x,u^{k}\raw u,\ \text{with }f(x^{k})\raw f(x)\,\\
 & \quad\text{and }\liminf_{y\neq x,y\raw x^{k}}\tfrac{f(y)-f(x^{k})-\left\langle u^{k},y-x^{k}\right\rangle }{\|y-x^{k}\|}\ge0\,\text{as }k\raw\infty\bigr\}.
\end{align*}
Using the limiting subdifferential, we can define the optimality measure
of the proposed algorithms. A point $x$ is known as a critical point
of Problem \eqref{main-problem} if $$\left[\nabla f(x)+\partial\phi(x)\right]\cap\partial h(x)\neq\emptyset,$$
and it is known as a stationary point of Problem \eqref{main-problem}
if $$\partial h(x)\subseteq\nabla f(x)+\partial\phi(x).$$ 
While it
can be readily seen that critical points are weaker than stationary
points, these two notions coincide when $h(x)$ is smooth: $\partial h(x)=\{\nabla h(x)\}$.

Throughout this paper, we make the following assumptions regarding
Problem \eqref{main-problem}.
\begin{enumerate}
\item $f(x)$ is continuously differentiable, and it is block-wise Lipschitz
smooth with parameters $L_{1},L_{2},...,L_{m}$.
\item $\phi(x)$ is convex block-wise separable. Specifically, $\phi(x)=\sum_{i=1}^{m}\phi_{i}(x_{(i)})$,
where $\phi_{i}:\mbb R^{d_{i}}\rightarrow\mbb R$ ($i\in[m]$) is
a proper convex lsc function. Furthermore, $\phi_{i}(x)$ has a simple
structure such that for $\gamma>0$, and $g,y\in\Rbb^{d_{i}}$, it
is relatively easy to solve the proximal problem: 
$
\min_{x\in\Rbb^{d_{i}}}\{\left\langle g,x\right\rangle +\phi_{i}(x)+\tfrac{\beta}{2}\|x-y\|_{i}^{2}\}.
$
\item $h(x)$ is a convex continuous function.
\item $F(x)$ is level-bounded in the sense that the lower level set $\{x:F(x)\;\le r\}$
is bounded for any $r\in\Rbb$.
\item There exists an optimal solution $x^{*}$ such that $F(x^{*})=\min_{x}F(x)>-\infty$.
\end{enumerate}

\section{Randomized Coordinate Subgradient Descent\label{sec:rcd}}

Our goal in this section is to develop the proposed randomized coordinate
subgradient descent (RCSD) for the nonconvex problem \eqref{main-problem}
in Algorithm \ref{alg:rcsd}. This method can be regarded as a block-wise
proximal-subgradient-type algorithm that iteratively updates some
random coordinates while keeping the other coordinates fixed. 
Let $\gamma=[\gamma_{1},\gamma_{2},...,\gamma_{m}]^{T}$
be a positive vector. For any $y\in\Rbb^{d_{i}}$, the composite block
proximal mapping is given by
\begin{equation}
\Pcal_{i}\left(\bar{x}_{i},y_{i},\gamma_{i}\right)=\argmin_{x\in\Rbb^{d_{i}}}\left\{ \left\langle y_{i},x\right\rangle +\phi_{i}(x)+\tfrac{\gamma_{i}}{2}\|\bar{x}_{\i}-x\|_{\i}^{2}\right\} .\label{eq:mapping}
\end{equation}
Furthermore, we denote $\Pcal\left(\bar{x},y,\gamma\right)=\sum_{i=1}^{m}\Ubf_{i}\Pcal_{i}\left(\bar{x}_{i},y_{i},\gamma_{i}\right)$.

Algorithm
\ref{alg:rcsd} is broad enough to cover a variety of first order
methods for nonsmooth nonconvex optimization. For example, when setting
block number $m=1$, we recover Algorithm 2 in \cite{RN343}. When
assuming that $h(x)$ is void, we recover Algorithm 1-RCD in \cite{patrascu2015efficient}.

\begin{algorithm2e}
\KwIn{$x^{0}$;}

\For{k=0,1,2,...K}{

Sample $i_{k}\in[m]$ randomly with $\text{prob}(i_{k}=i)=p_{i}$\;

Compute $\nabla_{i_{k}}f(x^{k})$ and $v_{\ik}^{k}$ where $v^{k}\in\partial h(x^{k})$\;

$x_{\ik}^{k+1}=\Pcal_{i_{k}}\br{x_{\ik}^{k},\nabla_{i_{k}}f(x^{k})-v_{\ik}^{k},\gamma_{i_{k}}}$\;

$x_{j}^{k+1}=x_{j}^{k}$ if $j\neq i_{k}$\;

}

\caption{RCSD \label{alg:rcsd}}
\end{algorithm2e}

In order to analyze the convergence property of RCSD, we must first
establish some optimality measure. Let $v\in\partial h(x)$, we define the composite block subgradient
as
\begin{equation}
g_{i}(x,\nabla_{\i}f(x)-v_{i},\gamma_{i})=\gamma_{i}\left(x_{\i}-\Pcal_{i}(x_{i},\nabla_{\i}f(x)-v_{i},\gamma_{i})\right),\label{eq:comp-subgrad}
\end{equation}
and we define the composite subgradient as
\begin{equation}
g(x,\nabla f(x)-v,\gamma)=\tsum_{i=1}^{m}\Ubf_{i}g_{i}(x,\nabla_{\i}f(x)-v_{i},\gamma_{i}).
\end{equation}
The notations $g(x,\nabla f(x)-v,\gamma)$ and $g(x)$ are used interchangeably
when there is no ambiguity. According to the above definition, $\|g(\cdot)\|_{*}$
measures the progress in solving the subproblem. It can be seen that
$\|g(x)\|_{[s],*}\neq0$ when $x$ is a non-critical point, and $\|g(x)\|_{[s], *}=0$ for some $v\in\partial h(x)$
when $x$ is a critical point. This result is proven in the following
proposition:
\begin{proposition}
$\bar{x}$ is a critical point of \eqref{main-problem} if and only
if there exists $v\in\partial h(\bar{x})$ such that $g(\bar{x},\nabla f(\bar{x})-v,\gamma)=0$.
%$\bar{x}=\Pcal(\bar{x},\nabla f(x)-v,\gamma)$.
\end{proposition}
\begin{proof}
If $g(\bar{x},\nabla f(\bar{x})-v,\gamma)=0$, then $\bar{x}=\Pcal(\bar{x},\nabla f(\bar{x})-v,\gamma)$. Due to the
optimality condition for (\ref{eq:mapping}), we have $0\in\nabla_{i}f(\bar{x})-v_{i}+\partial\phi_{i}(\bar{x}_{i})$,
$i\in[m]$; hence $\bar{x}$ is a critical point. On the other hand,
if $\bar{x}$ is a critical point, then there exists $v\in\partial h(\bar{x})$
such that 
\begin{equation}
0\in\nabla_{i}f(\bar{x})-v_{i}+\partial\phi_{i}(\bar{x}_{i}),\quad i\in[m].\label{eq:opt-subprob}
\end{equation}
 For brevity, let us denote $u=\nabla f(\bar{x})-v$. We have
\begin{equation}
\begin{aligned}\left\langle u_{i},\bar{x}_{\i}\right\rangle +\phi_{i}(\bar{x}_{\i}) & \ge\min_{x\in\Rbb^{d_{i}}}\left\{ \left\langle u_{i},x\right\rangle +\phi_{i}(x)+\frac{\gamma_{i}}{2}\|\bar{x}_{\i}-x\|_{\i}^{2}\right\} \\
 & \ge\min_{x\in\Rbb^{d_{i}}}\bigl\{\left\langle u_{i},x\right\rangle +\phi_{i}(x)\bigr\}+\min_{x\in\Rbb^{d_{i}}}\frac{\gamma_{i}}{2}\|\bar{x}_{\i}-x\|_{\i}^{2}\\
 & =\left\langle u_{i},\bar{x}_{\i}\right\rangle +\phi_{i}(\bar{x}_{\i}),
\end{aligned}
\label{prop:middle-bound}
\end{equation}
where the final equality follows from the optimality condition (\ref{eq:opt-subprob}).
Hence we have strict equality in (\ref{prop:middle-bound}), and it
follows that
\[
\bar{x}_{i}=\argmin_{x\in\Rbb^{d_{i}}}\left\{ \left\langle u_{i},x\right\rangle +\phi_{i}(x)+\frac{\gamma_{i}}{2}\|\bar{x}_{\i}-x\|_{\i}^{2}\right\} ,\quad i\in[m].
\]
\end{proof}
We next present an important result that is related to the proximal
subproblem, presenting the proof for completeness.
\begin{lemma}
\label{lem:optimality-prox-map} Let $\psi(\cdot)$ be a convex lsc
function and define $z=\argmin_{y}\left\langle g,y\right\rangle +\frac{\gamma}{2}\|x-y\|^{2}+\psi(y),$
then we have
\begin{equation}
\left\langle g,z-y\right\rangle +\psi(z)-\psi(y)\le\frac{\gamma}{2}\,\bigl[\|x-y\|^{2}-\|z-y\|^{2}-\|x-z\|^{2}\bigr].\label{eq:3-point}
\end{equation}
\end{lemma}
\begin{proof}
Using the optimality condition, there exists $\xi\in\partial\psi(z)$
such that $g+\xi+\gamma(z-x)=0$. Plugging this into the convexity
condition$\psi(y)\ge\psi(z)+\left\langle \xi,y-z\right\rangle $,
we have $\psi(y)+\left\langle g+\gamma(z-x),y-z\right\rangle \ge\psi(z).$
Furthermore, for any $x,y,z$ we have the equality 
\[
\left\langle z-x,y-z\right\rangle =\frac{1}{2}\|x-y\|^{2}-\frac{1}{2}\|z-y\|^{2}-\frac{1}{2}\|x-z\|^{2}.
\]
Combining the above two results immediately yields (\ref{eq:3-point}).
\end{proof}
We present the general convergence property of RCSD in the following
theorem.
\begin{theorem}
\label{thm:rcd}Let $x^{0},x^{1},x^{2},...x^{k}$ be a sequence generated
by Algorithm \ref{alg:rcsd}.
\begin{enumerate}
\item Assume that $\gamma_{i}>\frac{L_{i}}{2}$. The sequence $\{x^{k}\}$
is bounded, and almost surely (a.s.), every limit point of the sequence
is a critical point.
\item Assume that $\gamma_{i}=L_{i}$ ($i\in[m]$) and that blocks are sampled
with probability $p_{i}\propto L_{i}^{1-s}$ ($0\le s\le1$). Then
\[
\min_{0\le k\le K}\Ebb\norm{g(x^{k},\nabla f(x^{k})-v^{k},\gamma)}_{[s],*}^{2}\le\tfrac{2(\sum_{j=1}^{m}L_{j}^{1-s})[F(x^{0})-F(x^{*})]}{K+1},
\]
where the expectation is with respect to $i_{0},i_{1},...,i_{K}.$
\end{enumerate}
\end{theorem}
\begin{proof}
First, due to the optimality condition, there exists $\xi_{\ik}^{k+1}\in\partial\phi_{\ik}(x_{\ik}^{k+1})$
such that
\begin{equation}
\nabla_{i_{k}}f(x^{k})-v_{\ik}^{k}+\xi_{\ik}^{k+1}=-\gamma_{i_{k}}\bigl(x_{\ik}^{k+1}-x_{\ik}^{k}\bigr).\label{eq:optimality-bcd-update}
\end{equation}
Using the convexity of $\phi(\cdot)$ and $h(\cdot)$, we obtain
\begin{align}
\phi(x^{k+1})-h(x^{k+1}) & \le\phi(x^{k})+\left\langle \xi^{k+1},x^{k+1}-x^{k}\right\rangle -h(x^{k})-\left\langle v^{k},x^{k+1}-x^{k}\right\rangle \nonumber \\
 & =\phi(x^{k})-h(x^{k})+\left\langle \xi_{\ik}^{k+1}-v_{\ik}^{k},x_{\ik}^{k+1}-x_{\ik}^{k}\right\rangle .\label{eq:convexity-bound}
\end{align}
In view of the relation (\ref{eq:optimality-bcd-update}), (\ref{eq:convexity-bound}),
and the block-wise smoothness of $f(x)$, we have
\begin{align}
F(x^{k+1}) & =f(x^{k+1})+\phi(x^{k+1})-h(x^{k+1})\nonumber \\
 & \le f(x^{k})+\left\langle \nabla_{\ik}f(x^{k}),x_{\ik}^{k+1}-x_{\ik}^{k}\right\rangle +\tfrac{L_{i_{k}}}{2}\|x_{\ik}^{k+1}-x_{\ik}^{k}\|_{\ik}^{2}\nonumber \\
 & \quad +\phi(x^{k+1})-h(x^{k+1})\nonumber \\
 & \le f(x^{k})+\left\langle \nabla_{\ik}f(x^{k})-v_{\ik}^{k}+\xi_{\ik}^{k+1},x_{\ik}^{k+1}-x_{\ik}^{k}\right\rangle +\tfrac{L_{i_{k}}}{2}\|x_{\ik}^{k+1}-x_{\ik}^{k}\|_{\ik}^{2}\nonumber \\
 & \quad +\phi(x^{k})-h(x^{k}) \nonumber \\
 & \le F(x^{k})+\left(\tfrac{L_{i_{k}}}{2}-\gamma_{i_{k}}\right)\|x_{\ik}^{k+1}-x_{\ik}^{k}\|_{\ik}^{2}. \label{eq:rbcd-recur}
\end{align}
Together with the identity $\|x_{\ik}^{k+1}-x_{\ik}^{k}\|_{\ik}^{2}=\|x^{k+1}-x^{k}\|^{2}$,
we have $F(x^{k+1})+\left(\gamma_{i_{k}}-\tfrac{L_{i_{k}}}{2}\right)\|x^{k+1}-x^{k}\|^{2}\le F(x^{k}).$
Let $\bar{\gamma}=\min_{1\le i\le m}\left(\gamma_{i}-\frac{L_{i}}{2}\right)$,
then we immediately see that $\{F(x^{k+1})+\bar{\gamma}\|x^{k+1}-x^{k}\|^{2}\}$
is a non-increasing sequence. Hence $\{x^{k}\}$ is a bounded sequence
due to the level-boundedness of $F(\cdot)$. Moreover, summing up
the relation (\ref{eq:rbcd-recur}) over $k=0,1,2,...,K$, we obtain
\begin{equation}
\bar{\gamma}\tsum_{k=0}^{K}\|x^{k+1}-x^{k}\|^{2}\le F(x^{0})-F(x^{K+1}).\label{eq:sq-sum-rbcd}
\end{equation}
Consequently, we conclude that $\lim_{k\rightarrow\infty}\|x^{k}-x^{k+1}\|=0$.

For brevity, let $g_{\i}^{k}=g_{\i}(x^{k},\nabla_{\ik}f(x^{k})-v_{\i}^{k},\gamma_{\i})$
and $g^{k}=\sum_{i=1}^{m}\Ubf g_{\i}^{k}$. We have 
\begin{equation}\label{eq:rbcd-mid-001}
\Ebb_{i_{k}}\left[\left(\gamma_{i_{k}}-\tfrac{L_{i_{k}}}{2}\right)\|x_{\ik}^{k+1}-x_{\ik}^{k}\|_{\ik}^{2}\right] =\tsum_{i=1}^{m}p_{i}\tfrac{\gamma_{i}-L_{i}/2}{\gamma_{i}^{2}}\|g_{i}^{k}\|^{2}\ge\beta\,\|g^{k}\|^{2},
\end{equation}
where $\beta=\min_{1\le i\le m}\{p_{i}\frac{\gamma_{i}-L_{i}/2}{\gamma_{i}^{2}}\}$.
In view of \eqref{eq:rbcd-recur} and \eqref{eq:rbcd-mid-001}, we have $$E_{i_{k}}[F(x^{k+1})-F^{*}]+\beta\,\|g^{k}\|^{2}\le F(x^{k})-F^{*}.$$
According to the supermartingale convergence theorem, the sequence
$\{F(x^{k+1})-F^{*}\}$ converges a.s. and $\sum_{k=0}^{\infty}\|g^{k}\|^{2}<\infty$
a.s. It follows that $\lim_{k\rightarrow\infty}g^{k}=0$ a.s. Let
us define
\begin{equation}
\wtil x^{k+1}=\argmin_{x}\{\left\langle \nabla f(x^{k})-v^{k},x\right\rangle +\phi(x)+\tsum_{i=1}^{m}\tfrac{\gamma_{i}}{2}\|x_{\i}^{k}-x_{\i}\|_{\i}^{2}\}.\label{eq:xtilde}
\end{equation}
Since $g^{k}=\sum_{i}^{m}\gamma_{i}\Ubf_{i}(x_{\i}^{k}-\wtil x_{\i}^{k+1})$,
we have $\lim_{k\raw+\infty}\wtil x^{k+1}-x^{k}=0$ a.s. Based on
the optimality of $\wtil x^{k+1}$, we have
\begin{equation}
\phi(\wtil x^{k+1})\le\phi(x)+\left\langle \nabla f(x^{k})-v^{k},x-\wtil x^{k+1}\right\rangle +\tsum_{i=1}^{m}\tfrac{\gamma_{i}}{2}\|x_{\i}^{k}-x_{\i}\|_{\i}^{2}-\tsum_{i=1}^{m}\tfrac{\gamma_{i}}{2}\|x_{\i}^{k}-\wtil x_{\i}^{k+1}\|_{\i}^{2}.\label{eq:phi-tilde-bound}
\end{equation}
Let $\bar{x}$ be a limit point of the sequence $\left\{ x^{k}\right\} $.
Passing to a subsequence if necessary, we have $\lim_{k\raw\infty}x^{k}=\bar{x}$.
By the continuity of $\nabla f(x)$, we have $\lim_{k\rightarrow\infty}\nabla f(x^{k})=\nabla f(\bar{x})$.
Taking $k\raw\infty$ in (\ref{eq:phi-tilde-bound}) and $x=\bar{x}$,
we have 
\[
\limsup_{k}\phi(\wtil x^{k+1})\le\phi(\bar{x}),\quad a.s.
\]
because $\lim_{k}\wtil x^{k+1}=\lim_{k}x^{k}=\bar{x}$ a.s. Since $\phi(x)$ is a lsc function, we have $\lim_{k\raw\infty}\phi(\wtil x^{k+1})=\phi(\bar{x})$.

Moreover, by the optimality condition of (\ref{eq:xtilde}), we have
\[
0=\nabla f(x^{k})-v^{k}+\tsum_{i=1}^{m}\gamma_{i}\Ubf_{i}(\wtil x_{\i}^{k+1}-x_{\i}^{k})+u^{k+1},
\]
for some $u^{k+1}\in\partial\phi(\wtil x^{k+1})$. Due to the boundedness
of $\{x^{k}\}$ and continuity of $h(x)$, $\{v^{k}\}$ is also a
bounded sequence. Passing to a subsequence if necessary, we have $\lim_{k\rightarrow\infty}v^{k}=\bar{v}$
a.s. for some $\bar{v}$, and $\lim_{k\rightarrow\infty}u^{k+1}=\bar{v} -\nabla f(\bar{x})=\bar{u}$
a.s. By graph continuity of limiting subdifferentials, we have $\bar{u}\in\partial\phi(\bar{x})$
and $\bar{v}\in\partial h(\bar{x})$ . Therefore, we conclude that
$\bar{x}$ is a.s. a critical point.

For the second part, summing up (\ref{eq:rbcd-recur}) over $k=0,1,2,...,K$
and rearranging the terms accordingly, we have 
\begin{equation}
\tsum_{k=0}^{K}\Bigl(\gamma_{i_{k}}-\tfrac{L_{i_{k}}}{2}\Bigr)\|x_{\ik}^{k+1}-x_{\ik}^{k}\|_{\ik}^{2}\le F(x^{0})-F(x^{K+1}).\label{eq:rcsd-sum-sq-dist}
\end{equation}
Since we assume that $\gamma_{i}=L_{i}$ and $p_{i}=L_{i}^{1-s}/\bigl(\sum_{j=1}^{m}L_{j}^{1-s}\bigr)$,
the following identity holds
\begin{equation}
\Ebb_{i_{k}}\Bigl[\Bigl(\gamma_{i_{k}}-\tfrac{L_{\ik}}{2}\Bigr)\|x_{\ik}^{k+1}-x_{\ik}^{k}\|_{\ik}^{2}\Bigr]=\tsum_{i=1}^{m}\tfrac{p_{i}}{2L_{i}}\|g_{\i}\|_{\i}^{2}=\tfrac{1}{2T_{1-s}}\|g^{k}\|_{[s],*}^{2}.\label{eq:expect-square-rbcd}
\end{equation}
 Taking the expectation of (\ref{eq:rcsd-sum-sq-dist}) with respect
to $\{i_{k}\}$ and using (\ref{eq:expect-square-rbcd}),we have 
\[
\tsum_{k=0}^{K}\Ebb\norm{g^{k}}_{[s],*}^{2}\le2T_{1-s}[F(x^{0})-F(x^{K+1})].
\]
\end{proof}
\begin{remark}
Notice that the sublinear rate in Theorem \ref{thm:rcd} is typical
for first order methods on nonsmooth and nonconvex problems. For instance,
if $h(x)$ is void and uniform sampling ($s=1$) is performed, we
recover the rate obtained for 1-RCD in \cite{patrascu2015efficient}.
Another difference between our work and \cite{patrascu2015efficient}
is that our analysis also adapts to the strategy of nonuniform sampling
$(s<1)$, whereby the composite subgradient is measured by the norm
$\|\cdot\|_{[s],*}$. Suppose that our goal is to have some $\varepsilon$-accurate
solution ($\min_{k}\Ebb\bigl[\|g^{k}\|^{2}\bigr]\le\varepsilon$),
the total number of iterations required by non-uniform sampling RCSD
(with $s=0$) is $O\bigl(\tfrac{\sum_{i=1}^{m}L_{i}}{\vep}\bigl[F(x^{0})-F(x^{*})\bigr]\bigr)$.
In contrast, since $\norm{g^{k}}_{[1],*}^{2}\ge\tfrac{1}{L_{\max}}\norm{g^{k}}^{2}$,
the bound provided by uniform sampling RCSD ($s=1$) is $O\bigl(\tfrac{L_{\max}m}{\vep}\bigl[F(x^{0})-F(x^{*})\bigr]\bigr)$.
\end{remark}

\section{Randomly Permuted Coordinate Descent\label{sec:rand-perm-cd}}

In this section, our goal is to develop a randomly permuted coordinate
descent (RPCD) method in Algorithm \ref{alg:rpcd}. When analyzing
the convergence of cyclic or permuted CD, we normally require the
triangle inequality to bound the gradient norm by the sum of point
distances. However, this may be difficult to achieve in the nonsmooth
setting since subgradient is not necessarily Lipschitz continuous.
To avoid this problem, Algorithm \ref{alg:rpcd} computes the subgradient
of $h(x)$ only once after updating all the blocks but it always
uses the new block gradient of $f(x)$ while updating the block
variables. The visiting order of block-coordinate variables can be
either deterministic or randomly shuffled, provided that all the blocks
are updated in each outer loop of the algorithm. The following theorem
summarizes the main convergence property of RPCD.

\begin{algorithm2e}
\KwIn{$x^{0}$;}

\For{k=0,1,2,...K-1}{

Compute $v^{k}\in\partial h(x^{k})$ and set $\wtil x^{0}=x^{k}$\;

Generate permutation: $\pi_{0},\pi_{1},\pi_{2},\ldots\pi_{m-1}$\;

\For{$t=0,2,...,m-1$}{

Set $\wtil x_{\pt}^{t+1}=\Pcal_{\pi_{t}}\bigl(\wtil x_{\pt}^{t},\nabla_{\pi_{t}}f\bigl(\wtil x^{t}\bigr)-v_{\pt}^{k},\gamma_{\pi_{t}}\bigr)$\;

Set $\wtil x_{j}^{t+1}=\wtil x_{j}^{t}$ if $j\neq\pi_{t}$\;

}

Set $x^{k+1}=\wtil x^{m}$\;

}

\caption{RPCD \label{alg:rpcd}}
\end{algorithm2e}

\begin{theorem}
Let $x^{1},x^{2},...,x^{K}$ be the generated sequence in Algorithm
\ref{alg:rpcd}.
\begin{enumerate}
\item If $\gamma_{i}\ge L_{i}$ ($1\le i\le m$), then the sequence $\{x^{k}\}$
is bounded, $\lim_{k\rightarrow\infty}\|x^{k}-x^{k+1}\|<\infty,$
and every limit point is a critical point.
\item If we assume $\phi(x)=0$ and
set $\gamma_{i}=L_{i}$, then
\[
\min_{0\le k\le K}\|\nabla f(x^{k})-v^{k}\|^{2}\le4\left(L_{\max}+\tfrac{mL^{2}}{L_{\min}}\right)\tfrac{F(x^{0})-F(x^{*})}{(K+1)}.
\]
\end{enumerate}
\end{theorem}
\begin{proof}
First, using the optimality condition, there exists $\xi^{k+1}\in\partial\phi(\wtil x^{t+1})$
such that
\begin{equation}
\nabla_{\pi_{t}}f(x^{k})-v_{\pt}^{k}+\xi_{\pt}^{k+1}=\gamma_{\pi_{t}}\bigl(\wtil x_{\pt}^{t}-\wtil x_{\pt}^{t+1}\bigr).\label{eq:optimality-bcd-update-1}
\end{equation}
For the $k$-th subproblem, let $\wtil F(x)=f(x)+\phi(x)-\left\langle v^{k},x-x^{k}\right\rangle -h(x^{k})$
be the surrogate function. Due to the convexity of $\wtil F(x)$,
we obtain
\[
\wtil F(x)=f(x)+\phi(x)-\left\langle v^{k},x-x^{k}\right\rangle -h(x^{k})\ge f(x)+\phi(x)-h(x)=F(x),\quad\forall x.
\]
This bound is tight at $x^{k}$: $\wtil F(x^{k})=F(x^{k})$. Next
we develop some relation about the surrogate functions. We have 
\begin{align*}
\wtil F(\wtil x^{t+1}) & \le f(\wtil x^{t})+\langle \nabla_{\pi_{t}}f\br{\wtil x^{t}},\wtil x_{\pt}^{t+1}-\wtil x_{\pt}^{t}\rangle +\tfrac{L_{\pi_{t}}}{2}\|\wtil x_{\pt}^{t+1}-\wtil x_{\pt}^{t}\|_{\pt}^{2} \\
 & \quad -h(x^{k})-\langle v^{k},\wtil x^{t+1}-x^{k}\rangle +\phi\br{\wtil x^{t+1}}\\
 & \le f(\wtil x^{t})-\gamma_{\pi_{t}}\langle \wtil x_{\pt}^{t+1}-\wtil x_{\pt}^{t},\wtil x_{\pt}^{t+1}-\wtil x_{\pt}^{t}\rangle +\tfrac{L_{\pi_{t}}}{2}\|\wtil x_{\pt}^{t+1}-\wtil x_{\pt}^{t}\|_{\pt}^{2}\\
 & \quad\langle v_{\pt}^{k}-\xi_{\pt}^{k+1},\wtil x_{\pt}^{t+1}-\wtil x_{\pt}^{t}\rangle -h(x^{k})-\langle v^{k},\wtil x^{t+1}-x^{k}\rangle +\phi\br{\wtil x^{t+1}}\\
 & \le f(\wtil x^{t})+\br{\tfrac{L_{\pi_{t}}}{2}-\gamma_{\pi_{t}}}\|\wtil x_{\pt}^{t+1}-\wtil x_{\pt}^{t}\|_{\pt}^{2}-h(x^{k})-\langle v^{k},\wtil x^{t}-x^{k}\rangle \\
 & \quad +\phi\br{\wtil x^{t+1}}+\langle\xi_{\pt}^{k+1},\wtil x_{\pt}^{t}-\wtil x_{\pt}^{t+1}\rangle\\
 & \le\wtil F(\wtil x^{t})+\br{\tfrac{L_{\pi_{t}}}{2}-\gamma_{\pi_{t}}}\|\wtil x_{\pt}^{t+1}-\wtil x_{\pt}^{t}\|_{\pt}^{2}\\
 & \le\wtil F(\wtil x^{t})-\tfrac{L_{\pi_{t}}}{2}\|x_{\pt}^{k+1}-x_{\pt}^{k}\|_{\pt}^{2},
\end{align*}
where the last inequality uses the fact that $\|\wtil x_{\pt}^{t+1}-\wtil x_{\pt}^{t}\|_{\pt}^{2}=\|x_{\pt}^{k+1}-x_{\pt}^{k}\|_{\pt}^{2}$
and $\gamma_{i}-\frac{L_{i}}{2}\ge\frac{L_{i}}{2}$. Summing up the
above relation over $t=0,1,...,m-1$, we have
\begin{align}
\tfrac{1}{2}\norm{x^{k+1}-x^{k}}_{[1]}^{2} & \le\tsum_{t=0}^{m-1}\left[\wtil F(\wtil x^{t})-\wtil F(\wtil x^{t+1})\right] =\wtil F(x^{0})-\wtil F(x^{k+1}) \nonumber \\
   & \le F(x^{k})-F(x^{k+1}).\label{rpcd:square-dist-bound}
\end{align}
Therefore, $\{F(x^{k})\}$ is a non-increasing sequence and $\lim_{k\raw\infty}F(x^{k})$
exists. From the bounded level set assumption, we immediately have
that the sequence $\{x^{k}\}$ is bounded. Summing up the above inequality,
we obtain 
\begin{equation}
\tfrac{1}{2}\tsum_{k=0}^{K}\|x^{k+1}-x^{k}\|_{[1]}^{2}\le F(x^{0})-F(x^{K+1})\le F(x^{0})-F(x^{*})<+\infty.\label{rpcd:sum-square-dist-bound}
\end{equation}
Therefore, we have $\lim_{k\rightarrow\infty}\|x^{k}-x^{k+1}\|_{[1]}=0.$

Let $\bar{x}$ be a limiting point of $\left\{ x^{k}\right\} $. Passing
to a subsequence if necessary, we have $\lim_{k\rightarrow\infty}x^{k}=\bar{x}$.
Let us denote $y^{t+1}=\argmin_{x}\Bigl\{\bigl\langle\nabla f(\wtil x^{t})-v^{k},x\bigr\rangle+\phi(x)+\sum_{i=1}^{m}\frac{\gamma_{i}}{2}\|x_{\i}-\wtil x_{\i}^{t}\|_{\i}^{2}\Bigr\}.$
By this definition,  for $t=0,1,...,m-1$ and any $x$, we have
\begin{equation}
\phi(y^{t+1})\le\phi(x)+\bigl\langle\nabla f(\wtil x^{t})-v^{k},x-y^{t+1}\bigr\rangle+\tsum_{i=1}^{m}\tfrac{\gamma_{i}}{2}\|x_{\i}-\wtil x_{\i}^{t}\|_{\i}^{2}-\tsum_{i=1}^{m}\tfrac{\gamma_{i}}{2}\|y_{\i}^{t+1}-\wtil x_{\i}^{t}\|_{\i}^{2}.\label{rpcd:ytilde-bound}
\end{equation}
Notice that $y^{m}=x^{k+1}$. Taking $x=\bar{x}$, $t=m-1$ and letting
$k\raw\infty$ in (\ref{rpcd:ytilde-bound}), we have $\limsup_{k\raw\infty}\phi(x^{k+1})\le\phi(\bar{x})$.
According to the lower semicontinuity of $\phi(\cdot)$, we have $\lim_{k\raw\infty}\phi(x^{k+1})=\phi(\bar{x})$.
In addition, due to the optimality condition, we have 
\begin{equation}
\nabla f(\wtil x^{t})-v^{k}+u^{t+1}+\tsum_{i=1}^{m}\gamma_{i}\Ubf_{i}(y_{\i}^{t+1}-\wtil x_{\i}^{t})=0,\quad t=0,1,...,m-1,\label{rpcd:ytilde-opt}
\end{equation}
where $u^{t+1}\in\partial\phi(y^{t+1})$. Taking $t=m-1$ and letting
$k\rightarrow\infty$ in (\ref{rpcd:ytilde-opt}), since $\lim_{k\rightarrow\infty}y^{m}-\wtil x^{m-1}=0$,
we have $\lim_{k\raw\infty}\nabla f(\wtil x^{m-1})-v^{k}+u^{m}=0$,
where $u^{m}\in\partial\phi(y^{m})=\partial\phi(x^{k+1})$. By continuity
of $\nabla f(\cdot)$ we have $\lim_{k\rightarrow\infty}\nabla f(\wtil x^{m-1})=\nabla f(\bar{x})$.
Since $x^{k}$ is bounded, $v^{k}\in\partial h(x^{k})$ is also bounded.
Passing to a subsequence if necessary, we have $\lim_{k\raw\infty}v^{k}=\bar{v}$
and therefore $\lim_{k\raw\infty}u^{m}=\bar{v}-\nabla f(\bar{x})$.
Since $x^{k}\overset{f}{\rightarrow}\bar{x}$, $x^{k+1}\overset{\phi}{\rightarrow}\bar{x}$,
we have $\bar{v}\in\partial h(\bar{x})$ and $\bar{u}\in\partial\phi(\bar{x})$
based on graph continuity of limiting subdifferentials. Therefore, we have $\bar{v}\in\nabla f(\bar{x})+\partial\phi(\bar{x})$
and we conclude that $\bar{x}$ is a critical point.

For the second part, assume that $\phi(x)=0$ and $\gamma_{i}=L_{i}$.
From the smoothness of $f(x)$, we have
\[
\|\nabla_{\pi_{t}}f(\wtil x^{t})-\nabla_{\pi_{t}}f(x^{k})\|_{\pt}^{2}\le\|\nabla f(\wtil x^{t})-\nabla f(x^{k})\|^{2}\le L^{2}\|\wtil x^{t}-x^{k}\|^{2},\quad t=0,1,2,...,m.
\]
For any vector $a,b$ we have the inequality
\[
\norm a^{2}\le\br{\norm b+\norm{a-b}}^{2}\le2\norm b^{2}+2\norm{a-b}^{2}.
\]
Hence we can bound the squared subgradient norm by:
\begin{align*}
\|\nabla_{\pi_{t}}f(x^{k})-v_{\pt}^{k}\|_{\pt}^{2} & \le2\|\nabla_{\pi_{t}}f(\wtil x^{t})-v_{\pt}^{k}\|_{\pi_{t}}^{2}+2\|\nabla_{\pi_{t}}f(\wtil x^{t})-\nabla_{\pi_{t}}f(x^{k})\|_{\pt}^{2}\\
 & \le2\|\nabla_{\pi_{t}}f(\wtil x^{t})-v_{\pt}^{k}\|_{\pt}^{2}+2L^{2}\|\wtil x^{t}-x^{k}\|^{2}.
\end{align*}
Summing up the above relation over $\pi_{0},\pi_{1},\pi_{2},\ldots\pi_{m-1}$,
we have 
\begin{align}
\|\nabla f(x^{k})-v^{k}\|^{2} & \le\tsum_{t=0}^{m-1}\left[2\|\nabla_{\pi_{t}}f(\wtil x^{t})-v_{\pt}^{k}\|_{\pt}^{2}+2L^{2}\|\wtil x^{t}-x^{k}\|^{2}\right]\nonumber \\
 & =\tsum_{t=0}^{m-1}\left[2L_{\pi_{t}}^{2}\|\wtil x_{\pi_{t}}^{t+1}-\wtil x_{\pi_{t}}^{t}\|_{\pt}^{2}+2L^{2}\|\wtil x^{t}-x^{k}\|^{2}\right]\nonumber \\
 & \le\tsum_{t=0}^{m-1}\left[2L_{\pi_{t}}^{2}\|\wtil x_{\pt}^{t+1}-\wtil x_{\pt}^{t}\|_{\pt}^{2}+2\tfrac{L^{2}}{L_{\min}}\tsum_{s=0}^{t-1}L_{\pi_{s}}\|\wtil x_{\pi_{s}}^{s+1}-\wtil x_{\pi_{s}}^{s}\|_{s}^{2}\right]\nonumber \\
 & \le2\Bigl(L_{\max}+\tfrac{mL^{2}}{L_{\min}}\Bigr)\tsum_{t=0}^{m-1}L_{\pi_{t}}\|\wtil x_{\pt}^{t+1}-\wtil x_{\pt}^{t}\|_{\pt}^{2} \nonumber \\
 & \le2\Bigl(L_{\max}+\tfrac{mL^{2}}{L_{\min}}\Bigr)\norm{x^{k+1}-x^{k}}_{[1]}^{2}.\label{cyclic-sq-grad}
\end{align}
Here, the first equality is due to the equality $\nabla_{\pi_{t}}f(\wtil x^{t})-v_{\pi_{t}}^{t}=\gamma_{\pi_{t}}\left(\wtil x_{\pt}^{t}-\wtil x_{\pt}^{t+1}\right)$, and the second inequality is due to the
block Lipschitz smoothness. Putting (\ref{rpcd:sum-square-dist-bound})
and (\ref{cyclic-sq-grad}) together, we obtain
\begin{equation*}
\tsum_{k=0}^{K}\norm{\nabla f(x^{k})-v^{k}}^{2}  \le4\Bigl(L_{\max}+\tfrac{mL^{2}}{L_{\min}}\Bigr)\left[F(x^{0})-F(x^{*})\right].
\end{equation*}
\end{proof}

\begin{remark}
Although the complexity rate of RPCD has a larger multiplicative constant
than that of RCSD, RPCD and RCSD are based on substantially different
optimality measures: the rate of RPCD is deterministic and the rate
of RCSD is on expectation. Nevertheless, in our experiments (presented
in Section \ref{sec:Experiments}) , we observe that RPCD and RCSD
have very similar convergence performance. We note that similar observations
(\cite{RN18}) have been made when comparing RCD and RBGD in convex
setting.
\end{remark}

\section{Randomized Proximal DC method For Nonconvex Optimization\label{sec:acd-dc}}

In this section we will develop a new randomized proximal DC algorithm
based on ACD. Our method is based on the simple observation that $F(x)$
can be reformulated as a difference-of-convex function. Specifically,
suppose that $f(x)$ has a bounded negative curvature ($\mu$-weakly
convex):
\[
f(x)\ge f(y)+\left\langle \nabla f(y),x-y\right\rangle -\frac{\mu}{2}\|x-y\|^{2}.
\]
By this definition, $f(x)+\frac{\mu}{2}\|x\|^{2}$ is convex and hence
$f(x)$ can be expressed as a difference-of-convex function: $f(x)=\Bigl(f(x)+\frac{\mu}{2}\|x\|^{2}\Bigr)-\frac{\mu}{2}\|x\|^{2}$.
Therefore, we express $F(x)$ in the following DC form: 
\[
F(x)=f(x)+\frac{\mu}{2}\|x\|^{2}+\phi(x)-\Bigl(\frac{\mu}{2}\|x\|^{2}+h(x)\Bigr).
\]
For simplicity, we assume that $f(x)$ is convex for the remainder
of this section.

\begin{algorithm2e}
\KwIn{$x^{0}$\;}

\For{k=0,1,2,...K}{

Compute $v^{k}\in\partial h(x^{k})$ \;

Set
\begin{equation}
x^{k+1}=\argmin_{x}\wtil F(x)=f(x)+\phi(x)-h(x^{k})-\left\langle v^{k},x-x^{k}\right\rangle .\label{eq:dc-update}
\end{equation}

}

\caption{The DC algorithm\label{alg:DC}}
\end{algorithm2e}

There is a wealth of literature on DC optimization. For simplicity,
we summarize the most general DC algorithm (DCA) in Algorithm \ref{alg:DC}.
This approach is an iterative procedure that alternates between linearizing
the concave part ($-h(\cdot)\le-h(x^{k})-\left\langle v^{k},\cdot-x^{k}\right\rangle $)
and minimizing the majorized function ($\wtil F(\cdot)$) to specific
accuracy. To handle the subproblem (\ref{eq:dc-update}) more efficiently,
DCA often employs some external solvers, such as gradient
methods and interior point methods to obtain high-precision solutions.
However, the main drawback of this approach is that exactly minimizing
$\wtil F(x)$ can be potentially slow for many large-scale problems.
To avoid this difficulty, recent works (such as \cite{wen2018proximal,gotoh2018dc,RN340})
propose using a proximal DC algorithm (pDCA) by performing one step
of proximal gradient descent to solve (\ref{eq:dc-update}). pDCA
can be interpreted as an application of Algorithm \ref{alg:DC}  based on a different DC representation:
$
F(x)=\left[\frac{L}{2}\|x\|^{2}+\phi(x)\right]-\left[h(x)+\frac{L}{2}\|x\|^{2}-f(x)\right].
$
 It follows from the DC algorithm that we obtain $x^{k+1}$ by
\begin{align*}
x^{k+1} & =\argmin_{x}\Bigl\{\frac{L}{2}\|x\|^{2}+\phi(x)-\bigl\langle v^{k}+Lx^{k}-\nabla f(x^{k}),x\bigr\rangle\Bigr\}\\
 & =\argmin_{x}\Bigl\{\bigl\langle\nabla f(x^{k})-v^{k},x\bigr\rangle+\phi(x)+\frac{L}{2}\|x-x^{k}\|^{2}\Bigr\}.
\end{align*}
Computing the above proximal mapping can be much easier than solving
(\ref{eq:dc-update}), provided that the function $\phi(x)$ has specific
simple form. For example, when $\phi(x)$ is the lasso or elastic-net
penalty, $x^{k+1}$ 	is computed by the so-called soft-thresholding.
While pDCA offers significant improvements on the per-iteration computational
time for convergence to approximate critical point solutions, it appears
that pDCA does not efficiently exploit the convex structure of $f(x)+\phi(x)$.
For example, consider the extreme case that $h(x)=0$, then pDCA is
exactly the proximal gradient descent for minimizing convex composite
function. However, in such case, it is well known that proximal gradient
descent has suboptimal worst-case complexity and the optimal complexity
is achieved by Nesterov's accelerated methods.

\begin{algorithm2e}
\KwIn{$x_{0}$, $\mu$, $t$;}

Compute $\wtil{\mu}$\;

\For{k=0,1,2,...K}{

Compute $v^{k}\in\partial h(x^{k})$ \;

Set 
\begin{equation}
F_{k}(x)=f(x)+\phi(x)-h(x^{k})-\left\langle v^{k},x-x^{k}\right\rangle +\frac{\mu}{2}\|x-x^{k}\|_{[1]}^{2}\label{eq:Fk-pdc}
\end{equation}

Obtain $x^{k+1}$ from running Algorithm \ref{alg:APCG} with input
$F_{k}(\cdot)$, $x^{k}$ , $\wtil{\mu}$ and $t$\;

}

\caption{ACPDC\label{alg:acpdc}}
\end{algorithm2e}

To overcome these drawbacks, in Algorithm \ref{alg:acpdc}, we propose
a new randomized proximal DC method (ACPDC) by using ACD efficiently
for the DC subproblem. This algorithm is based on the idea that, at
the $k$-th iteration, we first form a majorized function $F_{k}(x)$
by linearizing $h(x)$ and adding a strongly convex function 
\[
\frac{\mu}{2}\|x-x^{k}\|_{[1]}^{2}=\frac{\mu}{2}\tsum_{i=1}^m L_{i}\|x_{\i}-x_{\i}^{k}\|_{i}^{2},
\]
and we then apply ACD to approximately minimize $F_{k}(x)$ to specific
accuracy. As we will show later, there is no need to obtain high precision
for the subproblems, because running a few iterations (order $O(m)$)
of ACD is sufficient to guarantee fast convergence. Using ACD allows
us to exploit the convex structure of $f(x)$ more efficiently than
by using proximal gradient descent, while at the same time retaining
the same $O\bigl(\frac{1}{\epsilon}\bigr)$ iteration complexity.
We confirm the empirical advantage of using ACDPC in selected real
applications in Section \ref{sec:Applications}.

\subsection*{The convex subproblem}

Before developing the main convergence result of ACPDC, let us discuss
the convex subproblem in Algorithm \ref{alg:acpdc} and examine its
complexity. This subproblem to minimize $F_{k}(x)$ can be described
in the following general form:
\begin{equation}
\min_{x}\wtil F(x)=\wtil f(x)+\phi(x)+\frac{\mu}{2}\|x-\bar{x}\|_{[1]}^{2},\label{cvx-subprob}
\end{equation}
where $\tilde{f}(x)$ is convex and continuously differentiable. In
the context of (\ref{eq:Fk-pdc}), $\wtil F(\cdot)$ has the following
properties:
\begin{enumerate}
\item $\wtil f(x)+\frac{\mu}{2}\|x-\bar{x}\|_{[1]}^{2}$ is block-wise Lipschitz
smooth with constant $\wtil L_{i}=(1+\mu)L_{i}$, $i=1,2,...,m$.
\item $\wtil F(\cdot)$ is $\wtil{\mu}$-strongly convex w.r.t. $\|\cdot\|_{\tone}$,
where $\wtil{\mu}=\frac{\mu}{1+\mu}$ and the norm is defined by $\norm x_{\tone}=\sqrt{\tsum_{i=1}^{m}\wtil L_{i}\norm{x_{\i}}_{\i}^{2}}$.
\end{enumerate}
In order to approach this problem efficiently, we adopt the accelerated
randomized proximal coordinate gradient method (APCG, \cite{lin2015an})
and express it in Algorithm \ref{alg:APCG}; APCG is a specific ACD
method for strongly convex and composite optimization. We refer to
\cite{lin2015an} for the complete convergence analysis of APCG, and
we rephrase one important result in the following theorem.
\begin{algorithm2e}
\KwIn{$\wtil F(x)$, $x^{0}$, $\wtil{\mu}$, $K$;}

Initialize $z^{0}=x^{0}$ and choose $\gamma_{0}\in[\wtil{\mu},1]$\;

\For{k=0,1,2,...,K-1}{

Compute $\alpha_{k}\in(0,\frac{1}{m}]$ from $m^{2}\alpha_{k}^{2}=(1-\alpha_{k})\gamma_{k}+\alpha_{k}\mu$ and
set $\gamma_{k+1}=(1-\alpha_{k})\gamma_{k}+\alpha_{k}\wtil{\mu}$,
$\beta_{k}=\frac{\alpha_{k}\mu}{\gamma_{k+1}}$\;

Set $y^{k}=\tfrac{1}{\alpha_{k}\gamma_{k}+\gamma_{k+1}}(\alpha_{k}\gamma_{k}x^{k}+\gamma_{k+1}z^{k})$\;

Sample $i_{k}$ uniformly at random from $\{1,2,...,m\}$ and update
\[
z^{k+1}=\argmin_{x}\{\tfrac{m\alpha_{k}}{2}\|x-(1-\beta_{k})z^{k}-\beta_{k}y^{k}\|_{[1]}^{2}+\langle\nabla_{i_{k}}f(y^{k}),x_{\ik}\rangle+\phi_{i_{k}}(x_{\ik})\};
\]

Set $x^{k+1}=y^{k}+m\alpha_{k}(z^{k+1}-z^{k})+\frac{\wtil{\mu}}{m}(z^{k}-y^{k})$\;

}

\caption{APCG\label{alg:APCG}}
\end{algorithm2e}

\begin{theorem}
\label{thm:apcg}Let $x^{*}$ be the optimal solution of Problem (\ref{cvx-subprob})
and assume that $\wtil F(x)$ is $\wtil{\mu}$-strongly convex with
$\norm{\cdot}_{\wtil{[1]}}$. In Algorithm \ref{alg:APCG}, we have
\[
\Ebb[\wtil F(x^{K})-\wtil F(x^{*})]\le\left(1-\tfrac{\sqrt{\wtil{\mu}}}{m}\right)^{K}\left(\wtil F(x^{0})-\wtil F(x^{*})+\tfrac{\wtil{\mu}}{2}\|x^{0}-x^{*}\|_{\wtil{[1]}}^{2}\right).
\]
\end{theorem}

\subsection*{Analysis of ACPDC}

To develop the complexity result of the overall procedure, we need
to define some terminating criterion. %
\begin{comment}
Inspired by the recent work \cite{sun2018enhanced}, we develop the
following criterion for an approximate critical point.
\begin{defn}
\label{def:approx-crit-point}We say that $x$ is a $\epsilon$-approximate
$\mu$-critical point w.r.t. a general norm $\|\cdot\|$ if for $v\in\partial h(x)$,
\[
F(x)\le\min_{y}f(y)+\phi(y)-h(x)-\left\langle v,y-x\right\rangle +\frac{\mu}{2}\|y-x\|^{2}+\epsilon.
\]
Similarly we say that $\bar{x}$ is a stochastic $\epsilon$-approximate
$\mu$-critical point with respect to $\|\cdot\|$ if
\[
\Ebb[F(\bar{x})]\le\Ebb\left[\min_{x}f(x)+\phi(x)-h(\bar{x})-\left\langle \bar{v},x-\bar{x}\right\rangle +\frac{\mu}{2}\|x-\bar{x}\|^{2}\right]+\epsilon
\]
\end{defn}
%
It is a routine matter to verify that $\bar{x}$ is exactly a critical
point when $\epsilon=0$, namely, when
\[
x=\argmin_{y}f(y)+\phi(y)-h(x)-\left\langle v,y-x\right\rangle +\frac{\mu}{2}\|y-x\|^{2}
\]
\end{comment}
Let us denote 
\begin{equation}
F_{\mu}(y,x,v)=f(y)+\phi(y)-h(x)-\left\langle v,y-x\right\rangle +\frac{\mu}{2}\|y-x\|_{[1]}^{2},\quad v\in\partial h(x),\label{eq:F-mu-surrogate}
\end{equation}
and define $\bar{x}=\argmin_{y}F_{\mu}(y,x,v)$. We define the prox-mapping
$$p(x,v,\mu)=\mu\tsum_{i=1}^{m}\left[L_{i}\Ubf_{i}(x_{i}-\bar{x}_{i})\right].$$
Based on the definition,  $x$ is exactly a critical
point when $\|p(x,v,\mu)\|_{[1]}=0$, and the norm of $\|p(x,v,\mu)\|_{[1]}$
can be viewed as a measure of the optimality of $x$. Moreover, notice
that the norm $\|p(x,v,\mu)\|_{[1]}$ is also related to the accuracy
of $x$ for minimizing $F_{\mu}(\cdot,x,v)$. Assume that $F(x)=F_{\mu}(x,x,v)\le F_{\mu}(\bar{x},x,v)+\epsilon$,
then we have
\[
F_{\mu}(x,x,v)\ge\frac{\mu}{2}\|x-\bar{x}\|_{[1]}^{2}+F_{\mu}(\bar{x},x,v)\ge\frac{\mu}{2}\|x-\bar{x}\|_{[1]}^{2}+F_{\mu}(x,x,v)-\epsilon,
\]
where the first inequality is due to the strong convexity of $F_{\mu}(\cdot,x,v)$
and the optimality of $\bar{x}$. Therefore, we conclude
\begin{equation}
\frac{1}{2\mu}\|p(x,v,\mu)\|_{[1],*}^{2}=\frac{\mu}{2}\|x-\bar{x}\|_{[1]}^{2}\le\epsilon.\label{eq:prox-bound-eps}
\end{equation}

We next justify the usage of $\|p(x,v,\mu)\|_{[1],*}$ by showing
that this criterion is quantitatively equivalent to the earlier proposed
criterion of using composite subgradient at $x$. Here, the equivalence
means that the two values are different up to some constant factors. Before
making more rigorous argument, we first simplify notations. Throughout
this section we denote the composite mapping as $\wtil x=\argmin_{y}F^{\gamma}(y,x,v)$
($\gamma>0$) where we define
\begin{equation}
F^{\gamma}(y,x,v)=\bigl\langle\nabla f(x)-v,y-x\bigr\rangle+\phi(y)+\frac{\gamma}{2}\|y-x\|_{[1]}^{2},\quad v\in\partial h(x).\label{eq:F-gm-surrogate}
\end{equation}
In this way, the earlier proposed stepsize parameter $\gamma_{i}$
in Algorithm \ref{alg:rcsd} takes the form $\gamma_{i}=\gamma L_{i}$.
We still denote the composite subgradient by $g(x,\nabla f(x)-v,\gamma)$,
slightly abusing the notation when there is no ambiguity in the context.
Then we quantify the relations between these two criteria in the following
theorem.

\begin{comment}
Moreover, the relation between approximate and exact critical points
can be determined intuitively from the following example intuitively.
Consider the unconstrained problem with $\phi(x)=0$, and let $\bar{v}\in\partial h(\bar{x})$,
then $\wtil F(x)=f(x)-h(\bar{x})-\left\langle \bar{v},x-\bar{x}\right\rangle +\frac{\mu}{2}\|x-\bar{x}\|^{2}$
and let $\bar{x}^{*}=\argmin_{x}\wtil F(x)$. Due to the strong convexity
of $\wtil F(x)$, we have
\[
\|\bar{x}-\bar{x}^{*}\|^{2}\le\frac{2}{\mu}\left[\wtil F(\bar{x})-\wtil F(\bar{x}^{*})\right]\le\frac{2\epsilon}{\mu}.
\]
Using the optimality of $\bar{x}^{*}$, we have $\nabla f(\bar{x}^{*})-\bar{v}+\mu(\bar{x}^{*}-\bar{x})=0$.
We successively deduce
\begin{align*}
\|\nabla f(\bar{x})-\bar{v}\|^{2} & \le\norm{\nabla f(\bar{x})-\nabla f(\bar{x}^{*})+\nabla f(\bar{x}^{*})-\bar{v}}^{2}\\
 & \le2\norm{\nabla f(\bar{x})-\nabla f(\bar{x}^{*})}^{2}+2\|\nabla f(\bar{x}^{*})-\bar{v}\|^{2}\\
 & \le2L^{2}\|\bar{x}-\bar{x}^{*}\|^{2}+2\mu^{2}\|\bar{x}-\bar{x}^{*}\|^{2}\\
 & \le4\left(\frac{L^{2}}{\mu}+\mu\right)\epsilon.
\end{align*}
\end{comment}

\begin{theorem}
Let $\gamma\in[\mu,3\mu)$, then we have 
\begin{align}
\|g(x,\nabla f(x)-v,\gamma)\|_{[1],*} & \le\left(1+\tfrac{L}{\mu L_{\min}}\right)\left(1+\sqrt{\tfrac{L}{2L_{\min}\mu+L}}\right)\|p(x,v,\mu)\|_{[1],*},\label{prop:bound-p-bigger-g}\\
\|g(x,\nabla f(x)-v,\gamma)\|_{[1],*} & \ge\frac{\gamma}{\mu}\Bigl(1+\sqrt{\tfrac{\gamma-\mu+L/L_{\min}}{3\mu-\gamma}}\Bigr)^{-1}\|p(x,v,\mu)\|_{[1],*}.\label{prop:bound-g-bigger-p}
\end{align}
\end{theorem}
\begin{proof}
Using the optimality of $\wtil x$ and $\bar{x}$ for optimizing (\ref{eq:F-gm-surrogate})
and (\ref{eq:F-mu-surrogate}), respectively, and using Lemma \ref{lem:optimality-prox-map},
we have
\begin{equation}
\bigl\langle\nabla f(x)-v,\wtil x-y\bigr\rangle+\phi(\wtil x)-\phi(y)\le\tfrac{\gamma}{2}\|x-y\|_{[1]}^{2}-\tfrac{\gamma}{2}\|x-\wtil x\|_{[1]}^{2}-\tfrac{\gamma}{2}\|y-\wtil x\|_{[1]}^{2},\label{prop:opt1}
\end{equation}
and
\begin{equation}
f(\bar{x})+\phi(\bar{x})-\bigl\langle v,\bar{x}-y\bigr\rangle-f(y)-\phi(y)\le\tfrac{\mu}{2}\|x-y\|_{[1]}^{2}-\tfrac{\mu}{2}\|x-\bar{x}\|_{[1]}^{2}-\tfrac{\mu}{2}\|y-\bar{x}\|_{[1]}^{2},\label{prop:opt2}
\end{equation}
Plugging in $y=\bar{x}$ into the (\ref{prop:opt1}) and plugging
$y=\wtil x$ into (\ref{prop:opt2}) respectively and then adding
the two inequalities together, we have
\[
f(\bar{x})-f(\wtil x)+\bigl\langle\nabla f(x),\wtil x-\bar{x}\bigr\rangle\le\tfrac{\gamma-\mu}{2}\|x-\bar{x}\|_{[1]}^{2}-\tfrac{\gamma-\mu}{2}\|x-\wtil x\|_{[1]}^{2}-\tfrac{\mu+\gamma}{2}\|\wtil x-\bar{x}\|_{[1]}^{2},
\]
Moreover, due to the convexity and the Lipschitz smoothness of $f(x)$,
we deduce
\begin{equation*}
\begin{aligned}
& f(\bar{x})-f(\wtil x)+\bigl\langle\nabla f(x),\wtil x-\bar{x}\bigr\rangle \\
& = f(\bar{x})-f(x)-\bigl\langle\nabla f(x),\bar{x}-x\bigr\rangle-[f(\wtil x)-f(x)-\bigl\langle\nabla f(x),\wtil x-x\bigr\rangle]\\
& \ge -\tfrac{L}{2}\|x-\wtil x\|^{2}.
\end{aligned}
\end{equation*}
Combining the above two inequalities, we have 
\begin{equation}
\tfrac{\mu+\gamma}{2}\|\wtil x-\bar{x}\|_{[1]}^{2}\le\tfrac{\gamma-\mu}{2}\|x-\bar{x}\|_{[1]}^{2}+\tfrac{L}{2}\|x-\wtil x\|^{2}-\tfrac{\gamma-\mu}{2}\|x-\wtil x\|_{[1]}^{2}.\label{eq:bound-dist-prox}
\end{equation}
Moreover, due to the inequality $\|a+b\|^{2}\le2(\|a\|^{2}+\|b\|^{2})$,
it follows that 
\begin{equation*}
\begin{aligned}
\tfrac{\mu+\gamma}{2}\|\wtil x-\bar{x}\|_{[1]}^{2} & \le(\gamma-\mu)\left(\|x-\wtil x\|_{[1]}^{2}+\|\wtil x-\bar{x}\|_{[1]}^{2}\right)+\tfrac{L}{2}\|x-\wtil x\|^{2}-\tfrac{\gamma-\mu}{2}\|x-\wtil x\|_{[1]}^{2}\\
 & \le(\gamma-\mu)\|\wtil x-\bar{x}\|_{[1]}^{2}+\tfrac{\gamma-\mu+L/L_{\min}}{2}\|x-\wtil x\|_{[1]}^{2}.
\end{aligned}
\end{equation*}
Therefore, we have 
\[
\tfrac{3\mu-\gamma}{2}\|\wtil x-\bar{x}\|_{[1]}^{2}\le\tfrac{\gamma-\mu+L/L_{\min}}{2}\|x-\wtil x\|_{[1]}^{2}.
\]
Using the triangle inequality, we have 
\[
\|\bar{x}-x\|_{[1]}\le\|\wtil x-x\|_{[1]}+\|\wtil x-\bar{x}\|_{[1]}\le\Bigl(1+\sqrt{\tfrac{\gamma-\mu+L/L_{\min}}{3\mu-\gamma}}\Bigr)\|\wtil x-x\|_{[1]}.
\]
Notice that we have $\|g(x,\nabla f(x)-v,\gamma)\|_{[1],*}=\gamma\|\wtil x-x\|_{[1]}$
and $\|p(x,v,\mu)\|_{[1],*}=\mu\|\bar{x}-x\|_{[1]}$ by definition.
Then the result (\ref{prop:bound-g-bigger-p}) immediately follows.

Furthermore, let $\wtil z=\argmin F^{\wtil{\gamma}}(y,x, v)$ for $\wtil{\gamma}=\mu+\tfrac{L}{L_{\min}}$.
Placing $\gamma=\wtil{\gamma}$ and $\wtil x=\wtil z$ in relation
(\ref{eq:bound-dist-prox}), we have 
\begin{align*}
\tfrac{\mu+\wtil{\gamma}}{2}\|\wtil z-\bar{x}\|_{[1]}^{2} & \le\tfrac{\wtil{\gamma}-\mu}{2}\|x-\bar{x}\|_{[1]}^{2}+\tfrac{L}{2}\|\wtil z-x\|^{2}-\tfrac{\wtil{\gamma}-\mu}{2}\|x-\wtil z\|_{[1]}^{2}\\
 & \le\tfrac{L}{2L_{\min}}\|x-\bar{x}\|_{[1]}^{2}+\tfrac{L}{2L_{\min}}\|\wtil z-x\|_{[1]}^{2}-\tfrac{\wtil{\gamma}-\mu}{2}\|x-\wtil z\|_{[1]}^{2}\\
 & \le\tfrac{L}{2L_{\min}}\|x-\bar{x}\|_{[1]}^{2}.
\end{align*}
We similarly obtain
\begin{equation}
\|\wtil z-x\|_{[1]}\le\|\wtil z-\bar{x}\|_{[1]}+\|\bar{x}-x\|_{[1]}\le\left(1+\sqrt{\tfrac{L}{2L_{\min}\mu+L}}\right)\,\|\bar{x}-x\|_{[1]}.\label{eq:zbound-xbar}
\end{equation}

Next we derive a bound of $\|\wtil x-x\|_{[1]}$ in terms of $\|\wtil z-x\|_{[1]}$.
Due to the optimality of $\wtil z$ ($0\in\partial F^{\wtil{\gamma}}(\wtil z,x, v)$),
we have 
\[
0\in\nabla f(x)-v+\partial\phi(\wtil z)+\wtil{\gamma}\tsum_{i\in[m]}L_{i}\Ubf_{i}(\wtil z_{i}-x_{i}),
\]
let us denote a subgradient $\nabla F^{\gamma}(\wtil z,x)\in\partial F^{\gamma}(\wtil z,x)$
such that $0=\nabla F^{\gamma}(\wtil z,x)+(\wtil{\gamma}-\gamma)\tsum_{i\in[m]}\Ubf_{i}L_{i}(\wtil z_{i}-x_{i})$,
hence we have $\|\nabla F^{\gamma}(\wtil z,x)\|_{[1],*}=(\wtil{\gamma}-\gamma)\|\wtil z-x\|_{[1]}$.
Then using strong convexity of $F^{\gamma}(y,x)$ and optimality of
$\wtil x$ ($0\in\partial F^{\gamma}(\wtil x,x)$) , we have
\[
(\wtil{\gamma}-\gamma)\|\wtil z-x\|_{[1]}=\|\nabla F^{\gamma}(\wtil z,x)\|_{[1],*}\ge\gamma\|\wtil z-\wtil x\|_{[1]}\ge\gamma\bigl[\|\wtil x-x\|_{[1]}-\|\wtil z-x\|_{[1]}\bigr].
\]
We conclude that $\wtil{\gamma}\|\wtil z-x\|_{[1]}\ge\gamma\|\wtil x-x\|_{[1]}$.
In view of relation (\ref{eq:zbound-xbar}), we have

\[
\|\wtil x-x\|_{[1]}\le\left(\tfrac{\mu}{\gamma}+\tfrac{L}{L_{\min}\gamma}\right)\left(1+\sqrt{\tfrac{L}{2L_{\min}\mu+L}}\right)\,\|\bar{x}-x\|_{[1]}.
\]
Then relation (\ref{prop:bound-p-bigger-g}) follows.
\end{proof}
We are now ready to present the main convergence result of Algorithm
\ref{alg:acpdc} in the following theorem.
\begin{theorem}
\label{thm:acdc}Assume that $f(x)$ is convex, and that there exists
$M>0$ such that $M=\sup_{v\in\partial h(x)}\norm v<+\infty.$
In Algorithm \ref{alg:acpdc} if we set $t=t_0=\Bigl\lceil\ln4\tfrac{m}{\sqrt{\mu/(1+\mu)}}\Bigr\rceil$,
then
\begin{enumerate}
\item Every limit point of the sequence is a critical point, a.s.;
\item We have 
\[
\min_{1\le k\le K}\Ebb\left[\|p(x^{k},v^{k},\mu)\|_{[1],*}^{2}\right]\le\tfrac{2\mu\bigl[F(x^{0})-F(x^{*})+4M\|x^{0}-x^{*}\|+\mu\|x^{0}-x^{*}\|_{[1]}^{2}\bigr]}{K}.
\]
\end{enumerate}
\end{theorem}
\begin{proof}
Based on the earlier discussion, $F_{k}(x)$ is block-wise Lipschitz smooth with
constant $(1+\mu)L_{i}$, and  $\wtil{\mu}$-strongly convex ($\wtil{\mu}=\frac{\mu}{1+\mu}$)
with $\norm{\cdot}_{\tone}$, where $
\norm x_{\tone}=\sqrt{\sum_{i=1}^{m}(1+\mu)L_{i}\norm{x_{\i}}_{\i}^{2}}$.
For brevity, we denote
$\lambda=(1-{\sqrt{\wtil{\mu}}}/m){}^{t}$. 
Let $x^{k+1^{*}}$ be the optimal solution for minimizing $F_k(\cdot)$.
After running Algorithm \ref{alg:APCG} we obtain the following convergence
relation
\begin{align}
\Ebb[F_{k}(x^{k+1})-F_{k}(x^{k+1^{*}})] & \le\Bigl(1-\tfrac{\sqrt{\wtil{\mu}}}{m}\Bigr)^{t}\Bigl[F_{k}(x^{k})-F_{k}(x^{k+1^{*}})+\tfrac{\wtil{\mu}}{2}\|x^{k}-x^{k+1^{*}}\|_{\tone}^{2}\Bigr]\nonumber \\
 & \le2\lambda\left[F_{k}(x^{k})-F_{k}(x^{k+1^{*}})\right],\label{eq:subprob-conv}
\end{align}
where the expectation is over $x^{k+1}$. It then follows that
\begin{align*}
	\Ebb[F_{k}(x^{k+1})] & \le F_k(x^k)-(1-2\lambda)[F_k(x^k)-F_k(x^{k+1^{*}})]\\
		& \le F(x^k)-\tfrac{\mu(1-2\lambda)}{2}\|x^k-x^{k+1^{*}}\|^2_{[1]} \\
		& \le f(x^k)+\phi(x^k) -h(x^{k-1})-\langle v^{k-1},x^k-x^{k-1} \rangle -\tfrac{\mu(1-2\lambda)}{2}\|x^k-x^{k+1^{*}}\|^2_{[1]} \\
		& = F_{k-1}(x^k) -\tfrac{\mu}{2}\|x^k-x^{k-1}\|^2_{[1]} -\tfrac{\mu(1-2\lambda)}{2}\|x^k-x^{k+1^{*}}\|^2_{[1]}.
\end{align*}
Here the second inequality is due to strong convexity of $F_k(\cdot)$, and the third inequality is due to the convexity of $h(\cdot)$.
Therefore, based on the supermartingale convergence theorem, we have that $\lim_{k\raw \infty} F_k(x^{k+1})=\pi$ for some random variable $\pi$, $\tsum_{k=0}^{\infty}\|x^{k}-x^{k+1^{*}}\|^{2}<\infty$
a.s., $\tsum_{k=0}^{\infty}\|x^{k}-x^{k-1}\|^{2}<\infty$ and, hence, $\lim_{k\rightarrow\infty}\|x^{k}-x^{k+1^{*}}\|=0$
a.s. and $\lim_{k\rightarrow\infty}\|x^{k}-x^{k-1}\|=0$. Moreover, we see that $\{F(x^k)\}$ is a bounded sequence a.s.

Let $\bar{x}$ be any limit point of  $\{x^{k}\}$, passing
to a subsequence if necessary, we therefore have $\lim_{k\raw\infty}x^{k+1^{*}}=\bar{x}$
a.s. Since $x^{k+1^{*}}$ obtains the optimum of the subproblem, we
can adopt an argument analogous to the previous analysis to show that
$\limsup_{k\raw\infty}\phi(x^{k+1^{*}})=\phi(\bar{x})$ a.s. Hence, by
lower semi-continuity of $\phi(\cdot)$, we have $\lim_{k\raw\infty}\phi(x^{k+1^{*}})=\phi(\bar{x})$ a.s.
Moreover, based on the the optimality condition for minimizing $F_{k}(\cdot)$,
we obtain
\[
0=\nabla f(x^{k+1^{*}})+u^{k+1}-v^{k}+\wtil{\mu}\tsum_{\i=1}^{m}\wtil L_{\i}\Ubf_{i}\bigl(x_{\i}^{k+1^{*}}-x_{\i}^{k}\bigr),
\]
where $u^{k+1}\in\partial\phi(x^{k+1^{*}}).$ 
Due to the continuity of $h(x)$ and almost sure boundedness of $x^k$, we have $\lim_{k\raw\infty}v^{k}=\bar{v}\in\partial h(\bar{x})$ for some $\bar{v}$.
Taking $k\raw\infty$, we have 
\[
\lim_{k\raw\infty}u^{k+1}=\bar{u}=\bar{v}-\nabla f(\bar{x}),\quad a.s.
\]
Due to the graph continuity
of limiting subdifferential  we have $\lim_{k\raw\infty}u^{k+1}=\bar{u}\in\partial\phi(\bar{x})$ for some
$\bar{u}$. Thus by definition $\bar{x}$ is a.s. a critical point.

For the second part, let us establish some relation between different iterates. Taking expectation over $x^k$, we have
\begin{align}
& \Ebb\bigl[F_{k}(x^{k})-F_{k}(x^{k+1^{*}})\bigr] \nonumber \\
& \le\Ebb\bigl[F_{k-1}(x^{k})-F_{k}(x^{k+1^{*}})\bigr]\nonumber \\
 & \le\Ebb\bigl[F_{k-1}(x^{k^{*}})-F_{k}(x^{k+1^{*}})\bigr]+2\lambda\bigl[F_{k-1}(x^{k-1})-F_{k-1}(x^{k^{*}})\bigr],\label{eq:Fk-recur}
\end{align}
where the first inequality is obtained from $F_{k}(x^{k})=F(x^{k})\le F_{k-1}(x^{k})$
and the second inequality is obtained from the subproblem convergence
(\ref{eq:subprob-conv}).
\begin{comment}
By our choice of $t$, we have $\lambda<\tfrac{1}{2}$, and it follows
that
\begin{align*}
(1-2\lambda)\tfrac{\wtil{\mu}}{2}\Ebb\bigl[\|x^{k}-x^{k+1^{*}}\|_{\wtil{[1]}}^{2}\bigr]+2\lambda\Ebb\bigl[F_{k}(x^{k})\bigr]+(1-2\lambda)\Ebb\bigl[F_{k}(x^{k+1^{*}})\bigr] \\
\le 2\lambda F_{k-1}(x^{k-1})+(1-2\lambda)F_{k-1}(x^{k^{*}}).
\end{align*}
\end{comment}
Summing up the relation \eqref{eq:Fk-recur}
over $k=1,2,...,K$ and taking expectation over $x_1,x_2,...$, we have
\begin{align}
& \tsum_{k=1}^{K}\Ebb\bigl[F_{k}(x^{k})-F_{k}(x^{k+1^{*}})\bigr] \nonumber \\
& \le\Ebb\bigl[F_{0}(x^{1^{*}})-F_{K}(x^{K+1^{*}})\bigr]+2\lambda\tsum_{k=0}^{K-1}\Ebb\bigl[F_{k}(x^{k})-F_{k}(x^{k+1^{*}})\bigr].\label{eq:sum_fk_diff}
\end{align}
We next derive bounds on  $F_{k}(x^{k+1^{*}})$, $k=0,1,2,...$.
\begin{align}
F_{k}(x^{k+1^{*}}) & \le F_k(x^*)\nonumber \\
 & = F(x^{*})+\left[h(x^{*})-h(x^{k})-\left\langle v^{k},x^{*}-x^{k}\right\rangle \right]+\tfrac{\mu}{2}\|x^{*}-x^{k}\|_{[1]}^{2}\nonumber \\
 & \le F(x^{*})+2M\|x^{k}-x^{*}\|+\tfrac{\mu}{2}\|x^{*}-x^{k}\|_{[1]}^{2}.\label{eq:Fk_upper}
\end{align}
Here, the last inequality uses the relation
\begin{align*}
h(y)-h(x)-\left\langle \nabla h(x),y-x\right\rangle \le &\left\langle \nabla h(y)-\nabla h(x),y-x\right\rangle \\
\le & (\|\nabla h(y)\|+\|\nabla h(x)\|)\,\|x-y\|\le2M\,\|x-y\|,
\end{align*}
for any $x,y$, and $\nabla h(x)\in\partial h(x),\nabla h(y)\in\partial h(y).$ Moreover, we obtain
\begin{align}
F_{k}(x^{k+1^{*}}) & =\min_{x}f(x)+\phi(x)-h(x^{k})-\left\langle v^{k},x-x^{k}\right\rangle +\tfrac{\mu}{2}\|x-x^{k}\|_{[1]}^{2}\nonumber \\
 & \ge\min_{x}f(x)+\phi(x)-h(x)\nonumber \\
 & =F(x^{*}).\label{eq:Fk_lower}
\end{align}
In view of \eqref{eq:sum_fk_diff}, \eqref{eq:Fk_upper} and \eqref{eq:Fk_lower},
we have
\begin{align*}
& (1-2\lambda)\tsum_{k=1}^{K}\Ebb[F_{k}(x^{k})-F_{k}(x^{k+1^{*}})] \\
&\le  F_{0}(x^{1^{*}})-\Ebb [F_{K}(x^{K+1^{*}})]+2\lambda\,\Ebb[F(x^{0})-F_{0}(x^{1^{*}})]\\
& \le 2M\|x^{0}-x^{*}\|+\tfrac{\mu}{2}\|x^{0}-x^{*}\|_{[1]}^{2}+2\lambda\left[F(x^{0})-F(x^{*})\right].
\end{align*}
Taking $t\ge\ln4\frac{m}{\sqrt{\tilde{\mu}}}$, we have $\lambda=(1-\sqrt{\wtil{\mu}}/m)^{t}\le\exp(-(t\sqrt{\wtil{\mu}})/m)\le\tfrac{1}{4}.$
It follows that
\[
\tsum_{k=1}^{K}\Ebb\left[F_{k}(x^{k})-F_{k}(x^{k+1^{*}})\right]\le\left[F(x^{0})-F(x^{*})+4M\|x^{0}-x^{*}\|+\mu\|x^{0}-x^{*}\|_{[1]}^{2}\right].
\]
Moreover, due to (\ref{eq:prox-bound-eps}) and strong convexity of $F_k(\cdot)$, we have 
\[
\tfrac{1}{2\mu}\|p(x^{k},v^{k},\mu)\|_{[1],*}^{2}=\tfrac{\mu}{2}\|x^{k+1^{*}}-x^{k}\|_{[1]}^{2}\le F_{k}(x^{k})-F_{k}(x^{k+1^{*}}).
\]
Putting the above two relations together, we have 
\[
\tsum_{k=1}^{K}\Ebb\left[\|p(x^{k},v^{k},\mu)\|_{[1],*}^{2}\right]\le2\mu\bigl[F(x^{0})-F(x^{*})+4M\|x^{0}-x^{*}\|+\mu\|x^{0}-x^{*}\|_{[1]}^{2}\bigr].
\]
\end{proof}
\begin{remark}
Compared with  DCA, ACPDC  only requires ($\Ocal(\ln m)$) steps 
of ACD to approximately solve each subproblem. In order to obtain
an $\varepsilon$-accurate solution ($\min_{k}\Ebb\bigl[\|p(x^{k},v^{k},\mu)\|_{[1],*}^{2}\bigr]\le\varepsilon$),
the total number of block gradient computations is bounded by 
\[
N\in \Ocal\left(\ln4\tfrac{m}{\sqrt{\mu/(1+\mu)}}\cdot\tfrac{\mu\left[F(x^{0})-F(x^{*})+4M\|x^{0}-x^{*}\|+\mu\|x^{0}-x^{*}\|_{[1]}^{2}\right]}{\varepsilon}\right).
\]
\end{remark}

\section{Randomized Proximal Point Method for Weakly Convex Problems\label{sec:acd-pp}}

In this section we develop a new randomized proximal point algorithm
based on ACD, for minimizing weakly convex functions. We first make
some important assumptions. Specifically, we assume that $h(x)$ is
void and consider the following form
\begin{equation}
\min_{x}F(x)=f(x)+\phi(x).\label{eq:weakly-convex}
\end{equation}
Furthermore we assume that $f(x)$ is $\mu$-weakly convex:
\begin{align}
f(x) & \ge f(y)+\left\langle \nabla f(y),x-y\right\rangle -\frac{\mu}{2}\|x-y\|^{2}.\label{eq:weakly-cvx-assum-1}
\end{align}
We immediately see that the notion of weak convexity is implied by
Lipschitz smoothness. Suppose that $f(x)$ is Lipschitz smooth with
constant $L$: $\|\nabla f(x)-\nabla f(y)\|\le L\|x-y\|$, it is easy
to see that $f(x)\ge f(y)+\left\langle \nabla f(y),x-y\right\rangle -\frac{L}{2}\|x-y\|^{2}$,
namely, that $f(x)$ is $L$-weakly convex. Therefore, it is more
interesting to study nontrivial case when $\mu\neq L$. Throughout
this section, we consider an ill-conditioned weakly convex function
in the sense that $\mu\ll L_{i}$ ($ i\in [m]$). Such scenarios
often arise in a variety of machine learning applications. For example,
in regularized risk minimization, by adding a small nonconvex penalty
such as SCAD and MCP, the problem is weakly convex with a relatively
small value of $\mu$. Cases such as these will be discussed in more
detail in Section \ref{sec:Applications}.

To solve the above-mentioned problem, we present a new ACD-based proximal
point method (ACPP) in Algorithm \ref{alg:acd-pp}. Specifically,
at the $k$-th iteration, given the initial point $x^{k}$, we employ
APCG to approximately solve the following convex composite problem
with some appropriate accuracy:
\[
\min_{x}\,F_{k}(x)=F(x)+\mu\|x-x^{k}\|^{2}.
\]
It can be readily seen that $f(x)+\mu\|x-x^{k}\|^{2}$ is Lipschitz
smooth with $\wtil L=L+2\mu$ and block Lipschitz smooth with $\wtil L_{i}=L_{i}+2\mu$.
Let $s\in[0,1]$ and define the norm $\|x\|_{\ts}=\sqrt{\sum_{i=1}^{m}\wtil L_{i}^{s}\|x\|_{\i}^{2}}.$
Therefore, $F_{k}$ is $\wtil{\mu}_{s}$-strongly convex with norm
$\|\cdot\|_{\text{\ensuremath{\wtil{[s]}}}}$: $F_{k}(x)\ge F_{k}(y)+\bigl\langle\nabla F_{k}(y),x-y\bigr\rangle+\frac{\wtil{\mu}_{s}}{2}\|x-y\|_{\ts}^{2},$
where $\wtil{\mu}_{s}=\frac{\mu}{\wtil L_{\max}^{s}}.$

It should be pointed out that ACPP is closely related to the proximal
DC algorithm. Specifically, consider viewing the objective (\ref{eq:weakly-convex})
as the following difference-of-convex function: 
\begin{equation}
F(x)=\bigl[f(x)+\frac{\mu}{2}\|x\|^{2}+\phi(x)\bigr]-\frac{\mu}{2}\|x\|^{2}.\label{eq:DC-form2}
\end{equation}
To apply a proximal DC algorithm for minimizing (\ref{eq:DC-form2}),
we need to approximately minimize the following function: 
\begin{align*}
\wtil F_{k}(x) & =\bigl[f(x)+\frac{\mu}{2}\|x\|^{2}+\phi(x)\bigr]+\frac{\mu}{2}\|x-x^{k}\|^{2}-\frac{\mu}{2}\|x^{k}\|^{2}-\mu\bigl\langle x^{k},x-x^{k}\bigr\rangle \\
 & =F(x)+\mu\|x-x^{k}\|^{2}=F_{k}(x).
\end{align*}
Here $\wtil F_{k}$ is exactly the function to minimize in the ACPP subproblem.
As a consequence, ACPP can be viewed as a specific proximal DC algorithm
while the earlier technique developed for ACPDC can be adapted to
analyze the convergence of ACPP. Nevertheless, taking into account
the weakly-convex structure, we develop some new convergence analysis
based on the proximal point iteration. By properly choosing the parameters
and termination criterion, we establish new rates of convergence to
approximate stationary point solutions. We show that the convergence
performance of ACPP is much better than that of RCSD and pDCA when
the problem is unconstrained, smooth and ill-conditioned. In such
case, the convergence rates of all the compared algorithms are comparable
when expressing the convergence criteria in terms of $\|\nabla F(x)\|^{2}$.
ACPP has a much better complexity rate in minimizing $\|\nabla F(x)\|^{2}$
in comparison with other single stage CD methods.

Before develop the main convergence result, let us formally define
some optimality measures for problem (\ref{eq:weakly-convex}). %
\begin{comment}
Due to the weak convexity, we can directly develop terminating criteria
for convergence to stationary points.
\end{comment}
Following the setup in \cite{lan2018accelerated}, we say that a point
$x$ is an $(\vep,\delta)$-approximate stationary point if there
exists $\bar{x}$ such that 
\[
\|x-\bar{x}\|^{2}\le\delta\quad\text{and }\quad\bigl[\dis(0,\partial F(\bar{x}))\bigr]^{2}\le\varepsilon.
\]
Moreover, $x$ is a stochastic $(\vep,\delta)$-approximate stationary
point if

\[
\Ebb\|x-\bar{x}\|^{2}\le\delta\quad\text{and }\quad\Ebb\bigl[\dis(0,\partial F(\bar{x}))\bigr]^{2}\le\varepsilon.
\]
Here we define $\dis(y,X)=\inf_{x\in X}\|x-y\|$. Conceptually, an
approximate stationary point is an iterate in proximity to some nearly-stationary
point. Note that similar criteria have been proposed in \cite{davis2019proximally,drusvyatskiy2018efficiency}
for minimizing nonsmooth and weakly convex functions.

\begin{algorithm2e}
\KwIn{$x^{0}$, $\mu$, $t$;}

Compute $\wtil{\mu}_{1}$\;

\For{k=0,1,2,...,K}{

Compute $v^{k}\in\partial h(x^{k})$ \;

Set $F_{k}(x)=F(x)+\mu\|x-x^{k}\|^{2}$\;

Obtain $x^{k+1}$ from running Algorithm \ref{alg:APCG} with input
$F_{k}(x)$, $x^{k}$ , $\wtil{\mu}_{1}$ and $t$\;

}

Choose $\hat{k}$ from $\{2,3,...,K+1\}$ uniformly at random\;

\KwOut{$x^{\hat{k}}$.}

\caption{ACPP\label{alg:acd-pp}}
\end{algorithm2e}

In the following theorem, we develop the main convergence property
of Algorithm \ref{alg:acd-pp}.
\begin{theorem}
\label{thm:acd-composi}In Algorithm \ref{alg:acd-pp}, let $\kappa=\wtil L_{\max}/\wtil L_{\min}$,
$\eta=\frac{\sqrt{\wtil{\mu}_{1}}}{m}$ and assume that $\lambda=(1-\eta)^{t}<\frac{1}{2}$.
Then there exists a random $x^{\hat{k}^{*}}$ such that
\begin{align}
\Ebb\|x^{\hat{k}}-x^{\hat{k}^{*}}\|^{2} & \le\tfrac{4\kappa\lambda}{K\mu(1-2\lambda)}\Bigl\{\mu\|x^{0}-x^{*}\|^{2}+2\lambda[F(x^{0})-F(x^{*})]\Bigr\},\label{eq:acpp-pt-dist1}\\
\Ebb\Bigl[\dis(0,\partial F(x^{\hat{k}^{*}}))\Bigr]^{2} & \le\tfrac{8\kappa\mu}{K(1-2\lambda)}\Bigl\{\mu\|x^{0}-x^{*}\|^{2}+2\lambda[F(x^{0})-F(x^{*})]\Bigr\}.\label{eq:acpp-subdiff-bound1}
\end{align}
In particular, if we set $t=t_0=\bigl\lceil-\eta^{-1}\ln\wtil{\lambda}\bigr\rceil$
where $\wtil{\lambda}=\min\left\{ \tfrac{1}{4},\frac{\mu^{2}}{L^{2}},\frac{\mu}{2L}\right\} $,
then 
\begin{align}
\Ebb\|x^{\hat{k}}-x^{\hat{k}^{*}}\|^{2} & \le\tfrac{8\kappa\mu}{KL^{2}}\Bigl\{\mu\|x^{0}-x^{*}\|^{2}+\tfrac{\mu}{L}[F(x^{0})-F(x^{*})]\Bigr\},\label{eq:acpp-pt-dist2}\\
\Ebb\Bigl[\dis(0,\partial F(x^{\hat{k}^{*}}))\Bigr]^{2} & \le\tfrac{16\kappa\mu}{K}\Bigl\{\mu\|x^{0}-x^{*}\|^{2}+\tfrac{\mu}{L}[F(x^{0})-F(x^{*})]\Bigr\}.\label{eq:acpp-subdiff-bound2}
\end{align}
\end{theorem}
\begin{proof}
First of all, due to the convergence of APCG (Theorem \ref{thm:apcg}),
we have 
\[
F_{k}(x^{k+1})-F_{k}(x^{k+1^{*}})\le\lambda\bigl[F_{k}(x^{k})-F_{k}(x^{k+1^{*}})+\tfrac{\tilde{\mu}_{1}}{2}\|x^{k}-x^{k+1^{*}}\|_{\wtil{[1]}}^{2}\bigr],\quad k=0,1,2,....
\]
Moreover, since $x^{k+1^{*}}$ is optimal for minimizing $F_{k}(\cdot)$,
we have the upper-bound of $F_{k}(x^{k+1^{*}})$:
\begin{equation}\label{eq:fk-upper}
F_{k}(x^{k+1^{*}}) \le F_{k}(x^{*})=F(x^{*})+\mu\|x^{k}-x^{*}\|^{2},
\end{equation}
and the lower-bound of $F_{k}(x^{k+1^{*}})$:
\begin{equation}\label{eq:fk-lower}
F_{k}(x^{k+1^{*}}) =\min_{x}f(x)+\phi(x)+\mu\|x-x^{k}\|^{2}\ge\min_{x}f(x)+\phi(x)=F(x^{*}).
\end{equation}
For any $k=0,1,2,...$, we deduce the relation
\begin{align}
& \Ebb\Bigl[F_{k}(x^{k})-F_{k}(x^{k+1^{*}})\Bigr]  \nonumber \\
& \le \Ebb\Bigl[F_{k-1}(x^{k})-F_{k}(x^{k+1^{*}})\Bigr]\nonumber \\
&  \le \Ebb\Bigl[F_{k-1}(x^{k^{*}})-F_{k}(x^{k+1^{*}})\Bigr]  +\lambda\Ebb\Bigl[F_{k-1}(x^{k-1})-F_{k-1}(x^{k^{*}})+\frac{\tilde{\mu}_{1}}{2}\|x^{k-1}-x^{k^{*}}\|_{\tone}^{2}\Bigr]\nonumber \\
&  \le \Ebb\Bigl[F_{k-1}(x^{k^{*}})-F_{k}(x^{k+1^{*}})\Bigr]+2\lambda\Ebb\left[F_{k-1}(x^{k-1})-F_{k-1}(x^{k^{*}})\right],\label{eq:sum-sq-bound}
\end{align}
where the first inequality is due to $F_{k}(x^{k})\le F(x^{k})+\mu\|x-x^{k-1}\|^{2}=F_{k-1}(x^{k})$,
the second inequality is due to the convergence of APCG (Theorem \ref{thm:apcg}),
and the third inequality is due to the strong convexity of $F_{k-1}(\cdot)$.
Summing up the above over $k=1,2,3,...K$, and then rearranging terms
appropriately, we have 
\begin{align}
& \tsum_{k=1}^{K}\Ebb\bigl[F_{k}(x^{k})-F_{k}(x^{k+1^{*}})\bigr] \nonumber \\
&\le  \tfrac{1}{1-2\lambda}\Bigl\{ F_{0}(x^{1^{*}})-\Ebb\bigl[F_{K}(x^{K+1^{*}})\bigr]+2\lambda\Ebb\bigl[F_{0}(x^{0})-F_{0}(x^{1^{*}})\bigr]\Bigr\}.\nonumber \\
&\le \tfrac{1}{1-2\lambda}\Bigl\{\mu\|x^{0}-x^{*}\|^{2}+2\lambda[F(x^{0})-F(x^{*})]\Bigr\}.
\end{align}
Here, the second inequality uses (\ref{eq:fk-upper}) and (\ref{eq:fk-lower}).
Applying the randomness of $\hat{k}$, we have 
\begin{align}
\Ebb\bigl[F_{\hat{k}-1}(x^{\hat{k}-1})-F_{\hat{k}-1}(x^{\hat{k}^{*}})\bigr] & \le\tfrac{1}{K}\tsum_{k=1}^{K}\Ebb\bigl[F_{k}(x^{k})-F_{k}(x^{k+1^{*}})\bigr]\nonumber \\
 & \le\tfrac{1}{K(1-2\lambda)}\Bigl\{\mu\|x^{0}-x^{*}\|^{2}+2\lambda[F(x^{0})-F(x^{*})]\Bigr\}.\label{eq:fkhat}
\end{align}
We next derive an upper-bound on $\Ebb\|x^{\hat{k}}-x^{k^{*}}\|^{2}$.
Taking expectation over $x^{k}$, we have
\begin{align}
\Ebb\|x^{\hat{k}}-x^{\hat{k}^{*}}\|^{2} & \le\tfrac{1}{\wtil L_{\min}}\Ebb\|x^{\hat{k}}-x^{\hat{k}^{*}}\|_{\wtil{[1]}}^{2}\nonumber \\
 & \le\tfrac{2}{\tilde{\mu}_{1}\wtil L_{\min}}\Ebb[F_{\hat{k}-1}(x^{\hat{k}})-F_{\hat{k}-1}(x^{\hat{k}^{*}})]\nonumber \\
 & \le\tfrac{2\lambda}{\tilde{\mu}_{1}\wtil L_{\min}}\bigl[F_{\hat{k}-1}(x^{\hat{k}-1})-F_{\hat{k}-1}(x^{\hat{k}^{*}})+\tfrac{\tilde{\mu}_{1}}{2}\|x^{\hat{k}-1}-x^{\hat{k}^{*}}\|_{\wtil{[1]}}^{2}\bigr]\nonumber \\
 & \le\tfrac{4\kappa\lambda}{\mu}\bigl[F_{\hat{k}-1}(x^{\hat{k}-1})-F_{\hat{k}-1}(x^{\hat{k}^{*}})\bigr].\label{eq:dist-to-near-stat}
\end{align}
In addition, by the optimality of $x^{\hat{k}^{*}}$ we have $0\in\partial F(x^{\hat{k}^{*}})+2\mu(x^{\hat{k}^{*}}-x^{\hat{k}-1})$,
it follows that
\begin{align}
\Ebb\bigl[\dis(0,\partial F(x^{\hat{k}^{*}}))\bigr]^{2} & \le4\mu^{2}\Ebb\bigl[\|x^{\hat{k}^{*}}-x^{\hat{k}-1}\|^{2}\bigr]\le\tfrac{4\mu^{2}}{\wtil L_{\min}}\Ebb\bigl[\|x^{\hat{k}-1}-x^{\hat{k}^{*}}\|_{\tone}^{2}\bigr] \nonumber \\
 & \le 8\mu\kappa\Ebb\bigl[F_{\hat{k}-1}(x^{\hat{k}-1})-F_{\hat{k}-1}(x^{\hat{k}^{*}})\bigr].\label{eq:dist-stationary}
\end{align}
Consequently, we obtain (\ref{eq:acpp-pt-dist1}) by combining (\ref{eq:fkhat})
with (\ref{eq:dist-to-near-stat}) and obtain (\ref{eq:acpp-subdiff-bound1})
by combining (\ref{eq:fkhat}) with (\ref{eq:dist-stationary}).

Finally, given that $t\ge\bigl\lceil-\eta^{-1}\ln\wtil{\lambda}\bigr\rceil$
with $\wtil{\lambda}=\min\left\{ \frac{1}{4},\frac{\mu^{2}}{L^{2}},\frac{\mu}{2L}\right\} $,
we have $\lambda\le\exp(-\eta t)\le\wtil{\lambda}$, and therefore,
$(1-2\lambda)^{-1}\le2$. The results (\ref{eq:acpp-pt-dist2}) and
(\ref{eq:acpp-subdiff-bound2}) follow immediately.
\end{proof}

\subsection*{Unconstrained smooth weakly-convex optimization}
In Algorithm \ref{alg:acd-pp-smooth}, we describe a version of ACPP
for unconstrained smooth nonconvex optimization; this is a special
case of Problem \eqref{main-problem} with $\phi(x)$ being void. Since $F_{k}(\cdot)$
is a smooth and strongly convex function, it can be more efficiently
optimized by ACD with importance sampling (\cite{allen-zhu2016even,nesterov2017efficiency}).
For the sake of completeness, we provide an extension of APCG for
smooth optimization with non-uniform sampling in Algorithm \ref{alg:ACD-nuni}.
We state the convergence result in the following theorem and present
the formal proof in the appendix for brevity.

\begin{algorithm2e}[H]
\KwIn{$x_{0}$, $\mu$, $t$, $s$;}

\For{k=0,1,2,...,K}{

Set $F_{k}(x)=F(x)+\mu\|x-x^{k}\|^{2}$\;

Obtain $x^{k+1}$ from  Algorithm \ref{alg:ACD-nuni} with
input $F_{k}(x)$, $x^{k}$ , $\wtil{\mu}_{s}$ , $s$ and $t$\;

}

Choose $\hat{k}$ from $\{2,3,...,K+1\}$ uniformly at random\;

\KwOut{$x^{\hat{k}}$.}

\caption{ACPP for smooth nonconvex optimization \label{alg:acd-pp-smooth}}
\end{algorithm2e}

\begin{algorithm2e}[H]
\KwIn{A smooth convex function $f(x)$, $x_{0}$, $\mu_{s}$, $s$,
$K$;}

\For{k=0,1,2,...K-1}{

Set $y^{k}=(1-\alpha_{k})x^{k}+\alpha_{k}z^{k}$\;

Sample random block $i_{k}\in\left\{ 1,2,3,...,m\right\} $ with probability
$p_{i}\propto L_{i}^{(1-s)/2}$\;

\[
z^{k+1}=\argmin_{x}\Bigl\{\bigl\langle\nabla_{i_{k}}f(y^{k}),x_{\i}\bigr\rangle+\frac{p_{i_{k}}\gamma}{2}\|x-y^{k}\|_{[s]}^{2}+\frac{p_{i_{k}}\beta_{k}}{2}\|x-z^{k}\|_{[s]}^{2}\Bigr\}.
\]
Set
\[
\bar{x}^{k+1}=y^{k}+\frac{\alpha_{k}}{p_{i_{k}}}\Ubf_{\ik}(z_{\ik}^{k+1}-z_{\ik}^{k});
\]

Option I: $x_{\ik}^{k+1}=\argmin_{x}\left\{ \left\langle \nabla_{i_{k}}f(\bar{x}^{k+1}),x\right\rangle +\frac{L_{i_{k}}}{2}\|x-\bar{x}_{\ik}^{k+1}\|_{\ik}^{2}\right\} $
\;

and $x_{j}^{k+1}=\bar{x}_{j}^{k+1}$ for $j\neq i_{k}$\;

Option II: $x^{k+1}=\bar{x}^{k+1}$\;

}

\KwOut{$x^{K}$\;}\caption{ACD for smooth convex optimization\label{alg:ACD-nuni}}
\end{algorithm2e}

\begin{theorem}[Informal]
\label{thm:acd-smooth-informal}Assume that $f(x)$ is
smooth with block Lipschitz constant $\left\{ L_{i}\right\} _{1\le i\le m}$
and is $\mu_{s}$-strongly convex with norm $\|\cdot\|_{[s]}$, where
$s\in[0,1]$. Denote $T_{(1-s)/2}=\sum_{i=1}^{m}L_{i}^{(1-s)/2}$.
If the coordinates are sampled with probability $p_{i}\propto L_{i}^{(1-s)/2}$,
then we have
\[
\Ebb[f(x^{K})-f(x^{*})]\le\left(1-\tfrac{\sqrt{\mu_{s}}}{\sqrt{\mu_{s}}+T_{(1-s)/2}}\right)^{K}\left[f(x^{0})-f(x^{*})+\tfrac{\mu_{s}}{2}\|x-x^{0}\|_{[s]}^{2}\right].
\]
\end{theorem}
Theorem \ref{thm:acd-smooth-informal} describes the efficiency of
solving the subproblems of ACPP using ACD method. Incorporating this
result, we present the overall complexity result of Algorithm \ref{alg:acd-pp-smooth}
in the following theorem.
\begin{theorem}
\label{thm:acpp} Assume that $f(x)$ is $\mu$-weakly convex and
that $\phi(x)=0$ in Problem (\ref{eq:weakly-convex}). In Algorithm
\ref{alg:acd-pp-smooth}, let $\kappa_{s}=\max_{i}\wtil L_{i}^{s}/\min_{i}\wtil L_{i}^{s}$,
$T_{(1-s)/2}=\sum_{i=1}^{m}\wtil L_{i}^{(1-s)/2}$, $\wtil{\mu}_{s}={\mu}/{\wtil L_{\max}^{s}}$,
$\eta=\sqrt{\tilde{\mu}_{s}}/\bigl(\sqrt{\wtil{\mu}_{s}}+T_{(1-s)/2}\bigr)$,
and assume that $\lambda=(1-\eta)^{t}<\frac{1}{2}$, we have
\begin{equation}
\Ebb\|\nabla F(x^{\hat{k}})\|^{2}\le\tfrac{8\lambda L^{2}\kappa_{s}+16\mu^{2}\kappa_{s}}{K\mu(1-2\lambda)}\left[\mu\|x^{0}-x^{*}\|^{2}+2\lambda[F(x^{0})-F(x^{*})]\right].\label{eq:smooth-acpp-convg-main}
\end{equation}
In particular, if we set $t\ge\lceil-\eta^{-1}\ln\wtil{\lambda}\rceil$
with $\wtil{\lambda}=\min\{\frac{1}{8},\frac{\mu}{4L},\frac{\mu^{2}}{L^{2}}\}$,
in order to guarantee that $\Ebb\|\nabla F(x^{\hat{k}})\|^{2}\le\vep$,
the total number of block gradient updates is at most 
\[
%N_{\vep}=\biggl\lceil\Bigl(\mu+\sqrt{\mu}\wtil L_{\max}^{s}T_{(1-s)/2}\Bigr)\log\Bigl(\max\bigl\{8,\tfrac{4L}{\mu},\tfrac{L^{2}}{\mu^{2}}\bigr\}\Bigr)\biggr\rceil\cdot\biggl\lceil\left(\tfrac{40\kappa_{s}\|x^{0}-x^{*}\|^{2}}{\varepsilon}+1\right)\biggr\rceil.
N_{\vep}=\Ocal\left(\left(1+\wtil L_{\max}^{s}T_{(1-s)/2}/\sqrt{\wtil{\mu}_s}\right)\log\left(\max\bigl\{8,\tfrac{4L}{\mu},\tfrac{L^{2}}{\mu^{2}}\bigr\}\right)\left(\tfrac{40\mu^2\kappa_{s}\|x^{0}-x^{*}\|^{2}}{\varepsilon}+1\right)\right).
\]
\end{theorem}
\begin{proof}
Similar to relation (\ref{eq:sum-sq-bound}), we have 
\begin{align*}
&\Ebb\Bigl[F_{k}(x^{k})-F_{k}(x^{k+1^{*}})\Bigr]\\ 
&\le\Ebb\Bigl[F_{k-1}(x^{k^{*}})-F_{k}(x^{k+1^{*}})\Bigr]+2\lambda\Ebb\left[F_{k-1}(x^{k-1})-F_{k-1}(x^{k^{*}})\right].
\end{align*}
Summing up the above inequality over $k=1,2,3,...,K$ and rearranging
the terms accordingly, we arrive at
\begin{align}
& (1-2\lambda)\tsum_{k=1}^{K}\Ebb\bigl[F_{k}(x^{k})-F_{k}(x^{k+1^{*}})\bigr] \nonumber \\
& \le F_{0}(x^{1^{*}})-\Ebb\bigl[F_{K}(x^{K+1^{*}})\bigr]+2\lambda\Ebb\bigl[F_{0}(x^{0})-F_{0}(x^{1^{*}})\bigr]\nonumber \\
 & \le \mu\,\|x^{0}-x^{*}\|^{2}+2\lambda[F(x^{0})-F(x^{*})].\label{fk-res-sum}
\end{align}
In addition, taking the expectation conditioned on $x^{k}$, we have
\begin{align}
\Ebb\|x^{k+1}-x^{k+1^{*}}\|_{\ts}^{2} & \le\tfrac{2}{\wtil{\mu}_{s}}\Ebb[F_{k}(x^{k+1})-F_{k}(x^{k+1^{*}})]\nonumber \\
 & \le\tfrac{2\lambda}{\wtil{\mu}_{s}}\bigl[F_{k}(x^{k})-F_{k}(x^{k+1^{*}})+\tfrac{\tilde{\mu}_{s}}{2}\|x^{k}-x^{k+1^{*}}\|_{\ts}^{2}\bigr].\nonumber \\
\text{} & \le\tfrac{4\lambda}{\wtil{\mu}_{s}}\bigl[F_{k}(x^{k})-F_{k}(x^{k+1^{*}})\bigr].\label{eq:res-xk-bound}
\end{align}
For the $k$-th subproblem, we have 
\begin{align}
\|\nabla F(x^{k+1})\|^{2} & \le2\|\nabla F(x^{k+1})-\nabla F(x^{k+1^{*}})\|^{2}+2\|\nabla F(x^{k+1^{*}})\|^{2}\nonumber \\
 & \le2L^{2}\,\|x^{k+1}-x^{k+1^{*}}\|^{2}+8\mu^{2}\|x^{k}-x^{k+1^{*}}\|^{2}\nonumber \\
 & \le\tfrac{2L^{2}}{\wtil L_{\min}^{s}}\|x^{k+1}-x^{k+1^{*}}\|_{\ts}^{2}+\tfrac{8\mu^{2}}{\wtil L_{\min}^{s}}\|x^{k}-x^{k+1^{*}}\|_{\ts}^{2}\nonumber \\
 & \le\tfrac{8\lambda L^{2}\kappa_{s}+16\mu^{2}\kappa_{s}}{\mu}\bigl[F_{k}(x^{k})-F_{k}(x^{k+1^{*}})\bigr].\label{bound-grad-sq}
\end{align}
Here, the first inequality uses $\|a+b\|^{2}\le2\|a\|^{2}+2\|b\|^{2}$,
the second inequality uses Lipschitz continuity of $\nabla F(x)$
and the optimality condition $\nabla F\bigl(x^{k+1^{*}}\bigr)=2\mu\bigl(x^{k}-x^{k+1^{*}}\bigr)$,
the third inequality uses the relation $\|x\|_{\text{\ensuremath{\ts}}}^{2}\ge\wtil L_{\min}^{s}\|x\|^{2}$,
and the fourth inequality follows from (\ref{eq:res-xk-bound}) and
$\|x^{k}-x^{k+1^{*}}\|_{\ts}^{2}\le\frac{2}{\tilde{\mu}_{s}}\bigl[F_{k}(x^{k})-F_{k}(x^{k+1^{*}})\bigr]$.

Summing up the relation (\ref{bound-grad-sq}) over $k=0,1,2,...$
and then combining it with (\ref{fk-res-sum}), we have 
\[
\tsum_{k=1}^{K}\Ebb\|\nabla F(x^{k+1})\|^{2}\le\tfrac{8\lambda L^{2}\kappa_{s}+16\mu^{2}\kappa_{s}}{\mu(1-2\lambda)}\left[\mu\|x^{0}-x^{*}\|^{2}+2\lambda[F(x^{0})-F(x^{*})]\right].
\]
The result (\ref{eq:smooth-acpp-convg-main}) immediately follows.

In addition, $\lambda\le(1-\eta)^{t}\le\min\Bigl\{\tfrac{1}{8},\tfrac{\mu}{4L},\tfrac{\mu^{2}}{L^{2}}\Bigr\}$,
then we have
\begin{align*}
\Ebb\|\nabla F(x^{\hat{k}})\|^{2} & =\tfrac{1}{K}\tsum_{k=1}^{K}\Ebb\|\nabla F(x^{k+1})\|^{2}\le\tfrac{8\lambda L^{2}\kappa_{s}+16\mu^{2}\kappa_{s}}{\mu(1-2\lambda)K}(\mu+\lambda L)\|x^{0}-x^{*}\|^{2}\\
& \le\tfrac{40\mu^2\kappa_{s}}{K}\|x^{0}-x^{*}\|^{2},
\end{align*}
where the first inequality uses $F(x^{0})-F(x^{*})\le\frac{L}{2}\|x^{0}-x^{*}\|^{2}$.

Next we estimate the total number of block gradient computations $N_{\vep}$
for some $\varepsilon$-accurate solution. Let $K\ge\frac{40\mu\kappa_{s}}{\varepsilon}\|x^{0}-x^{*}\|^{2}$.
It suffices to choose
\[
N_{\vep}=t(K+1)\ge\left(1+\wtil L_{\max}^{s}T_{(1-s)/2}/\sqrt{\wtil{\mu}_s}\right)\log\left(\max\bigl\{8,\tfrac{4L}{\mu},\tfrac{L^{2}}{\mu^{2}}\bigr\}\right)\left(\tfrac{40\mu^2\kappa_{s}\|x^{0}-x^{*}\|^{2}}{\varepsilon}+1\right).
\]
\end{proof}
The above theorem describes the iteration complexity of ACD to obtain
an approximate stationary point when sampling probability takes the
form $p_{i}\propto\wtil L_{i}^{(1-s)/2}$. To the best of our knowledge,
the best rate of ACD (\cite{nesterov2017efficiency,allen-zhu2016even})
is achieved for $s=0$. In the following result, we develop specific
complexity result using such sampling strategy.
\begin{corollary}
\label{thm:smooth-acpp}If in Algorithm \ref{alg:acd-pp}, we set $s=0$ and choose sample probability $p_{i}\propto\sqrt{\wtil L_{i}}$,
the total number of block gradient computations is bounded by 
\begin{equation}
N_{\varepsilon}=\Ocal\left(\bigl(\tsum_{i=1}^{m}\sqrt{L_{i}}\bigr)\sqrt{\mu}\log\left(\tfrac{L}{\mu}\right)\tfrac{\mu\|x^{0}-x^{*}\|^{2}}{\varepsilon}\right).\label{eq:complex-smooth-uncon}
\end{equation}
\end{corollary}
\begin{proof}
By setting $s=0$, we have $\kappa_{s}=1$. The result immediately
follows.
\end{proof}
\begin{remark}
It is interesting to compare ACPP and RCD for their iteration complexities.
Our earlier analysis in Section \ref{sec:rcd} implies that RCD requires
$\Ocal\bigl(\bigl(\sum_{i=1}^{m}L_{i}\bigr)/\varepsilon\bigr)$ block
gradient computations for achieving an $\varepsilon$-accurate solution.
In comparison, Theorem \ref{thm:smooth-acpp} implies that ACPP needs
much less number of such computations when the  problem is ill-conditioned
with block Lipschitz constants $L_{i}\gg\mu$ ($i\in[m]$). 
\end{remark}

\section{Applications\label{sec:Applications}}

\subsection{Sparsity-inducing machine learning\label{subsec:Sparse-learning}}

As a notable application of Problem \eqref{main-problem}, we consider
an important class of sparse learning problems that are described
in the following form:\begin{mini}
{x\in\Rbb^d}{f(x)}{}{}
\addConstraint{\|x\|_0}{\le k}{}
{\label{l0-constraint}}
\end{mini}Here, $f(x)$ is some loss function and the $l_{0}$ norm constraint
promotes sparsity. Due to the nonconvexity and noncontinuity of $l_{0}$-norm,
direct optimization of the above problem is generally intractable.
An alternative way is to translate (\ref{l0-constraint}) into a regularized
problem with some sparsity-inducing penalty $\psi(x)$.\begin{mini}
{x\in\Rbb^d}{f(x)+\psi(x)}{}{}
{\label{l0-penalty}}
\end{mini}While convex relaxation such as the $l_{1}$ penalty ($\psi(x)=\lambda\|x\|_{1}$)
has been widely studied in the literature, nonconvex regularization
has received increasing popularity recently (for example, see \cite{gong2013a,thi2015dc,gotoh2018dc}).
Here, we consider a wide class of nonconvex regularizers that take
the form of a difference-of-convex (DC) function. For example, SCAD
(\cite{fan2001variable,wen2018proximal,gong2013a}) penalty has the
separable form $\psi_{\lambda,\theta}(x)=\sum_{i=1}^{m}\left[\phi_{\lambda}(x_{\i})-h_{\lambda,\theta}(x_{\i})\right]$
where $\phi_{\lambda}(\cdot)$ and $h_{\lambda,\theta}(\cdot)$ are
defined by
{
\begin{equation}
\phi_{\lambda}(x)=\lambda|x_{\i}|,\quad\text{and}\quad h_{\lambda,\theta}(x)=\begin{cases}
0 & \text{if }|x|\le\lambda\\
\frac{x^{2}-2\lambda|x|+\lambda^{2}}{2(\theta-1)} & \text{if }\lambda<|x|\le\theta\lambda\\
\lambda|x|-\frac{1}{2}(\theta+1)\lambda^{2} & \text{if }|x|>\theta\lambda
\end{cases}\label{eq:scad}
\end{equation}
}respectively. It is routine to check that $h_{\lambda,\theta}(x)$
is Lipschitz smooth with constant $\frac{1}{\theta-1}$.

Another type of interesting sparsity-inducing regularizers arises
from direct reformulation of the $l_{0}$ term. For instance, the
constraint $\left\{ \|x_{0}\|\le k\right\} $ can be equivalently
expressed as $\{\|x\|_{1}-|||x|||{}_{k}=0\}$, where $|||x|||_{k}$
is the largest-$k$ norm, defined by $|||x|||_{k}=\sum_{i=1}^{d}\bigl|x_{j_{i}}\bigr|$,
where $j_{1},j_{2},...j_{d}$ are coordinate indices in descending
order of absolute value. \cite{gotoh2018dc} used this observation
when studying the following nonconvex composite problem
\begin{equation}
\min_{x\in\Rbb^{d}}F(x)=f(x)+\lambda\|x\|_{1}-\lambda|||x|||{}_{k}.\label{eq:relaxed-ls}
\end{equation}
Notice that existing nonconvex coordinate descent \cite{patrascu2015efficient,xu2017a}
are not applicable to (\ref{eq:relaxed-ls}), since the norm $|||\cdot|||_{k}$
is neither smooth nor separable. In contrast, our methods can be applied
to solve Problem (\ref{eq:relaxed-ls}) since the concave part of
the objective is allowed to be nonsmooth as well as inseparable. Moreover,
efficient implementations of our methods are viable: in RPCD and ACPDC,
full subgradient of $|||\cdot|||_{k}$ is computed only once a while;
in RCSD, block subgradient of $|||\cdot|||_{k}$ can be computed in
logarithmic time by maintaining the coordinates in max heaps. Therefore,
in all three algorithms, the overheads to compute subgradient are
negligible.

\subsection{Experiments \label{sec:Experiments}}

We conduct empirical experiments on the nonconvex problems described
in Subsection \ref{subsec:Sparse-learning}. We compare our algorithms
with two gradient-based methods. The first algorithm is the proximal
DC algorithm (pDCA, \cite{RN343,gotoh2018dc}) which performs a single
step of proximal gradient descent in each iteration. The second algorithm
is an enhanced proximal DC algorithm using extrapolation technique
($\text{pDCA}_{e}$ \cite{wen2018proximal}). In the paper \cite{wen2018proximal},
$\text{pDCA}_{e}$ has been reported to obtain better performance
than pDCA on a variety of nonconvex learning problems.

\paragraph{Datasets}

We use both synthetic and real data. The \texttt{synthetic}
dataset is based on the study in \cite{RN343}. Namely, we generate
an $n\times d$ matrix $A$, of which the rows are generated from
$d$-dim Gaussian distribution $\Ncal(0,\Sigma)$. Here, the covariance
matrix $\Sigma$ satisfies: $\Sigma_{ii}=1$ ($1\le i\le d$) and
$\Sigma_{ij}=0.7$ ($i\neq j$). We generate the true solution $x^{*}$
from a random binary vector, of which $s$ nonzeros are chosen uniformly
without replacement. Real data are collected from online repositories.
Among them, \texttt{E2006-tfidf}, \texttt{E2006-log1p, real-sim},
\texttt{news20.binary}, \texttt{mnist} are from \cite{CC01a}, and\texttt{
UJIndoorLoc},\texttt{ dexter} are from \cite{Dua:2019}. More details
are described in Table \ref{tab:Datasets}.

\paragraph*{Parameter setting}

In CD methods, the block number $m$ is set to 10000 for both \texttt{news20.binary}
and \texttt{E2006-log1p} and is set to $\min\{1000,d\}$ for the other
datasets. For both ACPDC and ACPP, we choose $t$ from the range $[m, t_0]$ and tune it
as a hyper-parameter. For ACPDC,
we set $\mu=0.01$ because this value consistently yields good
performance. In all the experiments, $x^{0}=0$ is used as the initial
point for all the algorithms. The performance of CD methods is reported
on the average of ten replications.

\begin{table}[H]
\centering{}%
\begin{tabular}{cccccc}
\toprule 
Datasets & $n$ & $d$ & $m$ & Sparsity & Problem type\tabularnewline
\midrule 
\texttt{synthetic} & 500 & 5000 & 1000 & 100\% & R\tabularnewline
\midrule 
\texttt{E2006-tfidf} & 16087 & 150360 & 1000 & 0.826\% & R\tabularnewline
\midrule 
\texttt{E2006-log1p} & 16087 & 4272227 & 10000 & 0.141\% & R\tabularnewline
\midrule 
\texttt{UJIndoorLoc} & 19937 & 554 & 554 & 50\% & R\tabularnewline
\midrule 
\texttt{resl-sim} & 72309 & 20959 & 1000 & 0.25\% & C\tabularnewline
\midrule 
\texttt{news20.binary} & 19996 & 1355191 & 10000 & 0.034\% & C\tabularnewline
\midrule 
\texttt{mnist} & 60000 & 718 & 718 & 21.02\% & C\tabularnewline
\midrule 
\texttt{dexter} & 600 & 20000 & 1000 & 0.848\% & C\tabularnewline
\bottomrule
\end{tabular}\caption{\label{tab:Datasets}Dataset description. $n$ is the number of samples,
$d$ is the feature dimension, $m$ is the block number. R stands
for regression and C stands for classification.}
\end{table}

\paragraph{Logistic loss + largest-$k$ norm penalty}

Our first experiment examines the performance of RCSD, RPCD and ACPDC
on logistic loss classification with largest-$k$ norm penalty:

\[
\min_{x\in\Rbb^{d}}\frac{1}{n}\sum_{i=1}^{n}\log(1+\exp(-b_{i}(a_{i}^{T}x)))+\frac{\rho}{d}\bigl(\|x\|_{1}-|||x|||{}_{k}\bigr).
\]
We use the following real datasets: \texttt{resl-sim}, \texttt{news20.binary}, \texttt{dexter} and \texttt{mnist}. 
For the \texttt{mnist} dataset, we formulate
a binary classification problem by labeling the digits 0, 4, 5, 6,
8  positive and all other digits  negative. We plot the function
objectives with respect to number of gradient evaluations for various
values of weight $\rho$; our results are shown in Figure \ref{fig:Exper-logistic}.

We now make a number of important observations from Figure \ref{fig:Exper-logistic}.
First, $\text{pDCA}_{e}$ consistently outperforms pDCA in all experiments.
This result suggests that, despite the unclear theoretical advantage
of $\text{pDCA}_{e}$ over pDCA, extrapolation indeed has empirical
advantage in nonsmooth and nonconvex optimization. Second, RPCD exhibits
at least the same (sometimes better) performance as RCSD, while both
RPCD and RCSD have superior performance when compared with gradient
methods. Third, ACPDC achieves the best performance among all the
tested algorithms.

\begin{figure}[h]
\begin{centering}
\includegraphics[scale=0.24]{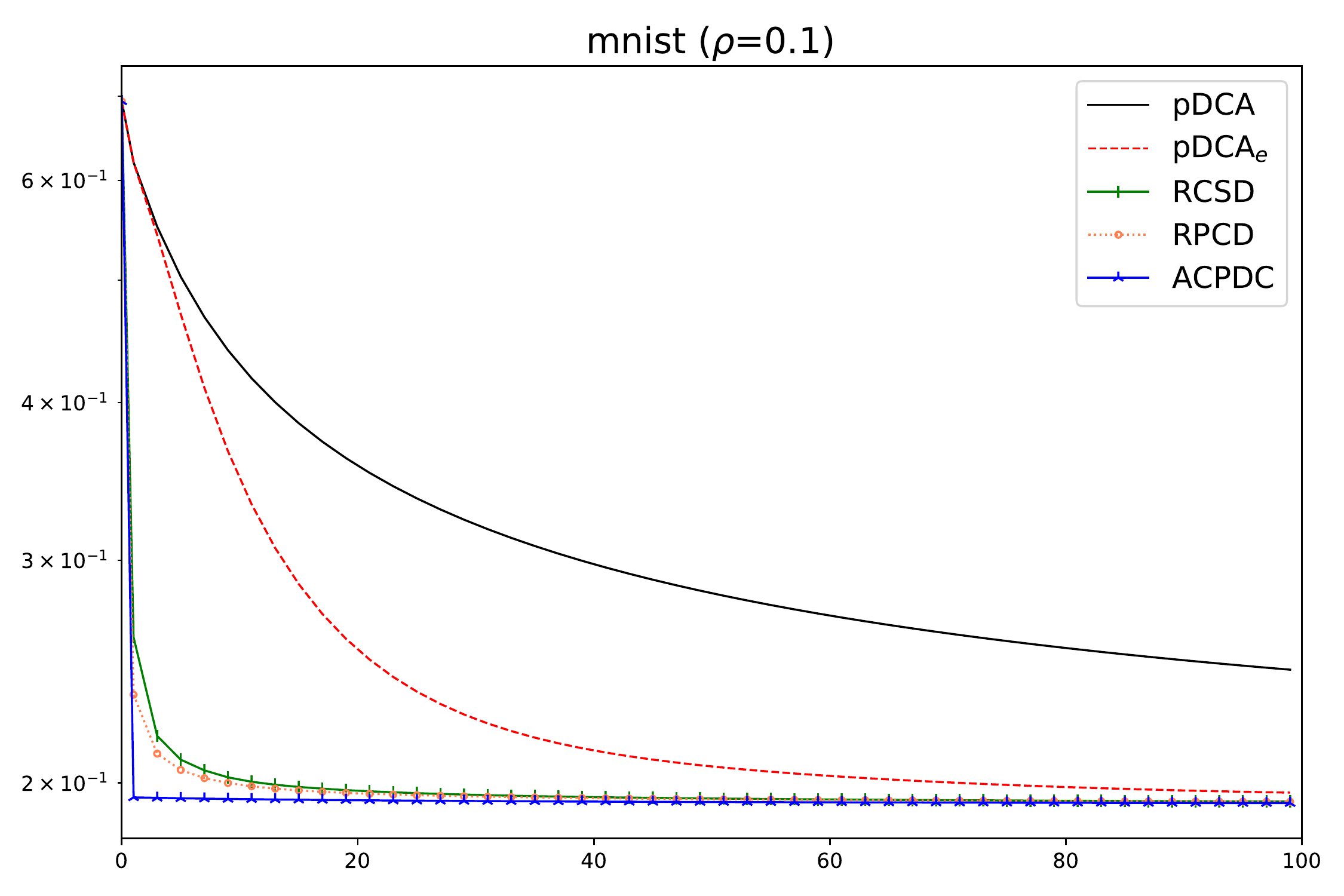}
\includegraphics[scale=0.24]{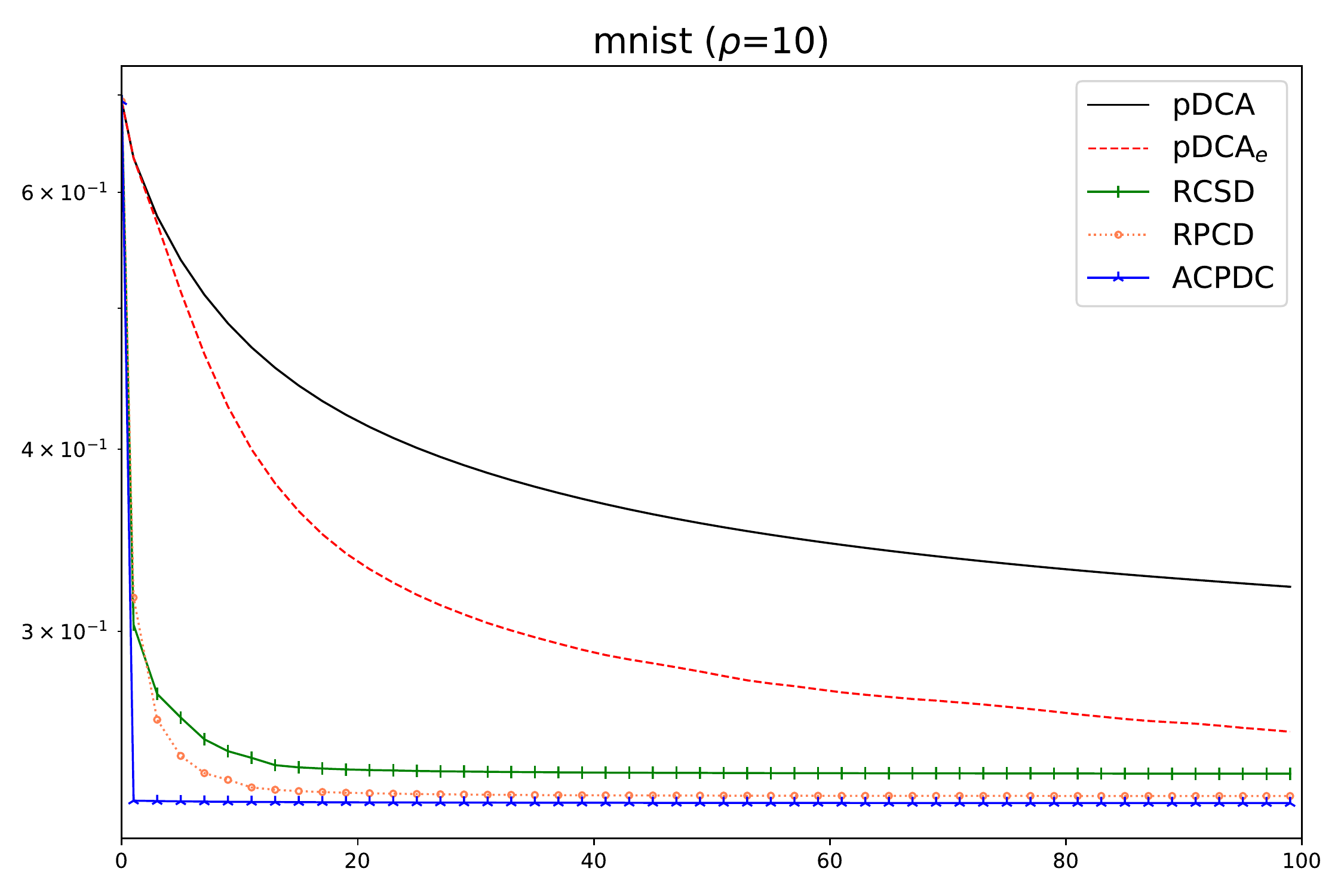}\includegraphics[scale=0.24]{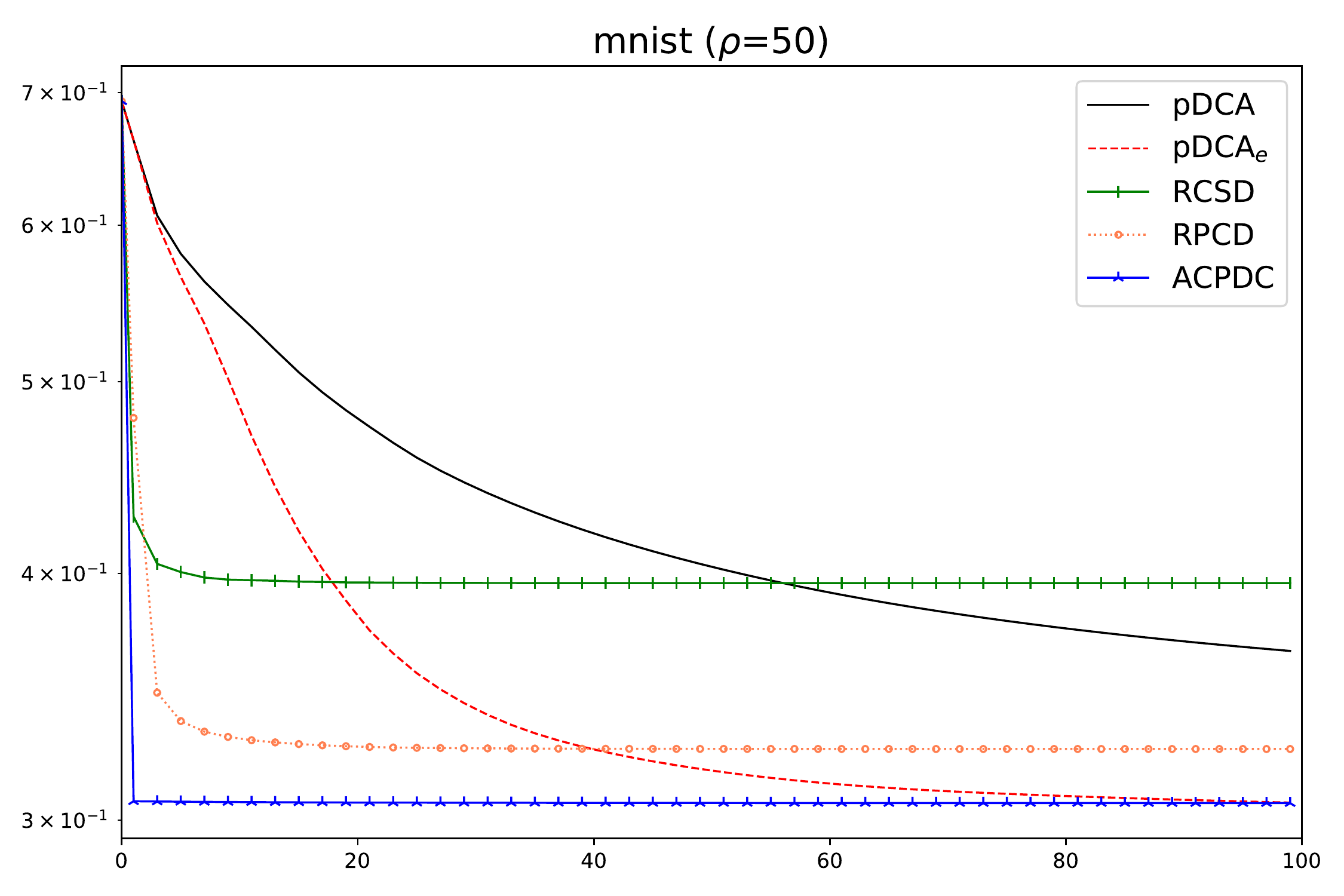}
\par\end{centering}
\begin{centering}
\includegraphics[scale=0.24]{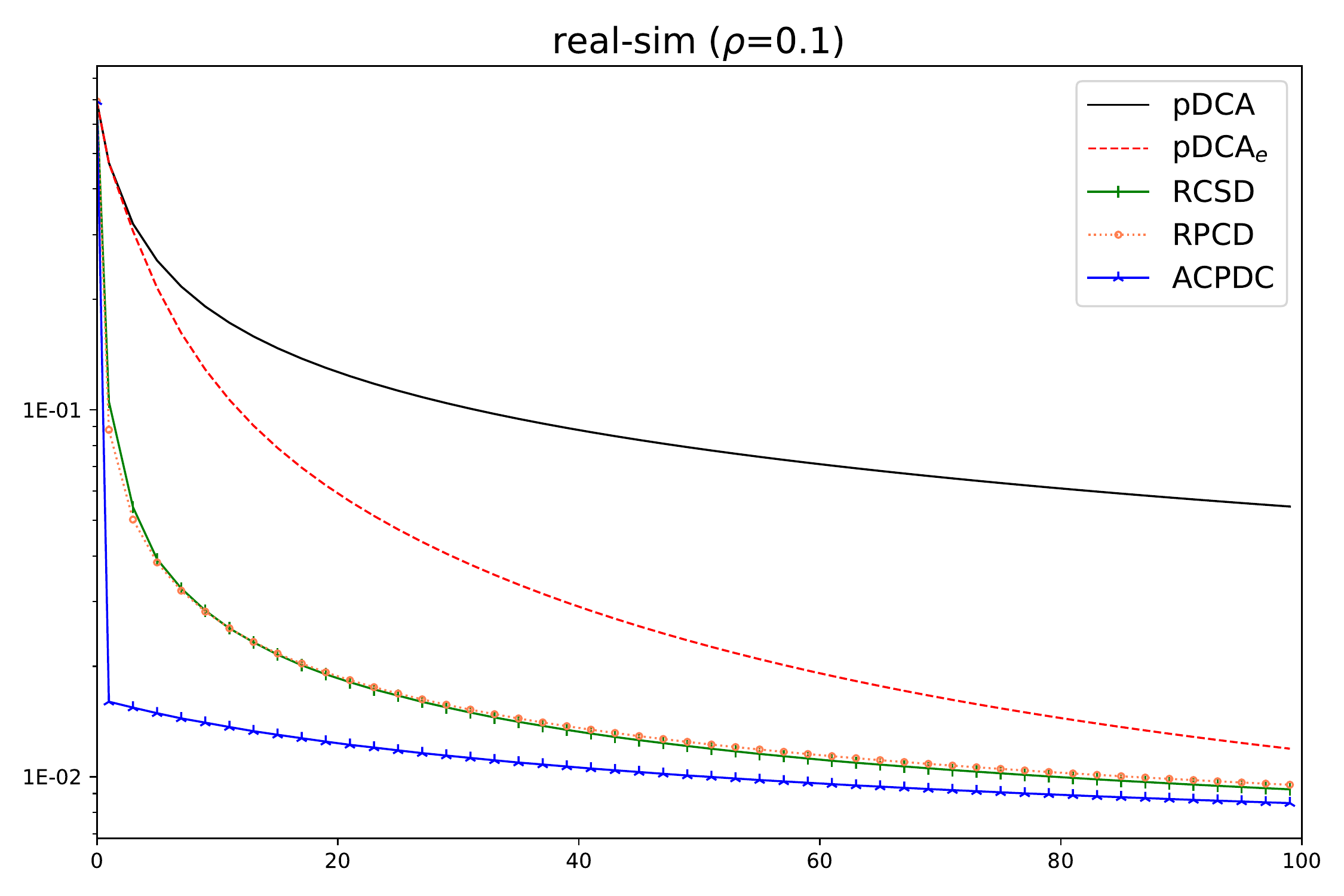}\includegraphics[scale=0.24]{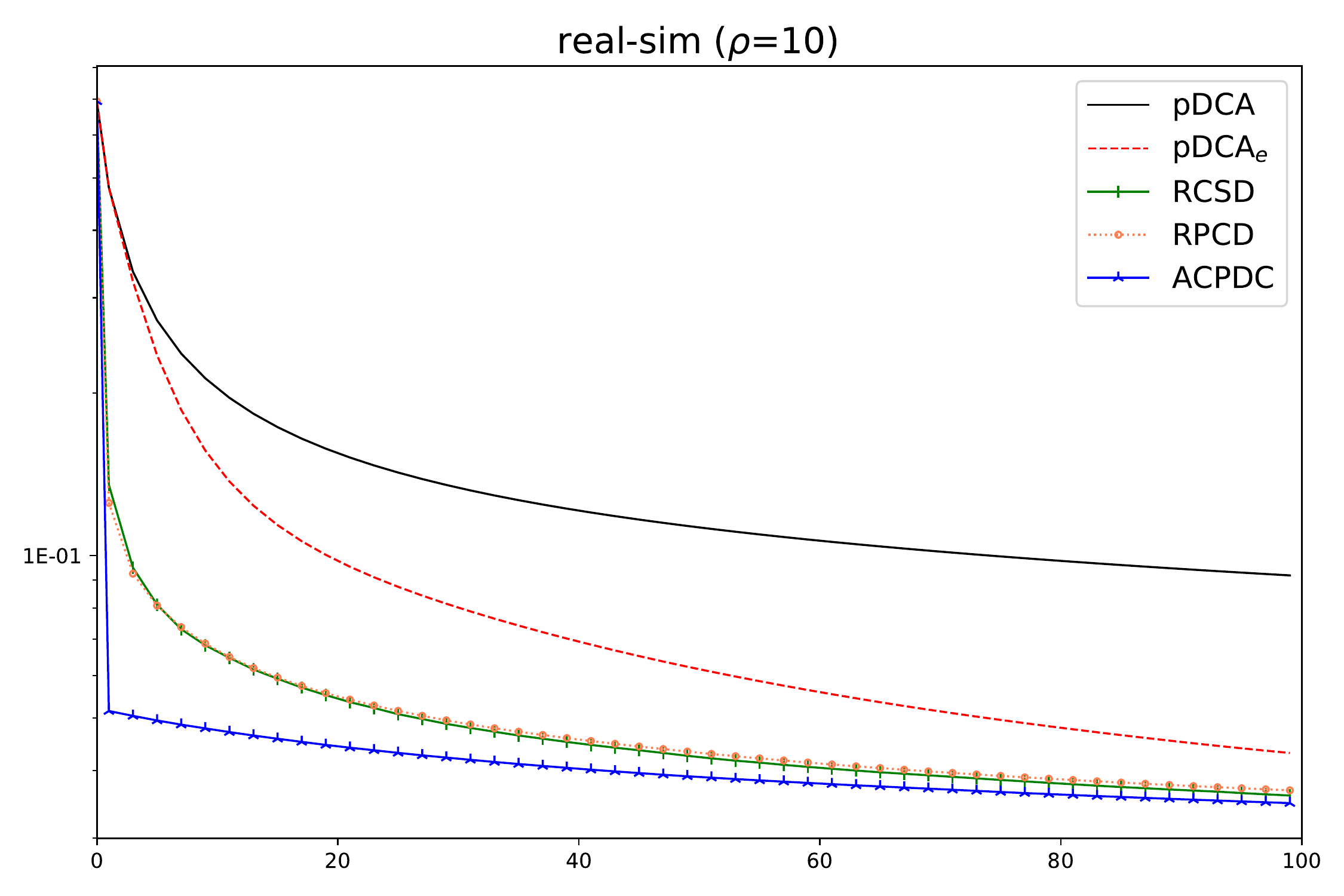}\includegraphics[scale=0.24]{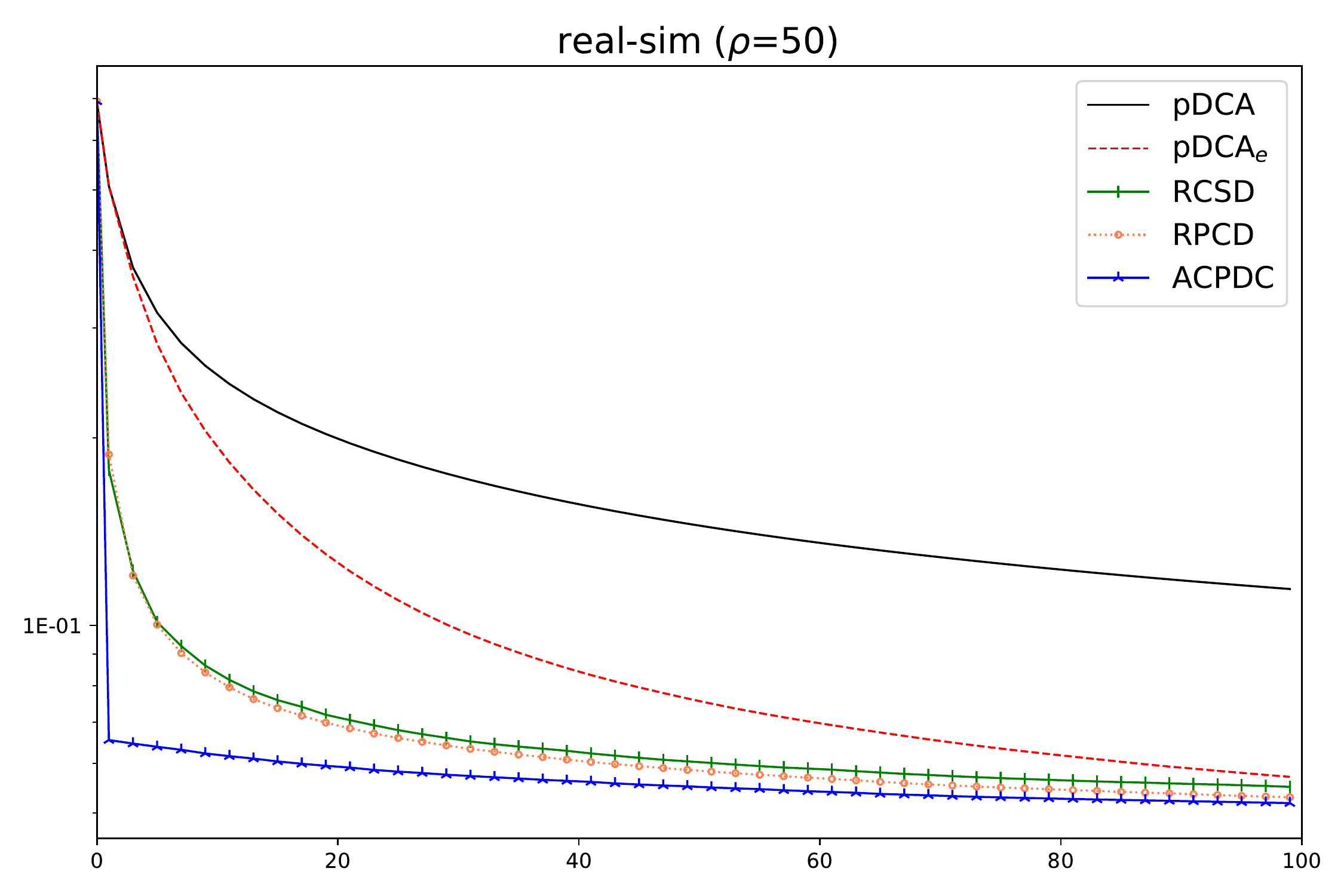}
\par\end{centering}
\begin{centering}
\includegraphics[scale=0.24]{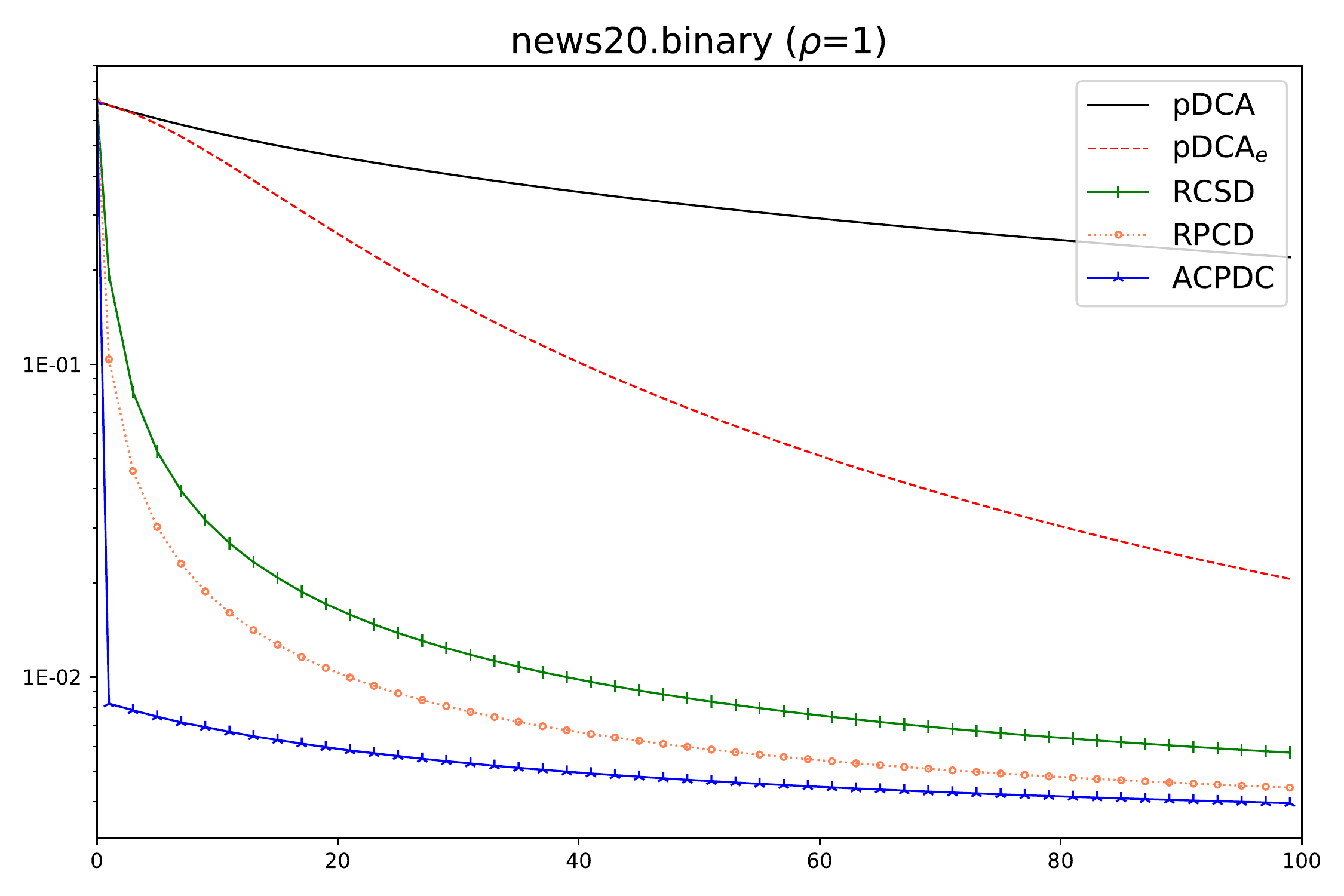}\includegraphics[scale=0.24]{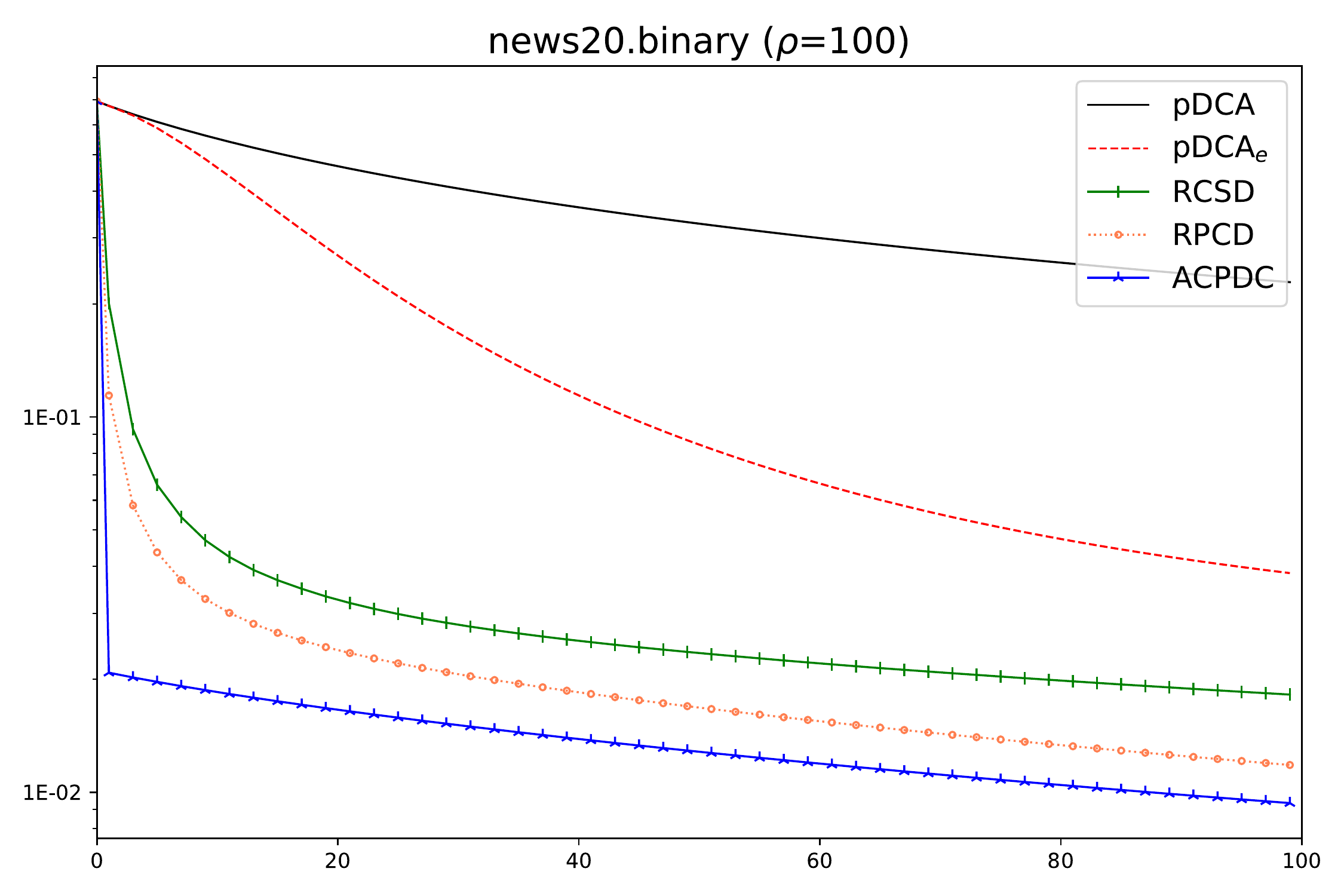}\includegraphics[scale=0.24]{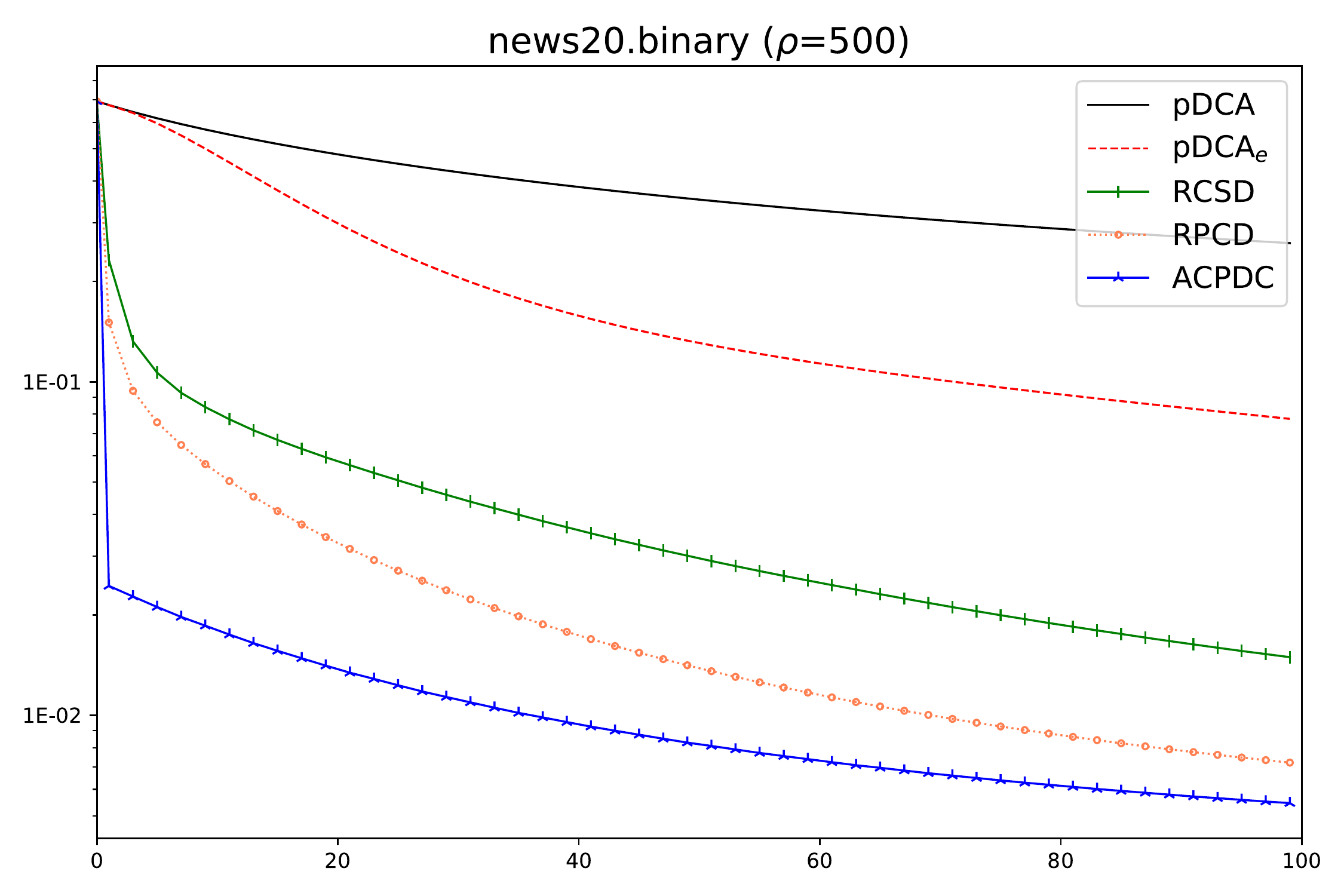}
\par\end{centering}
\begin{centering}
\includegraphics[scale=0.24]{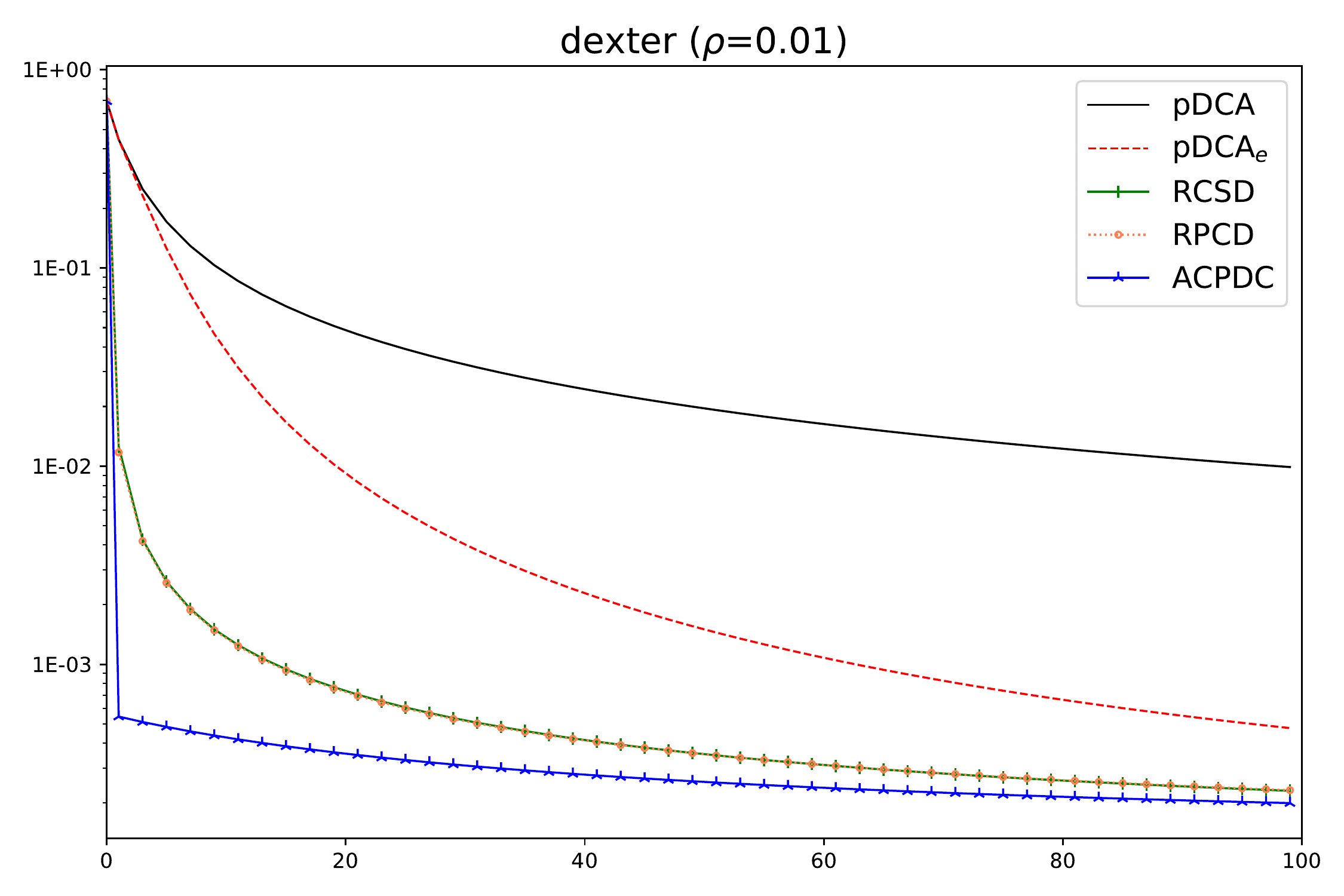}\includegraphics[scale=0.24]{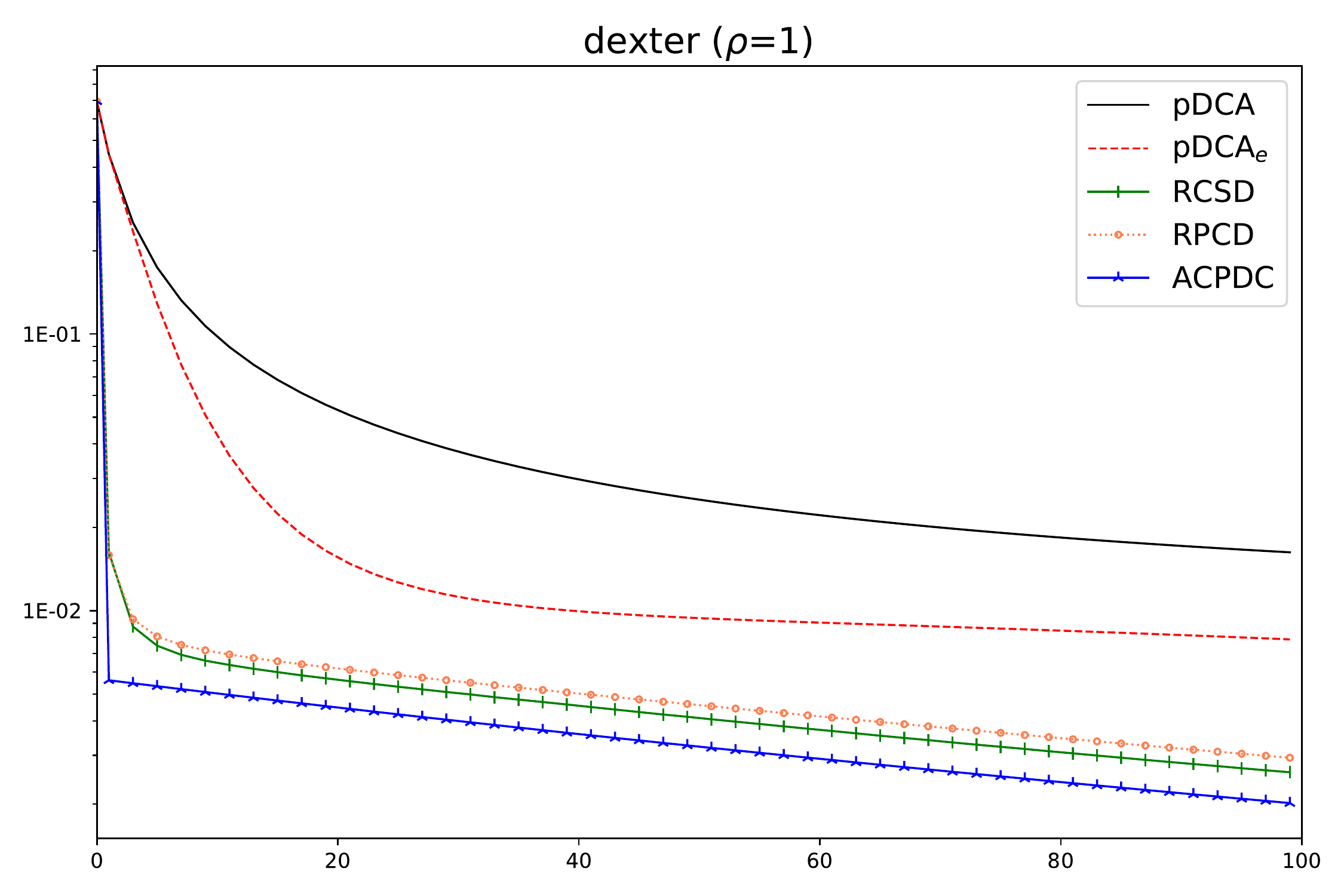}\includegraphics[scale=0.24]{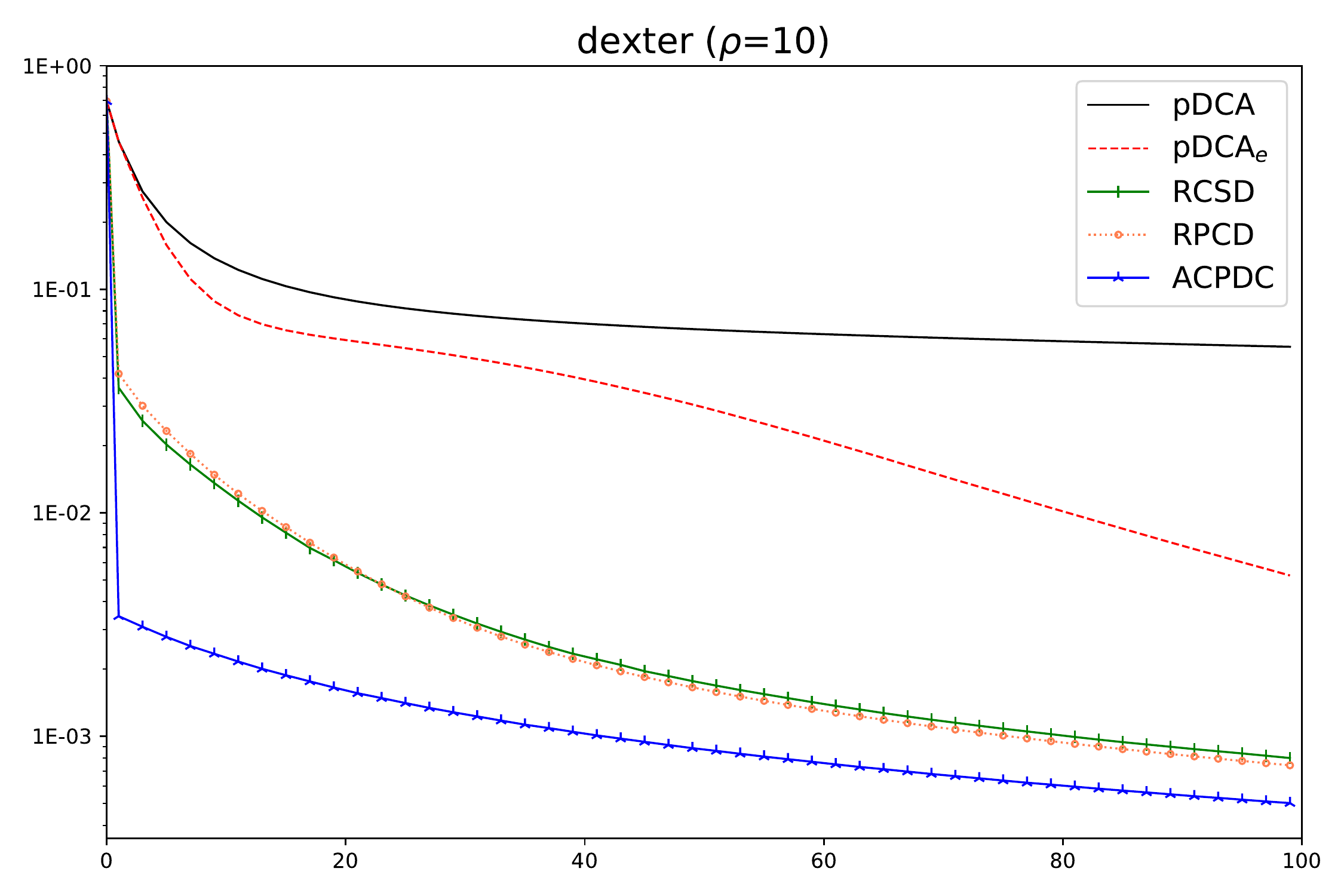}
\par\end{centering}
\centering{}\caption{\label{fig:Exper-logistic}Experimental results on logistic loss classification
with largest-$k$ norm penalty. $y$-axis: objective value (log scaled).
$x$-axis: number of passes to the dataset. Test datasets (from top
to bottom): \texttt{mnist},\texttt{ resl-sim},\texttt{ news20.binary}
and \texttt{dexter.}}
\end{figure}

\paragraph{Smoothed $l_{1}$ loss + SCAD penalty}

Our next experiment targets the $l_{1}$ loss regression with SCAD
penalty:
\begin{equation}
\min_{x\in\Rbb^{d}}F(x)=\frac{1}{n}\tsum_{i=1}^{n}\bigl|b_{i}-a_{i}^{T}x\bigr|+\frac{\rho}{d}\psi_{\lambda,\theta}(x).\label{eq:l1-scad-1}
\end{equation}
Due to the nonsmooth, convex and inseparable part, Problem (\ref{eq:l1-scad-1})
does not exactly satisfy our assumption. Fortunately, by introducing
a small term $\delta$, we can approximate $l_{1}$ loss by Huber
loss $H_{\delta}(a)$:
\[
H_{\delta}(a)=\begin{cases}
\frac{a^{2}}{2\delta} & |a|\le\delta\\
|a|-\tfrac{\delta}{2} & \text{o.w.}
\end{cases},\quad\delta>0.
\]
Our problem of interest, thus, follows:
\begin{equation}
\min_{x\in\Rbb^{d}}F_{\delta}(x)=\frac{1}{n}\sum_{i=1}^{n}H_{\delta}(b_{i}-a_{i}^{T}x)+\frac{\rho}{d}\psi_{\lambda,\theta}(x),\label{eq:l1-scad}
\end{equation}
Although smooth approximation introduces an $O(\delta)$ error, it
allows us to minimize $F_{\delta}(x)$ fast by using the gradient
information. It is easy to verify that $f_{\delta}(x)=\frac{1}{n}\sum_{i=1}^{n}H_{\delta}(b_{i}-a_{i}^{T}x)$
is block-wise Lipschitz smooth with $\frac{\|Ae_{i}\|^{2}}{n\delta}.$
In view of the Lipschitz smoothness of $f_{\delta}(x)$ and the weak
convexity of $\psi_{\lambda,\theta}(x)$, $F_{\delta}(x)$ has a condition
number of $O(\frac{\rho\,(\theta-1)}{d\,\delta})$. To set the parameters,
we choose $\delta$ in the range $\{10^{-2},10^{-3}\}$. Clearly,
$F_{\delta}$ is more difficult to minimize when $\delta$ is relatively
small.

\begin{figure}[h]
\begin{centering}
\includegraphics[scale=0.19]{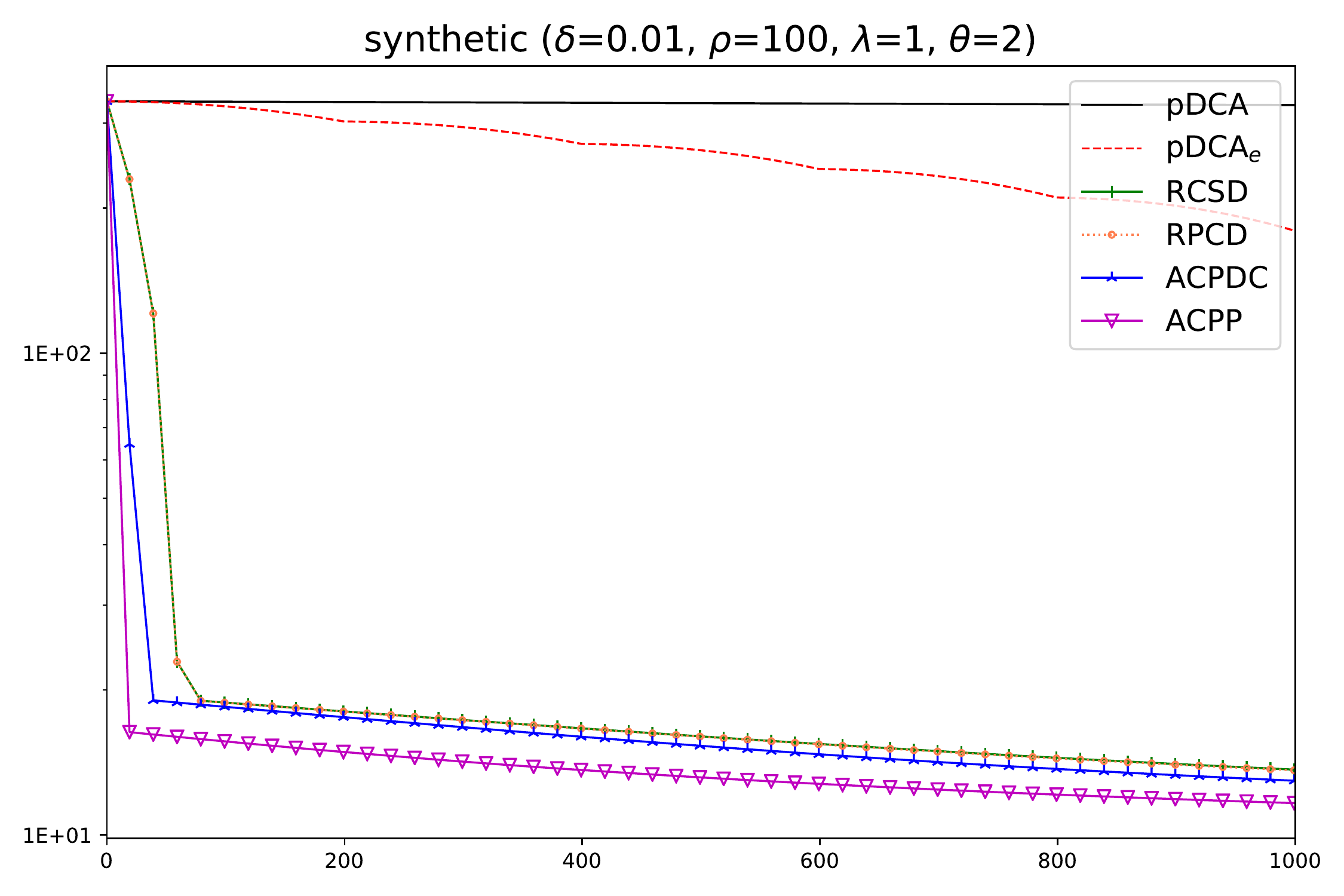}\includegraphics[scale=0.19]{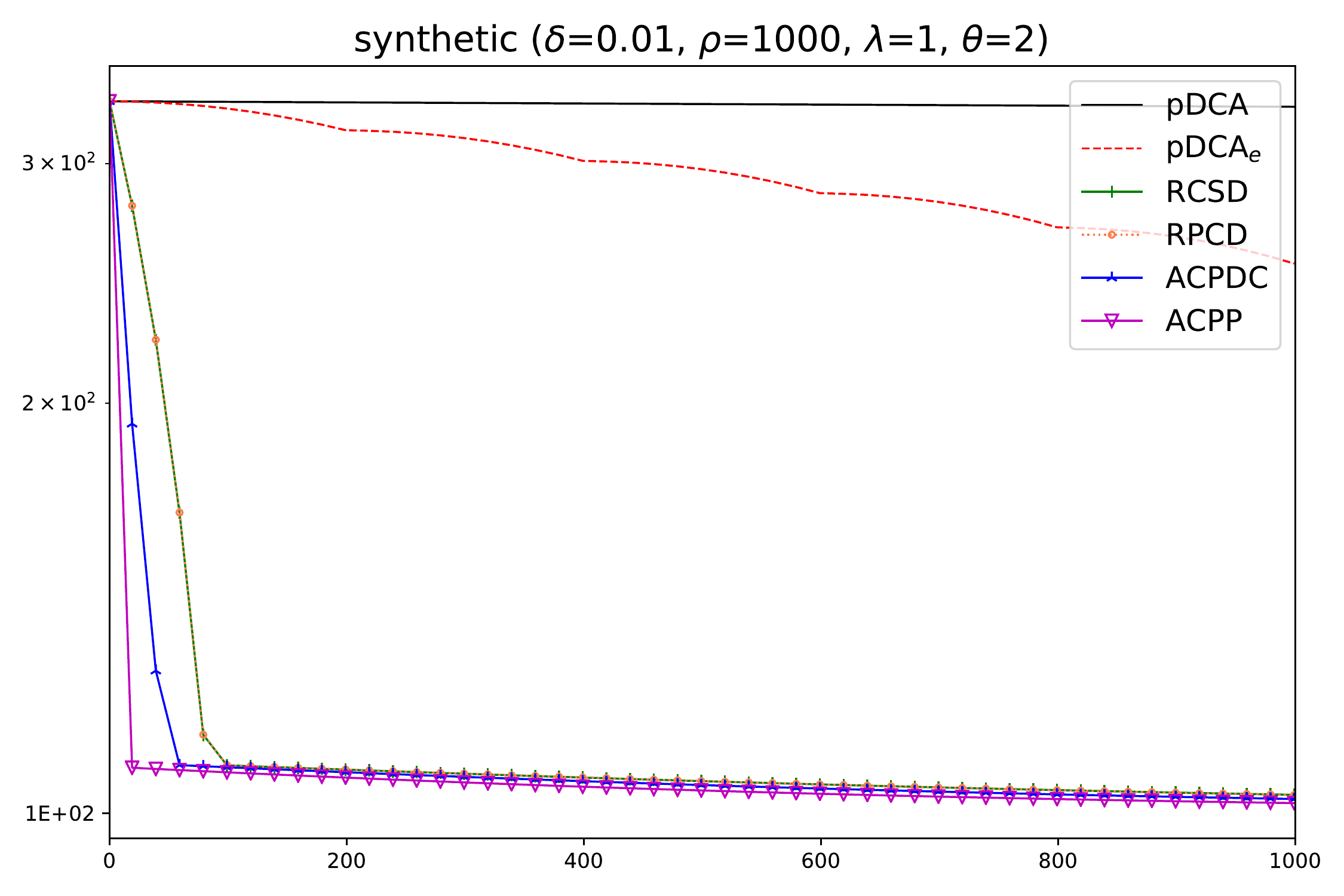}\includegraphics[scale=0.19]{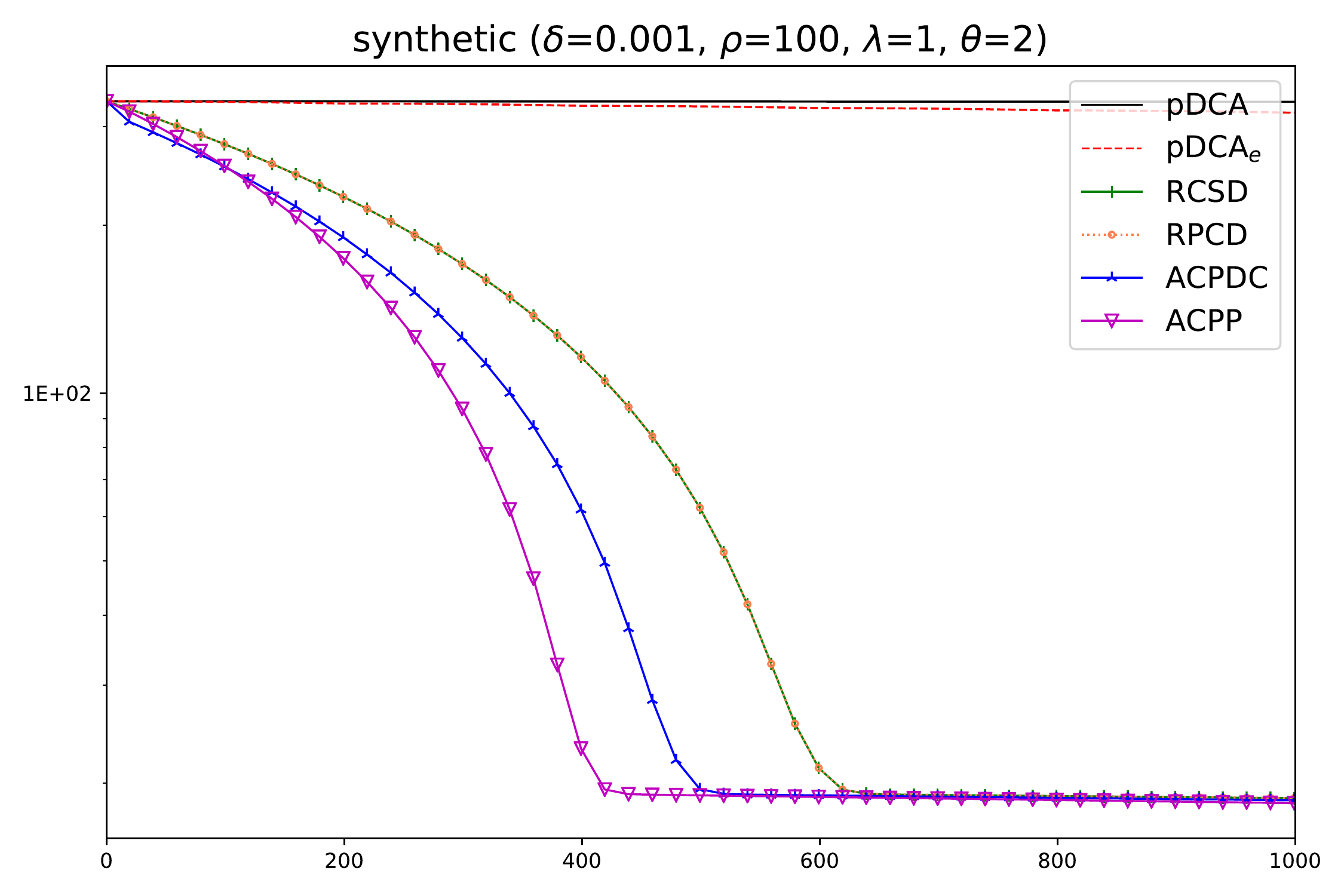}\includegraphics[scale=0.19]{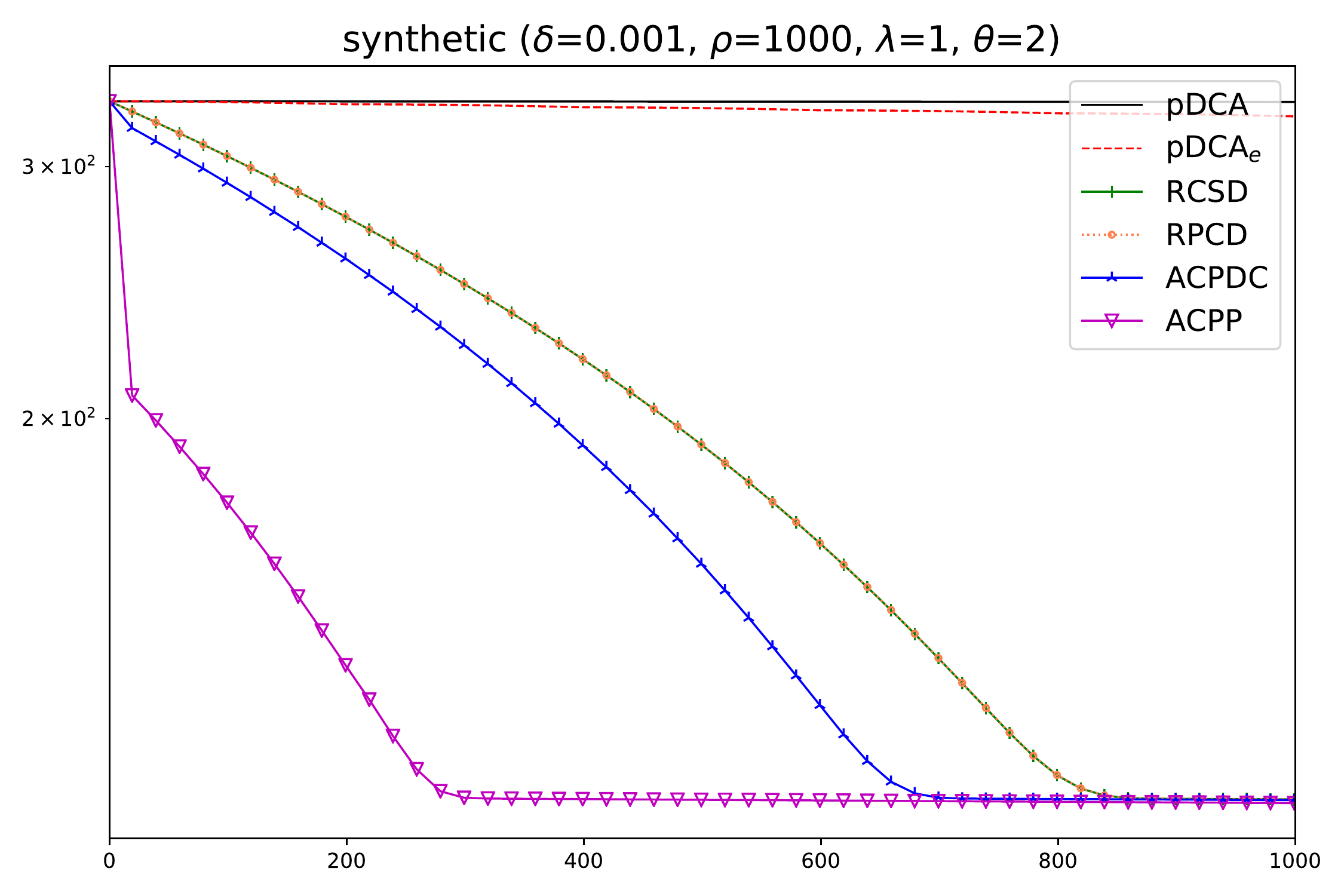}
\par\end{centering}
\begin{centering}
\includegraphics[scale=0.19]{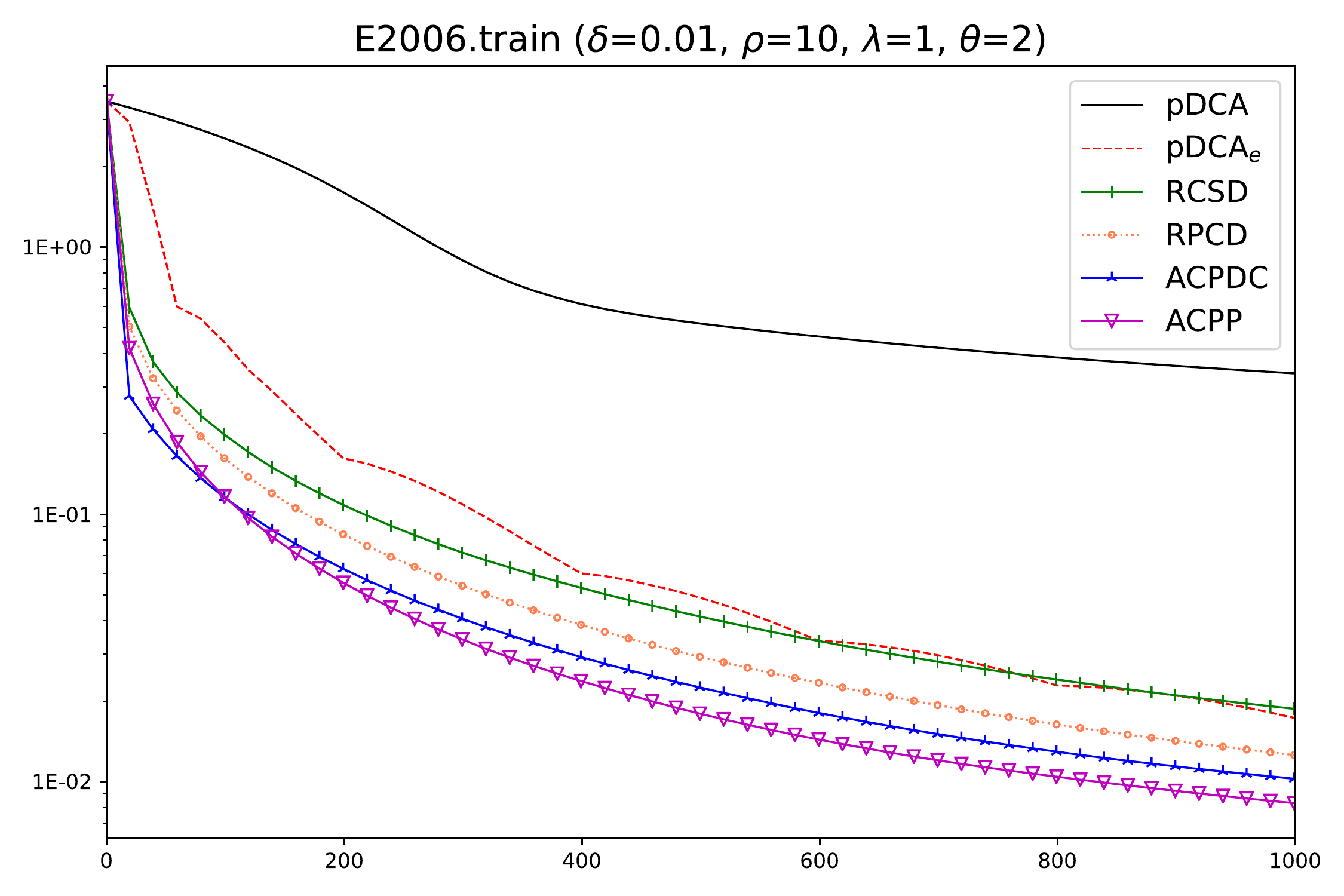}\includegraphics[scale=0.19]{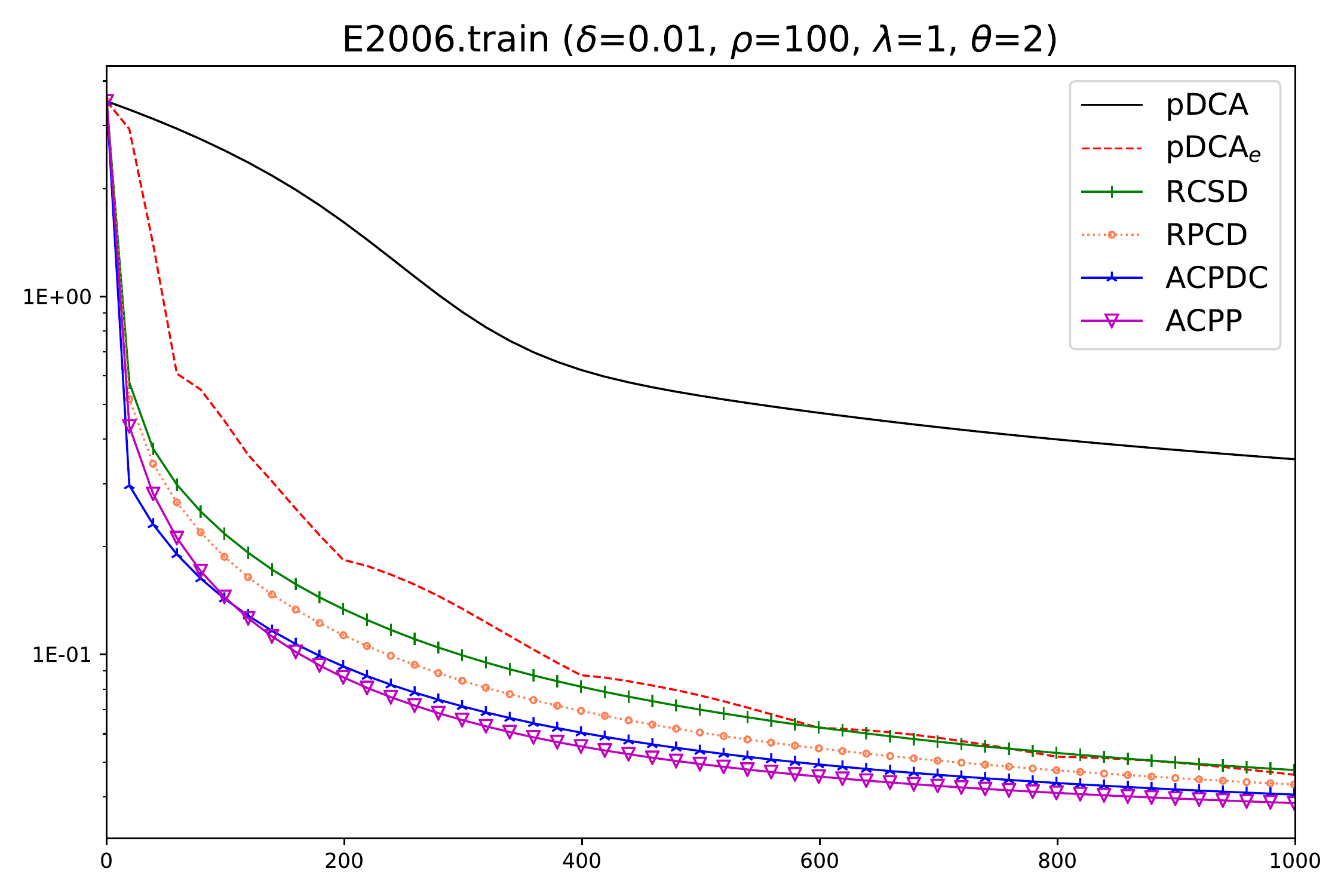}\includegraphics[scale=0.19]{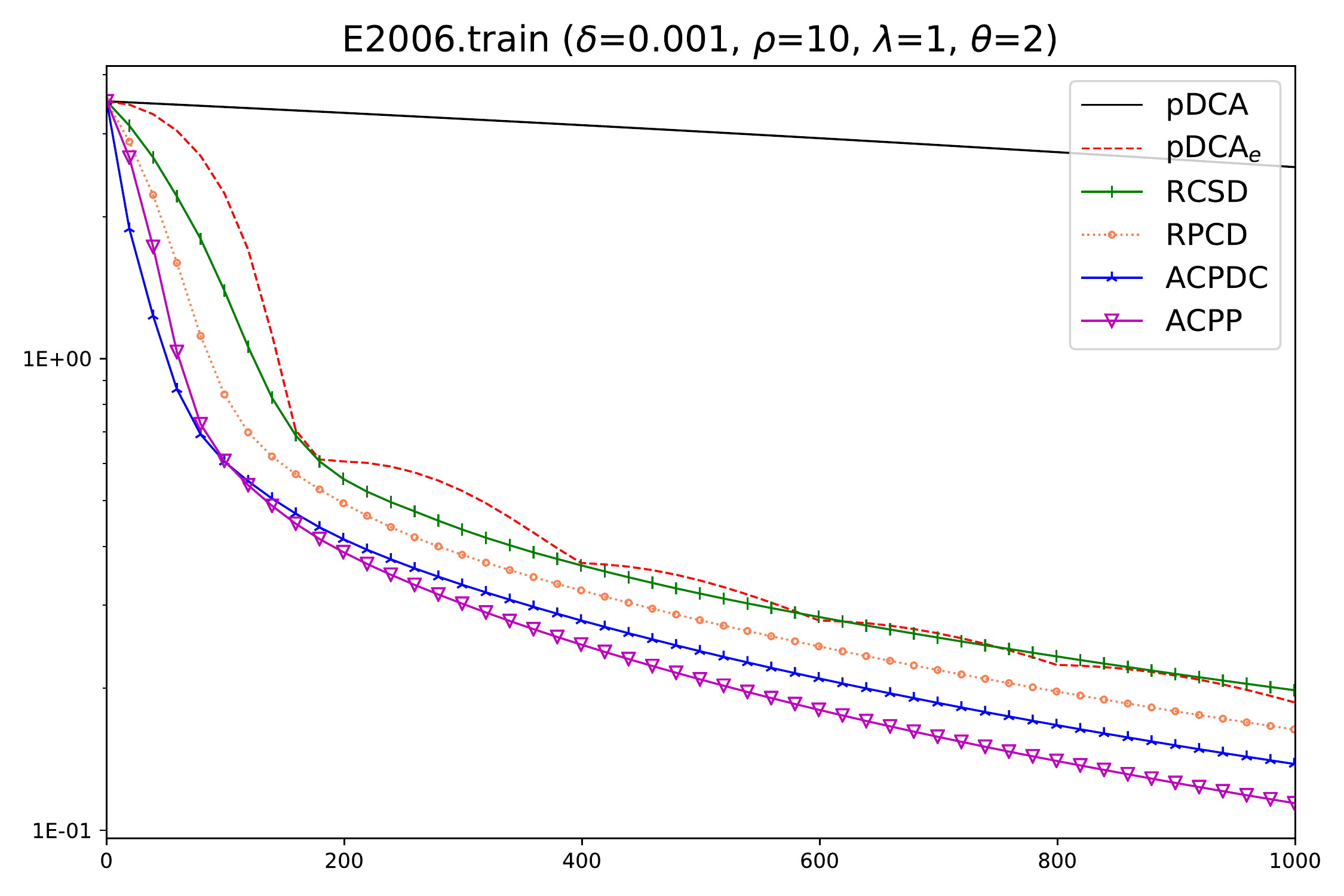}\includegraphics[scale=0.19]{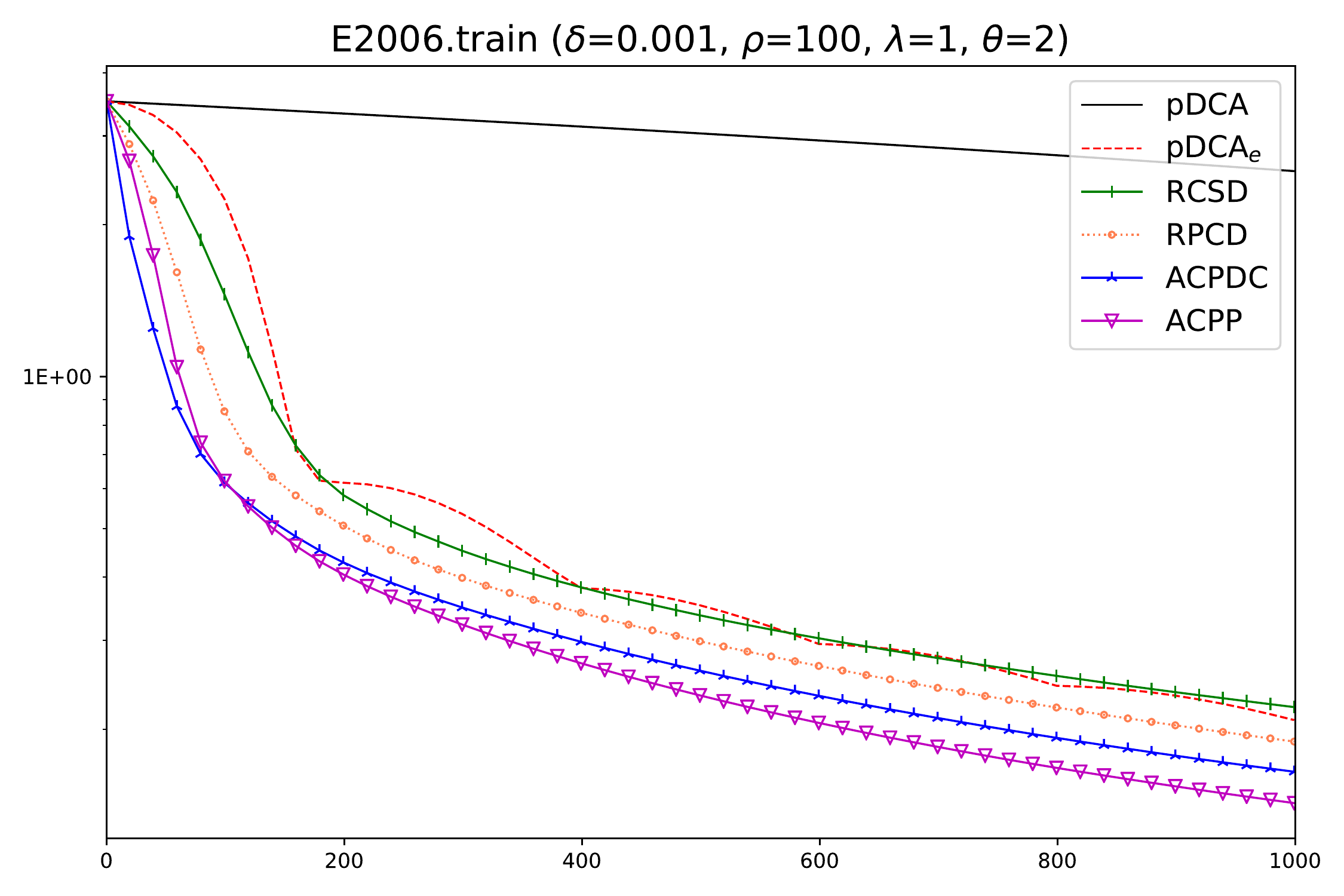}
\par\end{centering}
\begin{centering}
\includegraphics[scale=0.19]{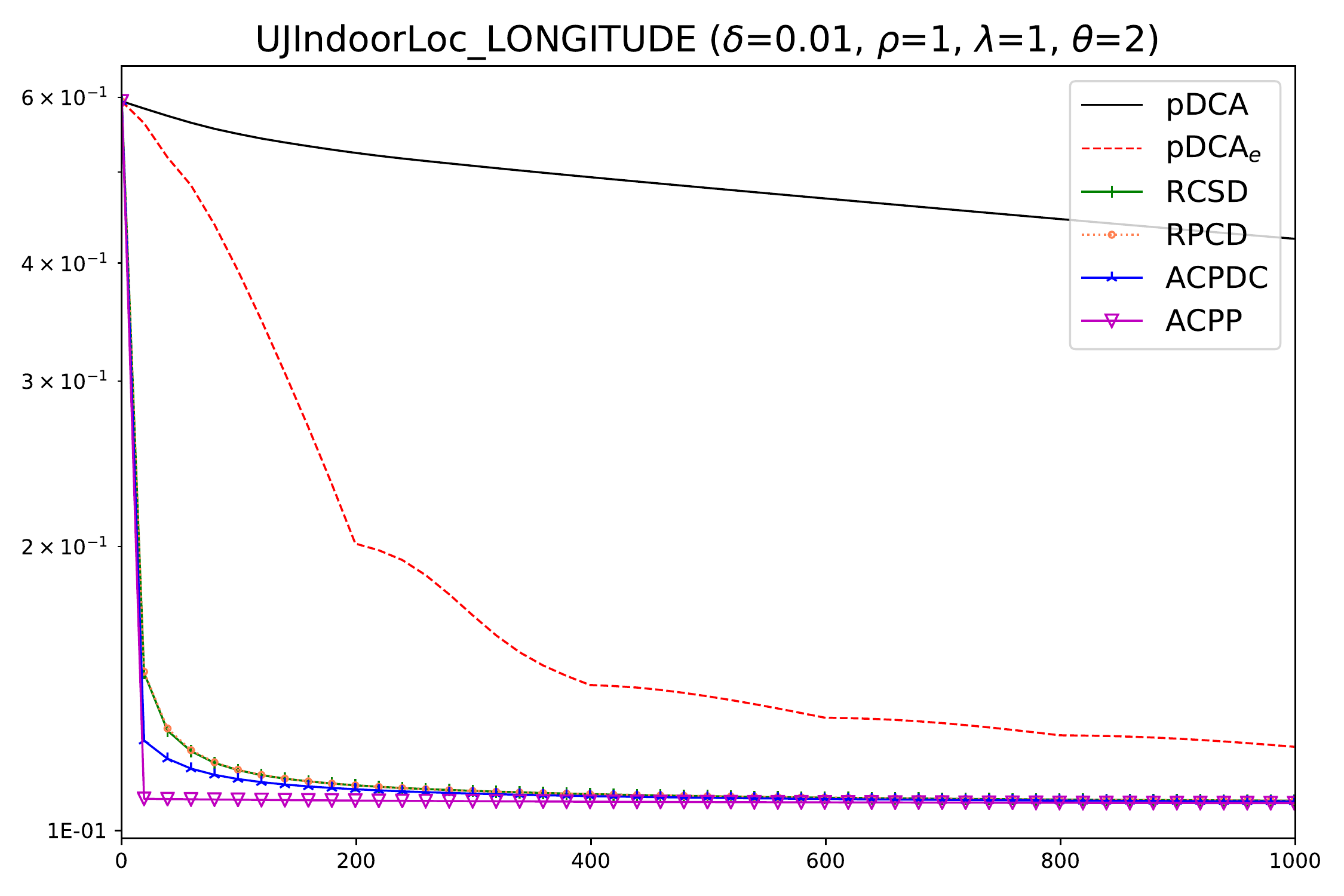}\includegraphics[scale=0.19]{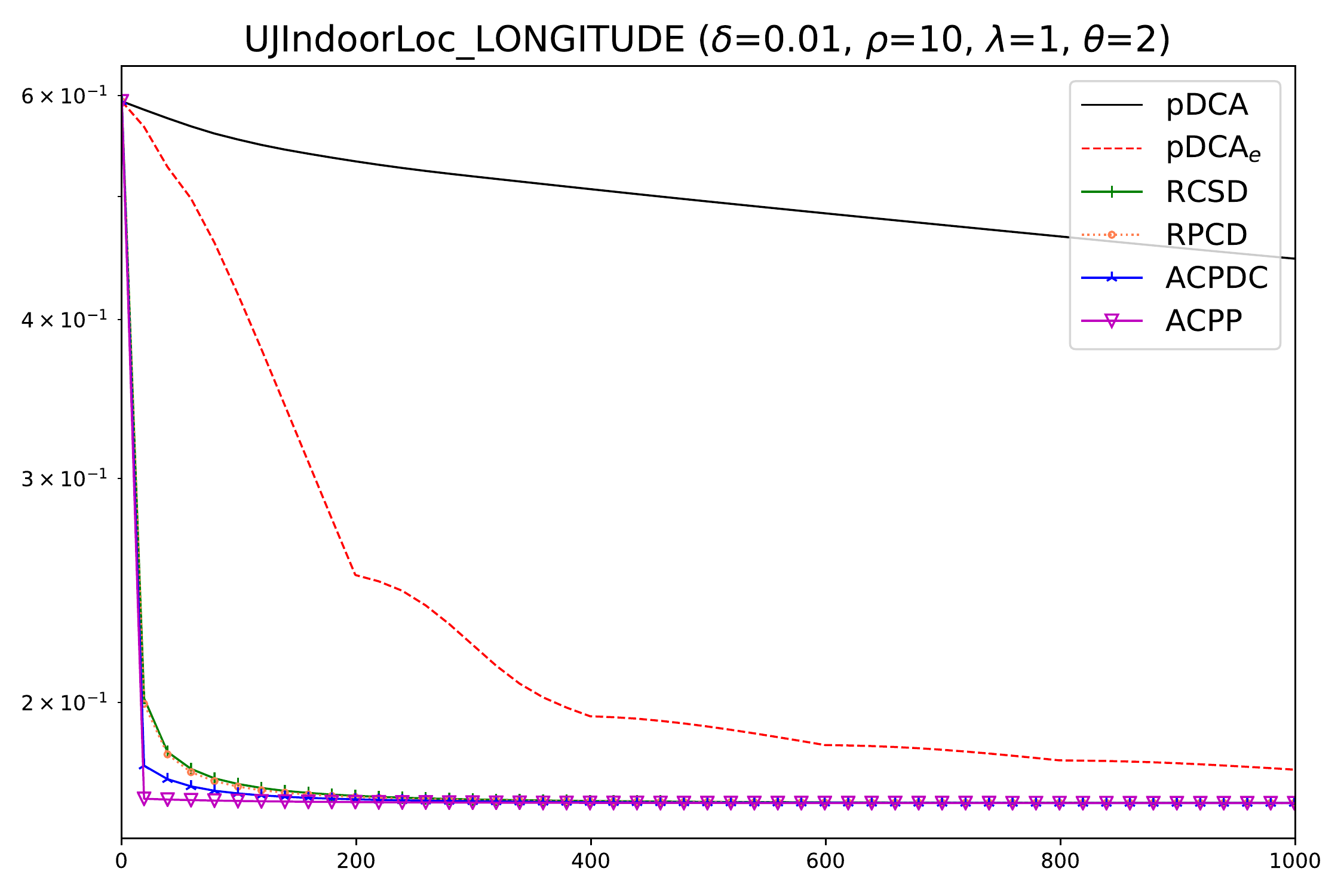}\includegraphics[scale=0.19]{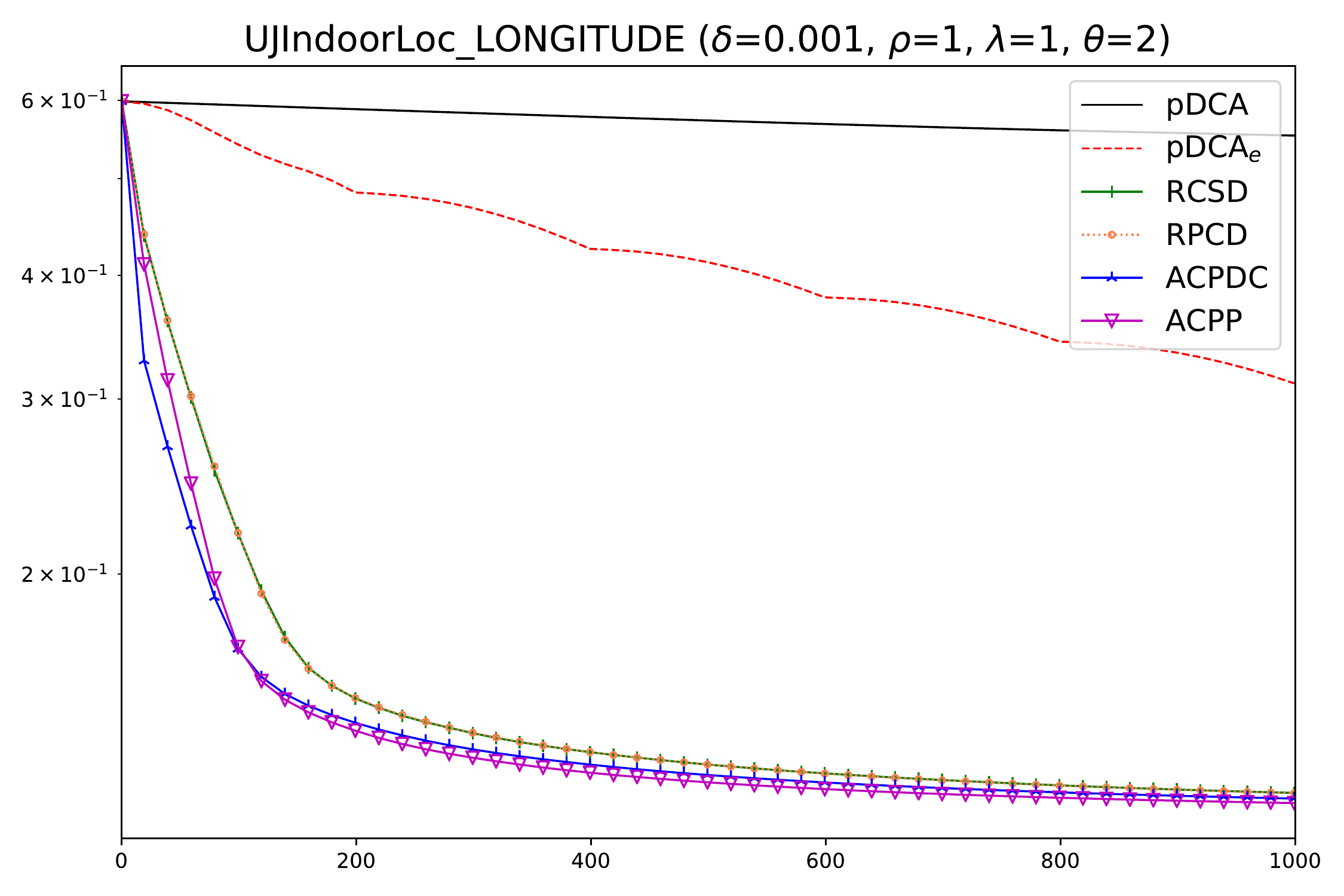}\includegraphics[scale=0.19]{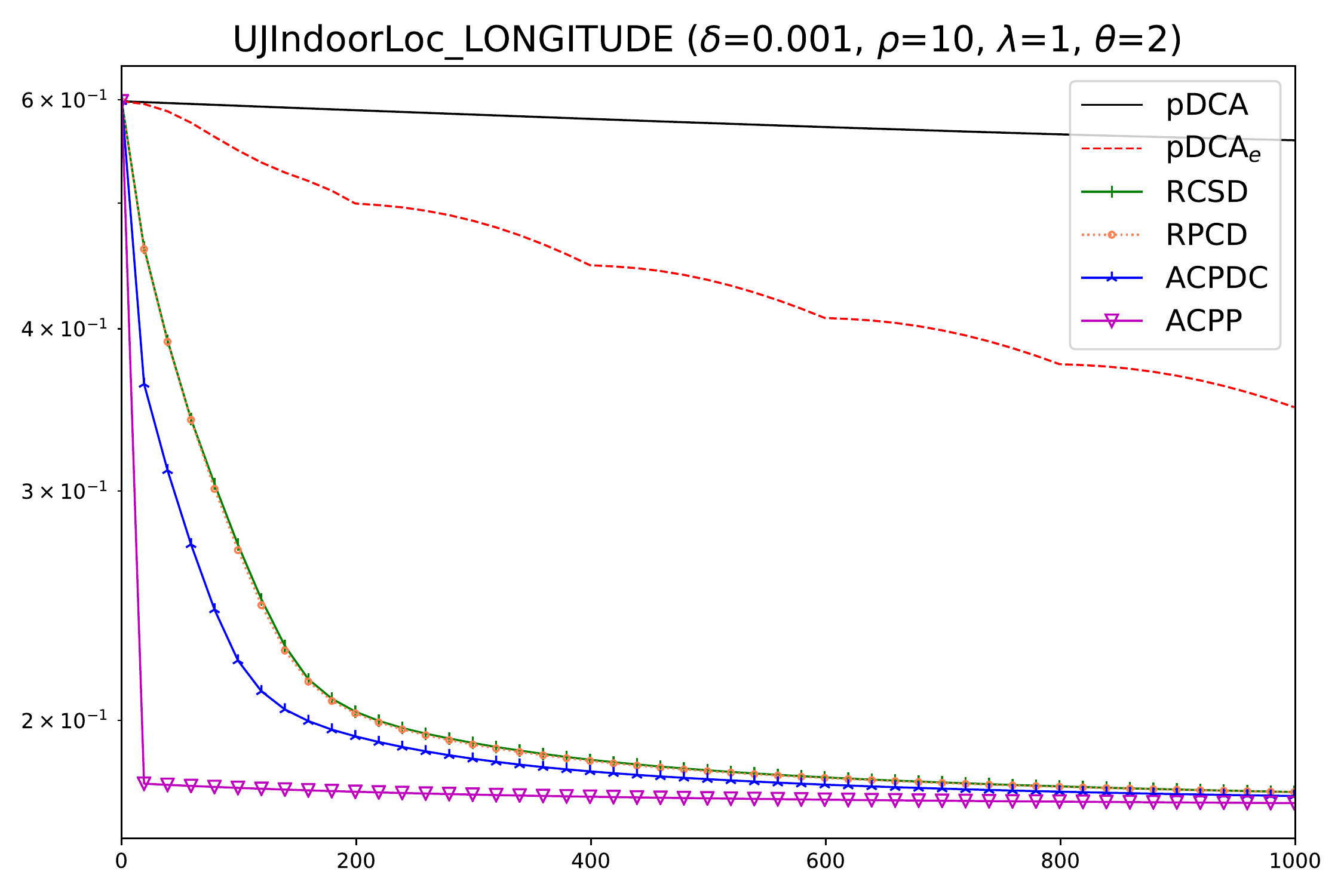}
\par\end{centering}
\begin{centering}
\includegraphics[scale=0.19]{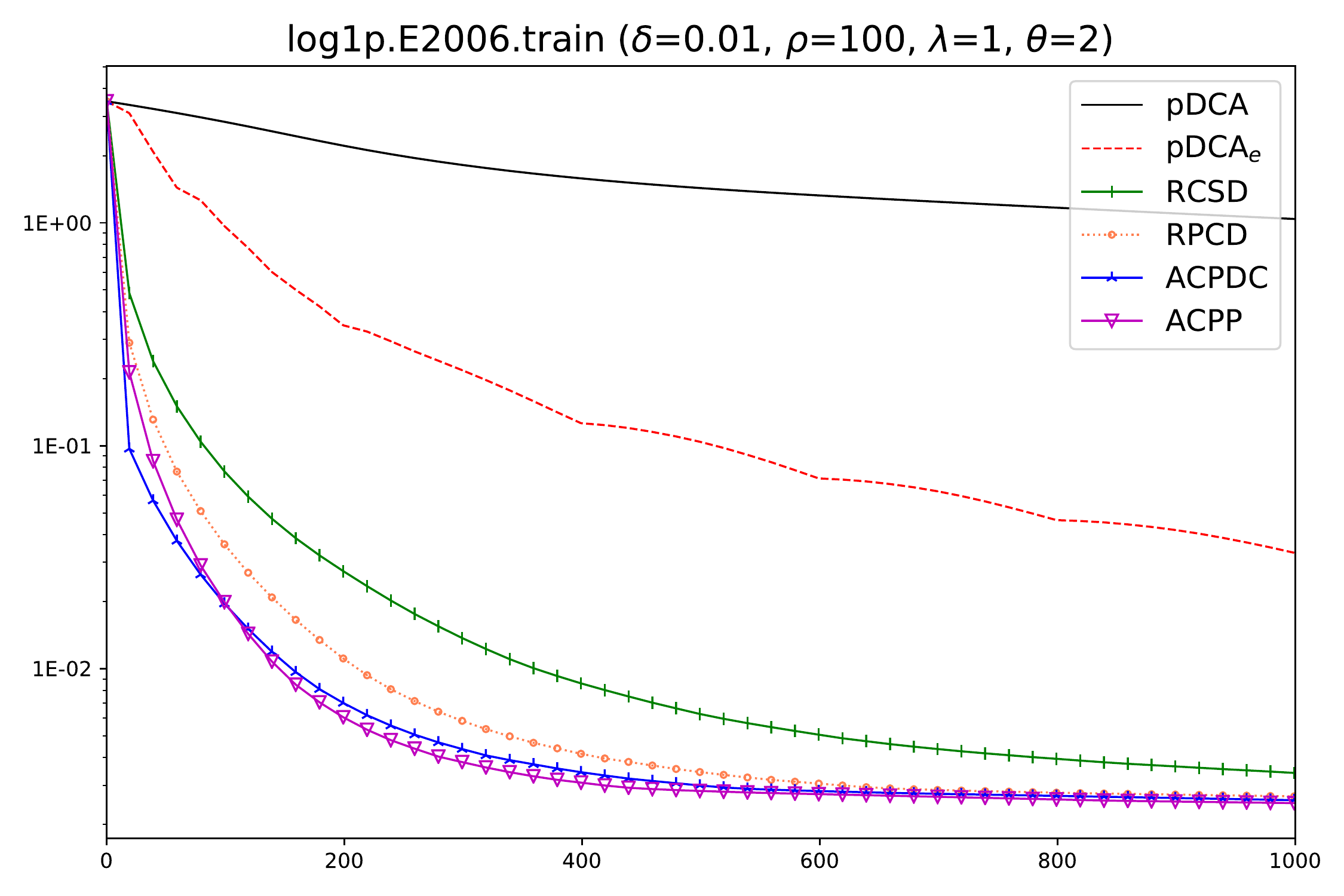}\includegraphics[scale=0.19]{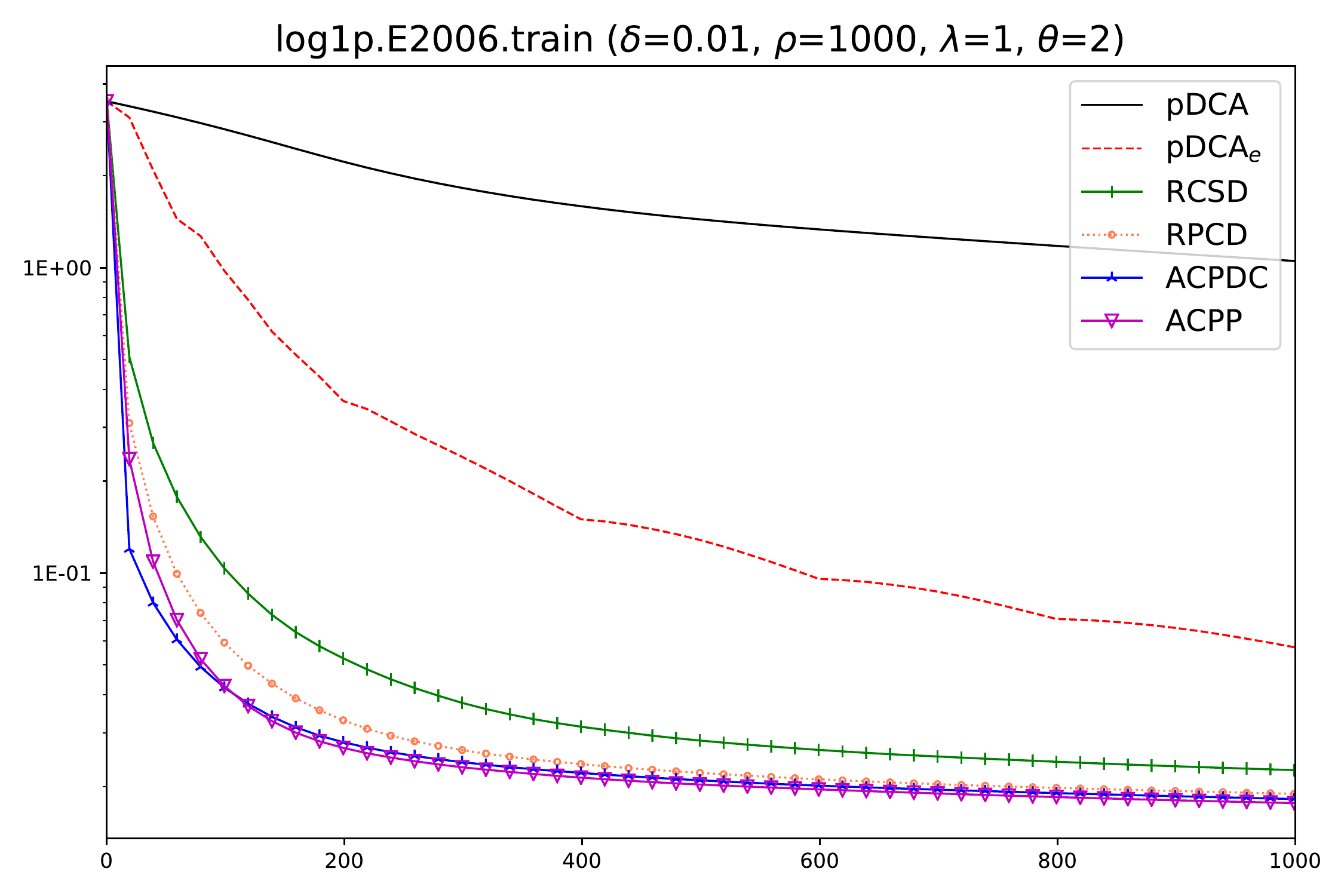}\includegraphics[scale=0.19]{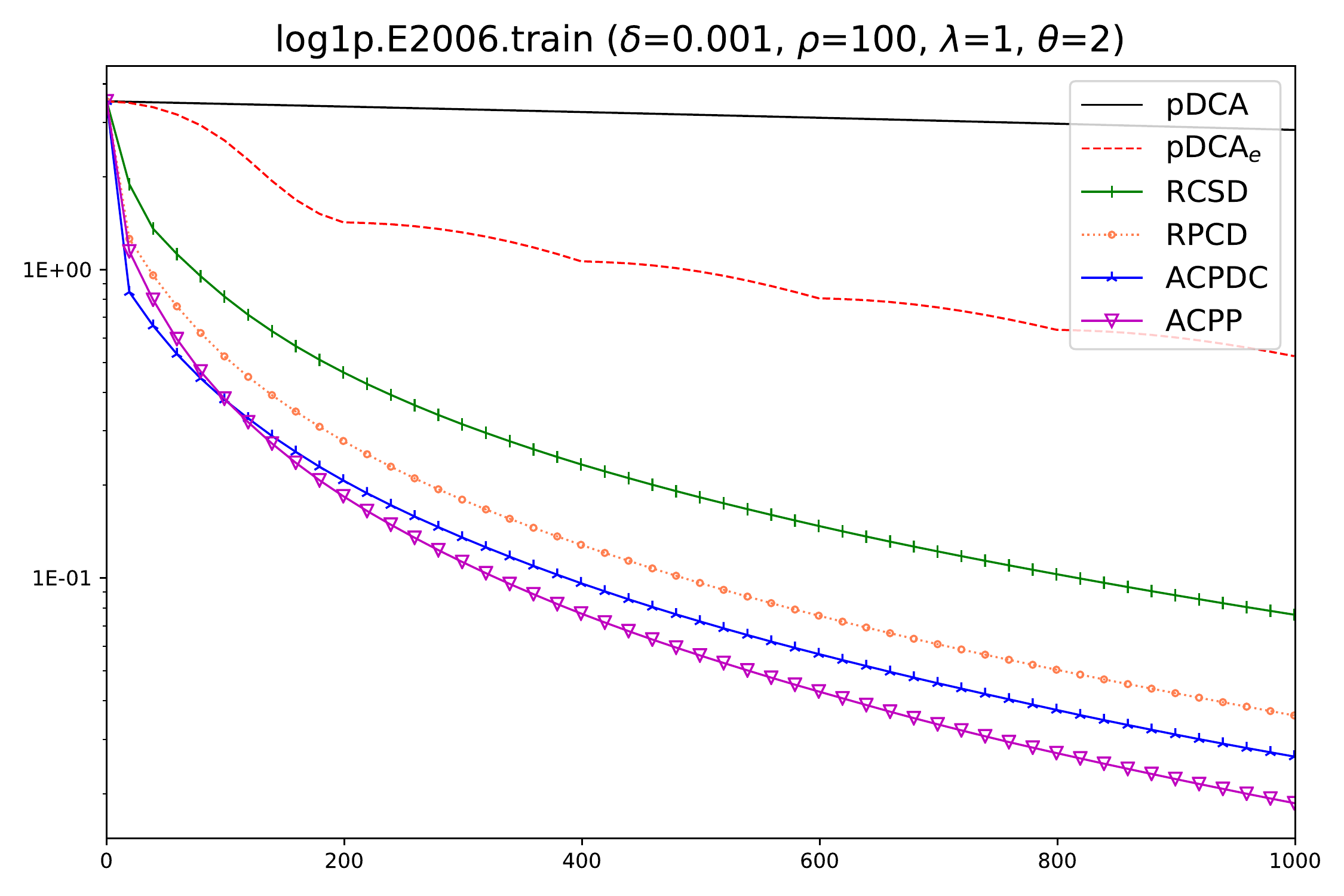}\includegraphics[scale=0.19]{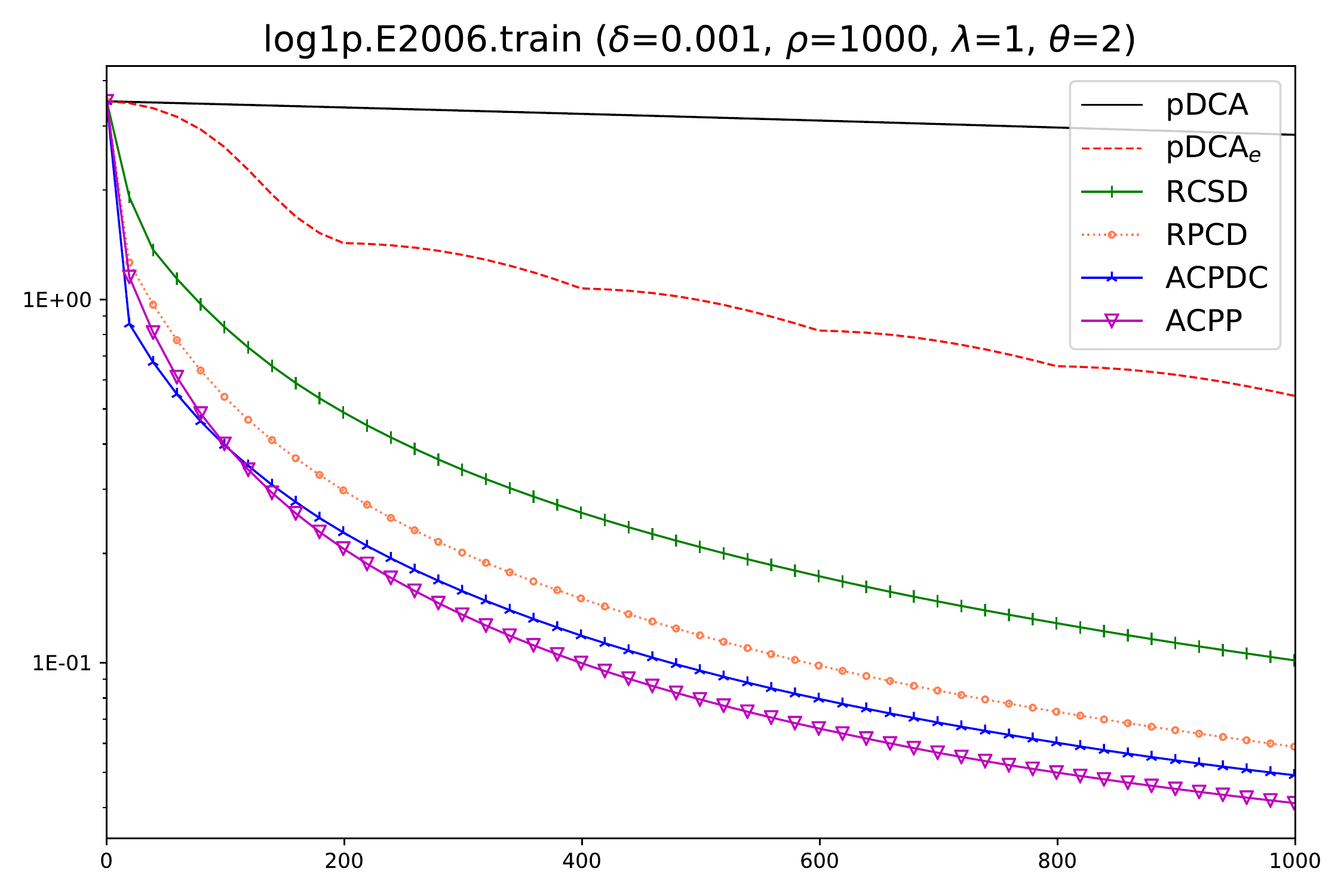}
\par
\end{centering}
\centering{}\caption{\label{fig:smooth-l1} Experimental results on smoothed $l_{1}$ regression
with SCAD penalty. $y$-axis: objective value (log scaled). $x$-axis:
number of passes to the dataset. Test datasets (from top to bottom): \texttt{synthetic},\texttt{
E2006-tfidf} and \texttt{UJIndoorLoc} and \texttt{E2006-log1p.}}
\end{figure}

We conduct the experiments on datasets \texttt{synthetic},
\texttt{E2006}, \texttt{UJIndoorLoc} and \texttt{log1p.E2006}.
In \texttt{UJIndoorLoc}, we consider  predicting the location
(longitude) of users inside of buildings. The dataset is preprocessed by
rescaling and shifting the longitude values to $[-1,1]$. We compare
all our proposed CD methods and the gradient methods pDCA and $\text{pDCA}_{e}$;
convergence performance is shown in Figure \ref{fig:smooth-l1}. When
the value of $\delta$ decreases, the approximation function $F_{\delta}$
is increasingly ill-conditioned, thereby being increasingly difficult
to optimize. Indeed, we observe that the convergence of all the tested
algorithms slows down when $\delta$ decreases from $10^{-2}$ to
$10^{-3}$. Meanwhile, we observe that, CD methods still perform consistently
better than gradient-based methods, and $\text{pDCA}_{e}$ performs
consistently better than pDCA. Moreover, we find that both ACPDC and
ACPP exhibit fast convergence and they both outperform RCSD and RPCD,
further confirming the advantage of using ACD in the nonconvex settings.

\section{Conclusion\label{sec:Conclusion}}

In this paper, we developed novel CD methods for minimizing a class
of nonsmooth and nonconvex functions. We developed randomized coordinate
subgradient descent (RCSD) and randomly permuted coordinate descent
(RPCD) methods, which naturally extend randomized coordinate descent
and cyclic coordinate descent, respectively, to the nonsmooth and
nonconvex settings, and establish their asymptotic convergence to
critical points and novel complexity results. We also developed a
new randomized proximal DC algorithm (ACPDC) for composite DC problems
and a new randomized proximal point algorithm (ACPP) for weakly convex
problems, based on the fast convergence of ACD for convex programming.
We developed new optimality measures and established iteration complexities
for the proposed algorithms. Both theoretical and experimental results
demonstrate the advantage of our proposed CD approaches over state-of-the-art
gradient-based methods.

\begin{comment}
While different optimality measures have been recently proposed (see
for example: \cite{xu2018stochastic,lan2018accelerated,drusvyatskiy2018efficiency,saeed-lan-nonconvex-2013}),
much work is needed to understand their relations. It would also be
interesting to study global convergence of coordinate descent under
additional assumptions, such as Kurdyka-Łojasiewivz condition. Another
direction is to study parallel and distributed CD methods for large
scale nonconvex optimization.
\end{comment}

%\pagebreak{}

\appendix
\appendixpage 

\section*{Convergence of ACD for convex smooth optimization\label{sec:nonun-acd}}

In this section, we propose Algorithm \ref{alg:ACD-nuni}, a variant
of ACD method with non-uniform sampling, for unconstrained smooth
optimization.
\begin{theorem}
In Algorithm \ref{alg:ACD-nuni}, choose the probability $p_{i}={L_{i}^{(1-s)/2}}/{T_{(1-s)/2}}$, $i\in[m]$. Define $\Gamma_{k}=\prod_{i=0}^{k}(1-\alpha_{i})^{-1}$, and assume that $\beta_{k}$, $\alpha_{k}$, $\gamma$ satisfy:
\begin{align}
\mu_{s} & \ge\gamma\label{eq:recur-gamma}\\
\beta_{k} & \ge\alpha_{k}T_{(1-s)/2}^{2},\quad k=0,1,2,...,\label{eq:recur-beta-alpha}\\
\Gamma_{k}\alpha_{k}(\beta_{k}+\gamma) & \ge\Gamma_{k+1}\alpha_{k+1}\beta_{k+1}\quad k=0,1,2,....
\end{align}
Then we have 
\[
\Gamma_{K-1}\Ebb[f(x^{K})-f(x^{*})]
%+\tfrac{\Gamma_{K-1}\alpha_{K-1}(\beta_{K-1}+\gamma)}{2}\Ebb\|x-z^{K}\|_{[s]}^{2}
\le f(x^{0})-f(x^{*})+\tfrac{\Gamma_{0}\alpha_{0}\beta_{0}}{2}\|x^{*}-x^{0}\|_{[s]}^{2}.
\]
In particular, if we choose $\alpha_{k}=\tfrac{\sqrt{\mu_{s}}}{\sqrt{\mu_{s}}+T_{(1-s)/2}}$
and $\beta_{k}=\sqrt{\mu_{s}}T_{(1-s)/2}$ and $\gamma=\mu_{s}$,
then we have 
\[
\Ebb[f(x^{K})-f(x^{*})]\le\left(1-\tfrac{\sqrt{\mu_{s}}}{\sqrt{\mu_{s}}+T_{(1-s)/2}}\right)^{K}\left[f(x^{0})-f(x^{*})+\tfrac{\mu_{s}}{2}\|x-x^{0}\|_{[s]}^{2}\right].
\]
\end{theorem}
\begin{proof}
We successively estimate the bound of $f(\bar{x}^{k+1})$ by 
\begin{align}
f(\bar{x}^{k+1}) & \le f(y^{k})+\left\langle \nabla f(y^{k}),\bar{x}^{k+1}-y^{k}\right\rangle +\tfrac{L_{i_{k}}}{2}\|\bar{x}^{k+1}-y^{k}\|^{2}\nonumber \\
 & =(1-\alpha_{k})[f(y^{k})+\langle \nabla f(y^{k}),x^{k}-y^{k}\rangle ]+\tfrac{L_{i_{k}}}{2}\|\bar{x}^{k+1}-y^{k}\|^{2}+\alpha_{k}f(y^{k})\nonumber \\
 & \quad+\alpha_{k}\langle\nabla f(y^{k}),\tfrac{1}{p_{i_{k}}}\Ubf_{\ik}z_{\ik}^{k+1}-y^{k}\rangle+\alpha_{k}\langle\nabla f(y^{k}),z^{k}-\tfrac{1}{p_{i_{k}}}\Ubf_{\ik}z_{\ik}^{k}\rangle\nonumber \\
 & \le(1-\alpha_{k})f(x^{k})+\tfrac{L_{i_{k}}\alpha_{k}^{2}}{2p_{i_{k}}^{2}}\|z_{\ik}^{k}-z_{\ik}^{k+1}\|_{\ik}^{2}+\alpha_{k}f(x)-\tfrac{\alpha_{k}\mu_{s}}{2}\norm{x-y^{k}}_{[s]}^{2}\nonumber \\
 & \quad+\tfrac{\alpha_{k}}{p_{i_{k}}}\langle \nabla_{i_{k}}f(y^{k}),z_{\ik}^{k+1}-x_{\ik}\rangle +\alpha_{k}\langle\nabla f(y^{k}),\tfrac{1}{p_{i_{k}}}\Ubf_{\ik}x_{\ik}-x\rangle\nonumber \\
 & \quad+\alpha_{k}\langle\nabla f(y^{k}),z^{k}-\tfrac{1}{p_{i_{k}}}\Ubf_{\ik}z_{\ik}^{k}\rangle,\label{eq:fk-bound}
\end{align}
where the first equality uses the following identity
\begin{equation*}
\bar{x}^{k+1}=y^{k}+\tfrac{\alpha_{k}}{p_{i}}\Ubf_{\ik}(z_{\ik}^{k+1}-z_{\ik}^{k})=(1-\alpha_{k})x^{k}+\tfrac{\alpha_{k}}{p_{i}}\Ubf_{\ik}z_{\ik}^{k+1}+(\alpha_{k}z^{k}-\tfrac{\alpha_{k}}{p_{i}}\Ubf_{\ik}z_{\ik}^{k}),
\end{equation*}
and the last inequality uses the strong convexity:
\[
f(y^{k})+\left\langle \nabla f(y^{k}),x-y^{k}\right\rangle \le f(x)-\tfrac{\mu_{s}}{2}\|x-y^{k}\|_{[s]}^{2}.
\]
According to Lemma (\ref{lem:optimality-prox-map}), we have
\begin{align}
\tfrac{\alpha_{k}}{p_{i_{k}}}\left\langle \nabla_{i_{k}}f(y^{k}),z_{\ik}^{k+1}-x_{\ik}\right\rangle  & \le\tfrac{\alpha_{k}\gamma}{2}[\|x-y^{k}\|_{[s]}^{2}-\|x-z^{k+1}\|_{[s]}^{2}-\|y^{k}-z^{k+1}\|_{[s]}^{2}]\nonumber \\
 & \quad+\tfrac{\alpha_{k}\beta_{k}}{2}[\|x-z^{k}\|_{[s]}^{2}-\|x-z^{k+1}\|_{[s]}^{2}-\|z^{k}-z^{k+1}\|_{[s]}^{2}]\nonumber \\
 & \le\tfrac{\alpha_{k}\gamma}{2}\|x-y^{k}\|_{[s]}^{2}+\tfrac{\alpha_{k}\beta_{k}}{2}\|x-z^{k}\|_{[s]}^{2}\nonumber \\
 & \quad-\tfrac{\alpha_{k}(\beta_{k}+\gamma)}{2}\|x-z^{k+1}\|_{[s]}^{2}-\tfrac{\alpha_{k}\beta_{k}}{2}\|z^{k}-z^{k+1}\|_{[s]}^{2}\label{eq:linear-term-bound}
\end{align}
In view of (\ref{eq:recur-beta-alpha}), we have 
\[
\tfrac{L_{i_{k}}\alpha_{k}^{2}}{2p_{i_{k}}^{2}}\|z_{\ik}^{k}-z_{\ik}^{k+1}\|_{\ik}^{2}-\tfrac{\alpha_{k}\beta_{k}}{2}\|z^{k}-z^{k+1}\|_{[s]}^{2}\le\tfrac{\alpha_{k}^{2}T_{(1-s)/2}^{2}-\alpha_{k}\beta_{k}}{2}\|z^{k}-z^{k+1}\|_{[s]}^{2}\le0.
\]
From (\ref{eq:recur-gamma}), we have $\tfrac{\alpha_{k}\gamma}{2}[\|x-y^{k}\|_{[s]}^{2}\le\tfrac{\alpha_{k}\mu_{s}}{2}\norm{x-y^{k}}_{[s]}^{2}$. 

Now putting (\ref{eq:fk-bound}) and (\ref{eq:linear-term-bound})
together, we obtain
\begin{align}
f(\bar{x}^{k+1}) & \le(1-\alpha_{k})f(x^{k})+\alpha_{k}f(x)+\tfrac{\alpha_{k}\beta_{k}}{2}\|x-z^{k}\|_{[s]}^{2}-\tfrac{\alpha_{k}(\beta_{k}+\gamma)}{2}\|x-z^{k+1}\|_{[s]}^{2}\nonumber\\
 & \quad+\alpha_{k}\langle\nabla f(y^{k}),\tfrac{1}{p_{i_{k}}}\Ubf_{\ik}x_{\ik}-x\rangle+\alpha_{k}\langle\nabla f(y^{k}),z^{k}-\tfrac{1}{p_{i_{k}}}\Ubf_{\ik}z_{\ik}^{k}\rangle. \label{eq:smooth-mid-01}
\end{align}
We next take the expectation on both sides of \eqref{eq:smooth-mid-01} over $i_k$.
Notice the identity $\Ebb_{i_{k}}\langle\nabla f(y^{k}),\tfrac{1}{p_{i_{k}}}\Ubf_{\ik}x_{\ik}-x\rangle=0$
and $\Ebb_{i_{k}}\langle\nabla f(y^{k}),z^{k}-\tfrac{1}{p_{i_{k}}}\Ubf_{\ik}z_{\ik}^{k}\rangle=0$.
Moreover, notice that for option II of Algorithm \ref{alg:acd-pp-smooth}, we have $f(x^{k+1})+\tfrac{1}{2L_{i_{k}}}\|\nabla_{i_{k}}f(\bar{x}^{k+1})\|_{\ik}^{2}\le f(\bar{x}^{k+1})$,
hence we  always guarantee $f(x^{k+1})\le f(\bar{x}^{k+1})$. Putting all these pieces together, we have 
\[
\Ebb_{i_{k}}[f(x^{k+1})-f(x)]\le(1-\alpha_{k})[f(x^{k})-f(x)]+\tfrac{\alpha_{k}\beta_{k}}{2}\|x-z^{k}\|_{[s]}^{2}-\tfrac{\alpha_{k}(\beta_{k}+\gamma)}{2}\Ebb_{i_{k}}\|x-z^{k+1}\|_{[s]}^{2}.
\]
Note that $\Gamma_{k}=\prod_{i=0}^{k}\tfrac{1}{1-\alpha_{i}}$. Then,
multiplying both sides of the above relation by $\Gamma_{k}$, and
then summing up over $k=0,1,2,...,K-1$, we have 
\begin{align*}
\Gamma_{K-1}\Ebb[f(x^{K})-f(x)] &\le f(x^{0})-f(x)+\tfrac{\Gamma_{0}\alpha_{0}\beta_{0}}{2}\|x-x^{0}\|_{[s]}^{2} \\
& \quad -\tfrac{\Gamma_{K-1}\alpha_{K-1}(\beta_{K-1}+\gamma)}{2}\Ebb\|x-z^{K}\|_{[s]}^{2}.
\end{align*}
Moreover, if we choose $\alpha_{k}=\tfrac{\sqrt{\mu_{s}}}{\sqrt{\mu_{s}}+T_{(1-s)/2}}$, $\beta_{k}=\sqrt{\mu_{s}}T_{(1-s)/2}$ and $\gamma=\mu_{s}$, then
we have 
\[
\Ebb[f(x^{K})-f(x^{*})]\le\left(1-\tfrac{\sqrt{\mu_{s}}}{\sqrt{\mu_{s}}+T_{(1-s)/2}}\right)^{K}\left[f(x^{0})-f(x^{*})+\tfrac{\mu_{s}}{2}\|x-x^{0}\|_{[s]}^{2}\right].
\]
\end{proof}
The best rate of Algorithm \ref{alg:ACD-nuni} is achieved at $s=0$.
We summarize the complexity of such a case in the following corollary.
\begin{corollary}
Let $x^{*}$ be the optimal solution. Assume that $f(\cdot)$ is strongly
convex with norm $\|\cdot\|^{2}$. If we choose $s=0$, then for any
$\epsilon>0$, 
\[
N_{\epsilon}\in\Ocal\left(\tfrac{\sqrt{\mu_{0}}+\sum_{i=1}^{m}\sqrt{L_{i}}}{\sqrt{\mu_0}}\log\tfrac{f(x^{0})-f(x^{*})+\tfrac{\mu_{0}}{2}\|x-x^{0}\|^{2}}{\epsilon}\right)
\]
 iterations of Algorithm \ref{alg:ACD-nuni} are required to obtain
an expected $\epsilon$-accurate solution.
\end{corollary}
\begin{proof}
From Theorem \ref{thm:acpp}, we have that 
\begin{align*}
\Ebb[f(x^{K})-f(x^{*})] & \le\left(1-\tfrac{\sqrt{\mu_{0}}}{\sqrt{\mu_{0}}+T_{1/2}}\right)^{K}\left[f(x^{0})-f(x^{*})+\tfrac{\mu_{0}}{2}\|x-x^{0}\|^{2}\right]\\
 & \le\exp\left(-\tfrac{\sqrt{\mu_{0}}}{\sqrt{\mu_{0}}+T_{1/2}}K\right)\left[f(x^{0})-f(x^{*})+\tfrac{\mu_{0}}{2}\|x-x^{0}\|^{2}\right].
\end{align*}
\end{proof}

\bibliographystyle{jabbrv_siam}
\bibliography{bcd}

\end{document}